\documentclass[12pt]{article}                
\usepackage{graphicx, xcolor}
\usepackage{psfrag}
\usepackage{amsmath, amssymb, amsthm}
\usepackage{mathptmx}
\usepackage{esvect, bm, bbm}
\usepackage{hyperref}
\usepackage{cleveref}
\usepackage{enumitem}

\hypersetup{
	colorlinks=true,
	linkcolor={blue!80!black},
	citecolor={green!50!black},
	urlcolor={red!50!black}
}

\newcommand{\N}{\mathbb{N}}
\newcommand{\Z}{\mathbb{Z}}
\newcommand{\Q}{\mathbb{Q}}
\newcommand{\R}{\mathbb{R}}
\newcommand{\Rp}{\mathbb{R}_{\geq 0}}
\newcommand{\Rpp}{\mathbb{R}_{>0}}
\newcommand{\1}{\mathbbm{1}}

\newcommand{\D}{\mathcal{D}}
\newcommand{\C}{\mathsf{C}}
\newcommand{\CC}{\mathcal{C}}
\newcommand{\EE}{\mathcal{E}}
\newcommand{\F}{\mathsf{F}}
\newcommand{\G}{\mathcal{G}}
\newcommand{\tG}{\widetilde{G}}
\renewcommand{\H}{\mathsf{H}}
\newcommand{\HH}{\mathcal{H}}

\newcommand{\I}{\mathcal{I}}
\newcommand{\K}{\mathcal{K}}
\renewcommand{\L}{\mathcal{L}}
\newcommand{\m}{\overline{m}}
\newcommand{\M}{\mathcal{M}}

\newcommand{\p}{\varphi}

\renewcommand{\P}{\mathbb{P}}

\newcommand{\TV}{\mathrm{TV}}
\newcommand{\td}{\widetilde}
\renewcommand{\bar}{\overline}

\newcommand{\abs}[1]{\lvert#1\rvert}
\newcommand{\norm}[1]{\lVert#1\rVert}
\newcommand{\Norm}[1]{{\vert\kern-0.1ex\vert\kern-0.1ex\vert#1\vert\kern-0.1ex\vert\kern-0.1ex\vert}}
\newcommand{\dNorm}[2]{{\vert\kern-0.1ex\vert\kern-0.1ex\vert#2\vert\kern-0.1ex\vert\kern-0.1ex\vert}_{#1,*}}

\newcommand{\ud}{\, \mathrm{d}}

\DeclareMathOperator{\E}{\mathbb{E}}
\DeclareMathOperator{\Var}{\operatorname{Var}}

\DeclareMathOperator{\Par}{\mathcal{P}}

\DeclareMathOperator{\BigO}{\operatorname{\mathcal{O}}}

\DeclareMathOperator{\Ber}{\operatorname{Ber}}
\DeclareMathOperator{\Poi}{\operatorname{Poi}}
\DeclareMathOperator{\Exp}{\operatorname{Exp}}
\DeclareMathOperator{\Bin}{\operatorname{Bin}}

\newcommand{\step}[1]{{\underline{\textit {Step #1}}}}

\newtheorem{theorem}{Theorem}[section]
\newtheorem{lemma}[theorem]{Lemma}
\newtheorem{proposition}[theorem]{Proposition}
\newtheorem{corollary}[theorem]{Corollary}
\newtheorem{definition}[theorem]{Definition}
\newtheorem{conjecture}{Conjecture}
\newtheorem*{acknowledgements}{Acknowledgements}

\theoremstyle{remark}
\newtheorem{remark}[theorem]{Remark}

\numberwithin{equation}{section}

\begin{document}

\title{\vspace{-1cm}Mutual information for the sparse \\
stochastic block model
}
\author{Tomas Dominguez\thanks{\textsc{\tiny Department of Mathematics, University of Toronto, tomas.dominguezchiozza@mail.utoronto.ca}} \and  Jean-Christophe Mourrat\thanks{\textsc{\tiny Department of Mathematics, ENS Lyon and CNRS, jean-christophe.mourrat@ens-lyon.fr}}}

\date{}
\maketitle
\vspace*{-0.7cm}
\begin{abstract}
We consider the problem of recovering the community structure in the stochastic block model with two communities. We aim to describe the mutual information between the observed network and the actual community structure in the sparse regime, where the total number of nodes diverges while the average degree of a given node remains bounded. Our main contributions are a conjecture for the limit of this quantity, which we express in terms of a Hamilton-Jacobi equation posed over a space of probability measures, and a proof that this conjectured limit provides a lower bound for the asymptotic mutual information. The well-posedness of the Hamilton-Jacobi equation is obtained in our companion paper. In the case when links across communities are more likely than links within communities, the asymptotic mutual information is known to be given by a variational formula. We also show that our conjectured limit coincides with this formula in this case.
\end{abstract}

\section{Introduction and main results}

The stochastic block model is the simplest generative model for networks with a community structure. It was first introduced in the machine learning and statistics literature \cite{Finberg, Holland, Yuchung, Harrison} but soon emerged independently in a variety of other scientific disciplines. In the theoretical computer science community it is often termed the planted partition model \cite{Boppana, Bui, Dyer} while the mathematics literature often refers to it as the inhomogeneous random graph model \cite{Bollobas}. Since its introduction, the stochastic block model has become a test bed for clustering and community detection algorithms used in social networks \cite{Newman_social}, protein-to-protein interaction networks \cite{Chen_protein}, recommendation systems \cite{Linden}, medical prognosis \cite{Sorlie}, DNA folding \cite{Cabreros}, image segmentation \cite{Jianbo} and natural language processing \cite{Ball} among others. In this paper we focus on the sparse stochastic block model with two communities which we now describe.

Consider $N$ individuals belonging to exactly one of two communities; it will be convenient to describe the communities using vectors of $\pm 1$ with the agreement that people with the same label belong to the same group. In this way, a vector
\begin{equation}
\sigma^*=\big(\sigma^*_1,\ldots,\sigma^*_N\big)\in \Sigma_N=\{-1,+1\}^N
\end{equation}
can be used to encode the two communities. The labels $\sigma^*_i\sim P^*$ are taken to be i.i.d.\@ Bernoulli random variables with probability of success $p\in (0,1)$ and expectation $\m$,
\begin{equation}
p=P^*(1)=\P\{\sigma_i^*=1\} \quad \text{and} \quad \m=\E \sigma_1^*=2p-1.
\end{equation}
The case $p=1/2$ is termed the symmetric stochastic block model, and for reasons that will become apparent below has received the greatest attention. The assignment vector $\sigma^*$ follows a product distribution,
\begin{equation}
\sigma^*\sim P_N^*=(P^*)^{\otimes N},
\end{equation}
so the expected sizes of the communities are $Np$ and $N(1-p)$. Using the assignment vector $\sigma^*$, a random undirected graph $\mathbf{G}_N = (G_{ij})_{i,j \le N}$ with vertex set $\{1,\ldots,N\}$ is constructed by stipulating that an edge between node $i$ and node $j$ is present with conditional probability
\begin{equation}\label{eqn: SBM edge probabilities aNbN}
\P\big\{G_{ij}=1|\sigma^*\big\}=\begin{cases}
a_N & \text{if } \sigma^*_i=\sigma^*_j\\
b_N & \text{if } \sigma^*_i\neq \sigma^*_j
\end{cases}
\end{equation}
for some $a_N,b_N\in (0,1)$ independently of all other edges. In other words, the probability that an edge is present between node $i$ and node $j$ depends only on whether or not the individuals $i$ and $j$ belong to the same community. To express \eqref{eqn: SBM edge probabilities aNbN} more succinctly, it is convenient to introduce the average and the gap of $a_N$ and $b_N$,
\begin{equation}
c_N=\frac{a_N+b_N}{2} \quad \text{and} \quad \Delta_N=\frac{a_N-b_N}{2}\in (-c_N,c_N),
\end{equation}
in such a way that
\begin{equation}\label{eqn: SBM edge probabilities cNdN}
\P\{G_{ij}=1|\sigma^*\}=c_N+\sigma^*_i\sigma^*_j\Delta_N.
\end{equation}
The data $\mathbf{G}_N=(G_{ij})$ is said to be sampled from the stochastic block model, and the inference task is to recover the assignment vector $\sigma^*$ as accurately as possible given the graph $\mathbf{G}_N$. In the case when $\Delta_N\leq 0$, it is more likely for an edge to be present between nodes in different communities and the model is called \emph{disassortative}. When $\Delta_N>0$ connections are more likely between individuals in the same community and the model is termed \emph{assortative}. 

Recently, the stochastic block model has attracted much renewed attention. On a practical level, we mention for instance extensions allowing for overlapping communities \cite{Airoldi} that have proved to be a good fit for real data sets in massive networks \cite{Gopalan}. On a theoretical level, the predictions put forth in \cite{DKMZ} using deep but non-rigorous statistical physics arguments have been particularly stimulating. The theoretical study of the stochastic block model has seen significant progress in two main directions: exact recovery and detection. The exact recovery task aims to determine the regimes of $a_N$ and $b_N$ for which there exists an algorithm that completely recovers the two communities with high probability. Of course, a necessary condition for exact recovery is connectivity of the random graph $\mathbf{G}_N$; this makes exact recovery impossible in the sparse regime. The sharp threshold for exact recovery was obtained in \cite{Abbe_recovery, Mossel_recovery}, where it was shown that in the symmetric dense regime, $p=1/2$, $a_N=a\log(N)/N$ and $b_N=b\log(N)/N$,  exact recovery is possible, and efficiently so, if and only if $\sqrt{a} - \sqrt{b} \ge 2 $. On the other hand, the detection task is to construct a partition of the graph $\mathbf{G}_N$ that is positively correlated with the assignment vector $\sigma^*$ with high probability. The sharp threshold for detection in the sparse regime was obtained in \cite{Massoulie, Mossel_reconstruction, Mossel_detection}, where it was shown that in the symmetric sparse regime, $p=1/2$, $a_N=a/N$ and $b_N=b/N$, detection is solvable, and efficiently so, if and only if $(a-b)^2>2(a+b)$. Notice that detection is much easier in the asymmetric case \cite{Caltagirone}. Indeed, the expected degree of node $i$ conditional on its community membership is given by
\begin{equation}
\E[\deg(i)|\sigma^*_i]=(N-1)(c_N+\m\Delta_N\sigma^*_i),
\end{equation}
so meaningful information about the community structure is revealed from the degrees of nodes.

Despite this clear picture regarding the thresholds for exact recovery and detection in the setting of two communities, several questions remain open. In this paper, we focus on the problem of quantifying exactly how much information about the communities can be recovered by observing the graph $\mathbf{G}_N$. The mutual information between the assignment vector $\sigma^*$ and the random graph~$\mathbf{G}_N$ is defined by
\begin{equation}\label{eqn: SBM mutual information}
I(\mathbf{G}_N; \sigma^*)=\E\log\frac{\P(\mathbf{G}_N,\sigma^*)}{\P(\mathbf{G}_N)\P(\sigma^*)}=\E\log \frac{\P(\mathbf{G}_N|\sigma^*)}{\P(\mathbf{G}_N)}.
\end{equation}
The asymptotic value of this mutual information has been computed in the dense regime \cite{Barbier_dense, Deshpande, Lelarge} and the sparse disassortative regime \cite{Abbe_disassortative, CKPZ}. Its determination in the assortative sparse regime has proved more challenging. After we posted a first version of this paper to the arXiv, this problem was resolved in \cite{yu2022ising} in the case $p = 1/2$, building upon the earlier works \cite{Abbe_assortative, kanade2016global, mossel2016belief, mossel2016local}. The approach developed there does not generalize well to more complex models such as when more than two communities are present \cite{gu2023uniqueness}. In contrast, our aim here is to propose a new approach to the analysis of the community detection problem that would be robust to model modifications. 
We will discuss in much greater detail the relationship between these very recent works and our contribution near the end of the introduction; see also \cite{Abbe_book, Abbe_more_communities, Abbe_learning} for more on problems with more than two communities.

Henceforth, we focus exclusively on the sparse stochastic block model with
\begin{equation}
c_N=\frac{c}{N} \quad \text{and} \quad \Delta_N=\frac{\Delta}{N}
\end{equation}
for some $c>0$ and some non-zero $\Delta\in (-c,c)$. The case $\Delta=0$ is trivial since it corresponds to the case where the graph $\mathbf{G}_N$ and the assignment vector $\sigma^*$ are independent. 
The probability \eqref{eqn: SBM edge probabilities cNdN} that an edge is present between node $i$ and node $j$ becomes
\begin{equation}\label{eqn: SBM edge probabilities}
\P\{G_{ij}=1|\sigma^*\}=\frac{c+\Delta \sigma_i^*\sigma_j^*}{N}
\end{equation}
for the family of conditionally independent Bernoulli random variables $\mathbf{G}_N=(G_{ij})$. The expected degree of any node $i$ remains bounded with $N$,
\begin{equation}
\E \deg(i)=\frac{N-1}{N}\big(c+\Delta \m^2\big),
\end{equation}
so we are indeed in the sparse regime. The likelihood of the model is given by
\begin{equation}\label{eqn: SBM likelihood}
\P\big\{\mathbf{G}_N=(G_{ij})|\sigma^*=\sigma\big\}=\prod_{i<j}\Big(\frac{c+\Delta \sigma_i\sigma_j}{N}\Big)^{G_{ij}}\Big(1-\frac{c+\Delta \sigma_i\sigma_j}{N}\Big)^{1-G_{ij}}=\frac{\exp\big(H_N^\circ(\sigma)\big)}{N^{\sum_{i<j}G_{ij}}},
\end{equation}
while Bayes' formula implies that the posterior of the model is the Gibbs measure
\begin{equation}\label{eqn: SBM posterior}
\P\big\{\sigma^*=\sigma|\mathbf{G}_N=(G_{ij})\big\}=\frac{\exp\big(H_N^\circ(\sigma)\big)P_N^*(\sigma)}{\int_{\Sigma_N}\exp\big(H_N^\circ(\tau)\big)\ud P_N^*(\tau)}
\end{equation}
associated with the Hamiltonian
\begin{equation}\label{eqn: SBM Hamiltonian 0}
H_N^\circ(\sigma)=\sum_{i<j}\log\bigg[\big(c+\Delta \sigma_i\sigma_j\big)^{G_{ij}}\Big(1-\frac{c+\Delta \sigma_i\sigma_j}{N}\Big)^{1-G_{ij}}\bigg].
\end{equation}
Moreover, up to an error vanishing with $N$ and a simple additive constant, the normalized mutual information~\eqref{eqn: SBM mutual information} coincides with the free energy
\begin{equation}\label{eqn: SBM free energy 0}
\overline{F}_N^\circ=\frac{1}{N}\E\log \int_{\Sigma_N}\exp H_N^\circ(\sigma)\ud P_N^*(\sigma).
\end{equation}
Indeed, \eqref{eqn: SBM likelihood} and Bayes' formula imply that
\begin{equation}
I(\mathbf{G}_N;\sigma^*)=\binom{N}{2}\E\log(c+\Delta \sigma_1^*\sigma_2^*)^{G_{12}}\Big(1-\frac{c+\Delta \sigma_1^*\sigma_2^*}{N}\Big)^{1-G_{12}}-N\overline{F}_N^\circ.
\end{equation}
Averaging with respect to the randomness of $G_{12}$ and Taylor expanding the logarithm reveals that
\begin{equation}\label{eqn: SBM MI and free energy 0}
\frac{1}{N}I(\mathbf{G}_N;\sigma^*)=\frac{1}{2}\E\big(c+\Delta \sigma_1^*\sigma_2^*\big)\log\big(c+\Delta \sigma_1^*\sigma_2^*\big)-\frac{c}{2}-\frac{\Delta \m^2}{2}-\overline{F}_N^\circ+\BigO\big(N^{-1}\big).
\end{equation}
To study the mutual information \eqref{eqn: SBM mutual information}, we take the perspective of statistical physics and instead focus on the free energy \eqref{eqn: SBM free energy 0}. Notice that the mutual information between two independent random variables is zero, so for $\Delta=0$, the free energy is found by setting the right-hand side of \eqref{eqn: SBM MI and free energy 0} equal to zero.

For technical reasons, it will be convenient to modify the free energy \eqref{eqn: SBM free energy 0} without changing its limiting value. Conditionally on $\sigma^*$, the modified Hamiltonian will be a sum of independent random variables, and the sum will be over a Poisson-distributed number of terms. The main advantage of this construction is that we can then conveniently vary the continuous parameter encoding the Poisson random variable, and in particular study derivatives with respect to this parameter. To be more precise, we introduce a random variable $\smash{\Pi_{1}\sim\Poi\binom{N}{2}}$ as well as an independent family of i.i.d.\@ random matrices $\smash{(G^k)_{k \in \N}}$ each having conditionally independent entries $\smash{(G_{i,j}^k)_{i,j\leq N}}$ taking values in $\{0,1\}$ with conditional distribution
\begin{equation}\label{eqn: SBM G distribution}
\P\big\{G_{i,j}^k=1|\sigma^*\big\}=\frac{c+\Delta \sigma_{i}^*\sigma_{j}^*}{N}.
\end{equation}
Given a collection of random indices $\smash{(i_k,j_k)_{k\in \N}}$ sampled uniformly at random from $\smash{\{1,\ldots,N\}^2}$, independently of the other random variables, we define the Hamiltonian $H_N$ on $\Sigma_N$ by
\begin{equation}\label{eqn: SBM Hamiltonian}
H_N(\sigma)=\sum_{k\leq \Pi_{1}}\log \bigg[\big(c+\Delta \sigma_{i_k}\sigma_{j_k}\big)^{G^k_{i_k,j_k}}\Big(1-\frac{c+\Delta \sigma_{i_k}\sigma_{j_k}}{N}\Big)^{1-G^k_{i_k,j_k}}\bigg],
\end{equation}
and write
\begin{equation}\label{eqn: SBM free energy}
\overline{F}_N=\frac{1}{N}\E\log \int_{\Sigma_N} \exp H_N(\sigma)\ud P_N^*(\sigma)
\end{equation}
for its associated free energy. We show in \Cref{SBM app equivalence} that the difference between the free energies in \eqref{eqn: SBM free energy 0} and \eqref{eqn: SBM free energy} tends to $0$ as $N$ tends to infinity. Together with \eqref{eqn: SBM MI and free energy 0}, this implies that
\begin{equation}\label{eqn: SBM MI and FE}
\frac{1}{N}I(\mathbf{G}_N;\sigma^*)=\frac{1}{2}\E\big(c+\Delta \sigma_1^*\sigma_2^*\big)\log\big(c+\Delta \sigma_1^*\sigma_2^*\big)-\frac{c}{2}-\frac{\Delta \m^2}{2}-\overline{F}_N+o_N(1).
\end{equation}
The problem of finding the asymptotic value of the mutual information \eqref{eqn: SBM mutual information} has therefore been reduced to the task of determining the limit of the free energy \eqref{eqn: SBM free energy}. The main contributions of this work are the conjecture that
\begin{equation}
\lim_{N\to \infty}\overline{F}_N=f(1,0),
\end{equation}
where $f(t,\mu)$ is the solution to an infinite-dimensional Hamilton-Jacobi equation defined in \eqref{eqn: SBM enriched free energy HJ equation}, and a proof that $f(1,0)$ provides an upper bound for the limit of the free energy; the matching lower bound will remain open.

To motivate and define the infinite-dimensional Hamilton-Jacobi equation \eqref{eqn: SBM enriched free energy HJ equation}, we will introduce an ``enriched'' free energy functional
by transforming the free energy \eqref{eqn: SBM free energy} into a function of a ``time'' variable $t \ge 0$ and a non-negative measure $\mu$. The ``time'' variable will be used to vary the intensity of the Poisson point process $\Pi_1$ appearing in the Hamiltonian \eqref{eqn: SBM Hamiltonian}. We introduced the Hamiltonian $H_N$ to replace $H_N^\circ$ in order to allow for convenient integration-by-parts-like calculations as we study derivatives with respect to this parameter $t$. The non-negative measure $\mu$ will be decomposed as $\mu = s \bar{\mu}$ for $s \geq 0$ and a probability measure $\bar{\mu}$. It will be used to consider a situation in which we also observe the graph of connections of a simpler setting in which each individual $i$ can form connections with its own set of neighbour candidates. To be more specific, each individual $i$ will have an independent number $\Poi(sN)$ of neighbour candidates indexed by the pairs $(i,k)$ for $k \leq \Poi(sN)$. Each candidate neighbour $(i,k)$ will be independently assigned a random ``type'' $x_{i,k}$ sampled from the distribution $\bar{\mu}$, and an edge will be present between individual $i$ and its candidate neighbour $(i,k)$ with probability $N^{-1} (c+ \Delta \sigma_i^* x_{i,k})$. In the inference problem, the ``types'' $x_{i,k}$ are revealed to the statistician. The lack of interactions between individuals makes this piece of information much simpler to understand than the original community detection problem we aim to make progress upon. In total, this allows us to define an ``enriched'' free energy, function of $t$ and $\mu$, and the quantity in \eqref{eqn: SBM free energy} can then be recovered by evaluating this enriched free energy at $t = 1$ and $\mu = 0$. We next aim to study whether variations in the $t$ variable can be suitably compensated by variations in the $\mu$ variable, leaving the free energy roughly constant. More precisely, we hope to discover that the derivative in~$t$ of this functional can be expressed, up to a small error, as a function of its derivative in $\mu$. On a heuristic level, one can see that this indeed seems to be possible, as will be clarified below. Combining ideas from the theory of viscosity solutions with the multi-overlap concentration result in \cite{BarMOC}, we will be able to prove one inequality between the limit free energy and the solution to the partial differential equation that arises. Although we expect the converse bound to also be valid, significant technical challenges prevent us from proving it at the moment. The difficulty is that the control we have on the ``small error'' appearing in the equation at finite $N$ is relatively weak. In particular, we cannot show, and do not expect, that it becomes small as $N$ tends to infinity for each individual choice of~$t$ and $\mu$. On the other hand, controlling the error after we perform a small averaging over $t$ and $\mu$ is possible, but does not suffice for the identification of the limit.

Let us now define the enriched free energy precisely. Denote by $\Pr[-1,1]$ the set of probability measures on $[-1,1]$, and given $\mu \in \Pr[-1,1]$, consider a sequence $x=(x_{i,k})$ of i.i.d.\@ random variables with law $\mu$. For each $s>0$ and $i\geq 1$, let $\Pi_{i,s}\sim \Poi(sN)$ be independent over $i\geq 1$, and introduce the Hamiltonian on $\Sigma_N$,
\begin{equation}\label{eqn: SBM Hamiltonian s}
\widetilde{H}_N^{s,\mu}(\sigma)=\sum_{i\leq N}\sum_{k\leq \Pi_{i,s}}\log\bigg[\big(c+\Delta \sigma_ix_{i,k}\big)^{\tG_{i,k}^x}\Big(1-\frac{c+\Delta \sigma_ix_{i,k}}{N}\Big)^{1-\tG_{i,k}^x}\bigg],
\end{equation}
where the random variables $(\tG_{i,k}^x)$ are independent with conditional distribution
\begin{equation}\label{eqn: SBM tG distribution}
\P\big\{\tG_{i,k}^x=1|\sigma^*,x\big\}=\frac{c+\Delta \sigma_i^*x_{i,k}}{N}.
\end{equation}
As alluded to above, this is the Hamiltonian associated with the task of inferring the signal $\sigma^*$ from the data
\begin{equation}
\widetilde{\D}_N^{s,\mu}=\big(\Pi_{i,s}, (x_{i,k})_{k\leq \Pi_{i,s}}, (\tG_{i,k})_{k\leq \Pi_{i,s}}\big)_{i\leq N},
\end{equation}
in the sense that the Gibbs measure associated with  $\smash{\widetilde{H}_N^{s,\mu}}$ is the conditional law of $\sigma^*$ given the data~$\smash{\widetilde{\D}_N^{s,\mu}}$ (as in the identity \eqref{eqn: SBM posterior} for the Hamiltonian $H_N^\circ$ and the data $\mathbf G_N$).
For each $t\geq 0$, let $\smash{\Pi_t\sim \Poi t\binom{N}{2}}$, and consider a time-dependent version of the Hamiltonian \eqref{eqn: SBM Hamiltonian} defined on $\Sigma_N$ by
\begin{equation}\label{eqn: SBM Hamiltonian t}
H_N^t(\sigma)=\sum_{k\leq \Pi_t}\log \bigg[\big(c+\Delta \sigma_{i_k}\sigma_{j_k}\big)^{G^k_{i_k,j_k}}\Big(1-\frac{c+\Delta \sigma_{i_k}\sigma_{j_k}}{N}\Big)^{1-G^k_{i_k,j_k}}\bigg].
\end{equation}
Notice that this is the Hamiltonian associated with the task of inferring the signal $\sigma^*$ from the data 
\begin{equation}
\D_N^t=\big(\Pi_t, (i_k,j_k)_{k\leq \Pi_t},(G^k_{i_k,j_k})_{k\leq \Pi_t}\big).
\end{equation}
We now introduce an enriched Hamiltonian on $\Sigma_N$,
\begin{equation}\label{eqn: SBM enriched Hamiltonian s}
\widetilde{H}_N^{t,s,\mu}(\sigma)=H_N^t(\sigma)+\widetilde{H}_N^{s,\mu}(\sigma),
\end{equation}
and denote by
\begin{equation}\label{eqn: SBM enriched free energy s}
\widetilde{F}_N(t,s,\mu)=\frac{1}{N}\E\log \int_{\Sigma_N} \exp \widetilde{H}_N^{t,s,\mu}(\sigma)\ud P_N^*(\sigma)
\end{equation}
its associated free energy. Observe that $\widetilde{F}_N(1,0,\mu)=\overline{F}_N$ and that \eqref{eqn: SBM enriched Hamiltonian s} is the Hamiltonian associated with inferring the signal $\sigma^*$ from the data
\begin{equation}\label{eqn: SBM enriched data}
\widetilde{\D}_N^{t,s,\mu}=(\D_N^t, \widetilde{\D}_N^{s,\mu}),
\end{equation}
where the randomness in these two data sets is taken to be independent conditionally on $\sigma^*$. To obtain a Hamilton-Jacobi equation, it will be convenient to reinterpret the enriched free energy~\eqref{eqn: SBM enriched free energy s} as a function of the time-parameter $t>0$ and a finite measure $\mu$; the parameter $s$ will become the total mass of this finite measure. 
We denote by $\M_s$ the space of signed measures on~$[-1,1]$,
\begin{equation}
\M_s=\big\{\mu \mid \mu \text{ is a signed measure on } [-1,1]\big\},
\end{equation}
and by $\M_+$ the cone of non-negative measures on this interval,
\begin{equation}\label{eqn: SBME M+}
\M_+=\big\{\mu \in \M_s\mid \mu \text{ is a non-negative measure}\big\}.
\end{equation}
We follow the convention that a signed measure can only take finite values, and in particular, every $\mu \in \M_+$ must have finite total mass. This implies that every non-zero measure $\mu \in \M_+$ induces a probability measure,
\begin{equation}
\bar{\mu}=\frac{\mu}{\mu[-1,1]}\in \Pr[-1,1].
\end{equation}
Given a measure $\mu \in \M_+$, we define the Hamiltonian $\smash{H_N^{t,\mu}}$ on $\Sigma_N$ by
\begin{equation}\label{eqn: SBM enriched Hamiltonian}
H_N^{t,\mu}(\sigma)=\widetilde{H}_N^{t,\mu[-1,1],\bar{\mu}}(\sigma),
\end{equation}
where $\widetilde{H}_N^{0,0}=0$ for the zero measure by continuity. The free energy associated with this Hamiltonian is given by
\begin{equation}\label{eqn: SBM enriched free energy}
\overline{F}_N(t,\mu)=\widetilde{F}_N\big(t,\mu[-1,1],\bar{\mu}\big),
\end{equation}
and once again $\overline{F}_N = \overline{F}_N(1,0)$, where $0$ denotes the zero measure. The free energy in \eqref{eqn: SBM enriched free energy} will be termed the enriched free energy, and in \Cref{sec: SBM HJ eqn derivation} we will show that, up to a ``small error'', it satisfies an infinite-dimensional Hamilton-Jacobi equation which we now describe.

Introduce the function $g:[-1,1]\to \R$ defined by
\begin{equation}\label{eqn: SBM g function}
g(z)=(c+\Delta z)\big(\log(c+\Delta z)-1\big)=(c+\Delta z)\log(c)+c\sum_{n\geq 2}\frac{(-\Delta/c)^n}{n(n-1)}z^n-c
\end{equation}
as well as the cone of functions
\begin{equation}\label{eqn: SBM cone of functions}
\CC_\infty=\bigg\{G_\mu : [-1,1]\to\R \mid G_\mu(x)=\int_{-1}^1 g(xy)\ud \mu(y) \text{ for some } \mu\in \M_+\bigg\}
\end{equation}
and the non-linearity $\C_\infty: \CC_\infty\to \R$ given by
\begin{equation}\label{eqn: SBM infinite non-linearity}
\C_\infty(G_\mu)=\frac{1}{2}\int_{-1}^1 G_\mu(x)\ud \mu(x).
\end{equation}
This non-linearity is well-defined by the Fubini-Tonelli theorem (see equations (1.6)-(1.7) in \cite{TD_JC_HJ}). Given a function $f:[0,\infty)\times \M_+\to \R$ and measures $\mu,\nu\in \M_+$, we denote by $D_\mu f(t,\mu;\nu)$ the Gateaux derivative of the function $f(t,\cdot)$ at the measure $\mu$ in the direction $\nu$,
\begin{equation}\label{eqn: SBME Gateaux derivative}
D_\mu f(t,\mu;\nu)=\lim_{\epsilon \to 0}\frac{f(t,\mu+\epsilon \nu)-f(t,\mu)}{\epsilon}.
\end{equation}
We will say that the Gateaux derivative of $f$ admits a density at the measure $\mu \in \M_+$ if there exists a bounded measurable function $x\mapsto D_\mu f(t,\mu,x)$ defined on the interval $[-1,1]$ with
\begin{equation}\label{eqn: SBME Gateaux derivative density}
D_\mu f(t,\mu;\nu)=\int_{-1}^1 D_\mu f(t,\mu,x)\ud \nu(x)
\end{equation}
for every measure $\nu \in \M_+$. We will often abuse notation and identify the density $D_\mu f(t,\mu,\cdot)$ with the Gateaux derivative $D_\mu f(t,\mu)$. In \Cref{sec: SBM HJ eqn derivation} we will show that, up to an error vanishing with $N$, the Gateaux derivative of the enriched free energy is indeed of the form $G_{\mu^*}$ for some $\mu^* \in \M_+$. In fact, the measure $\mu^*$ is the law of the Gibbs average of a uniformly sampled spin coordinate for the Gibbs measure associated with the Hamiltonian \eqref{eqn: SBM enriched Hamiltonian s}. We will also argue that, if this Gibbs measure satisfies suitable overlap concentration properties, then the time derivative of the free energy is essentially given by $\int_{-1}^1 g(xy) \ud \mu^*(x) \ud \mu^*(y)$. In short, assuming the validity of these overlap concentration properties, we are led to believe that the large-$N$ limit of the free energy \eqref{eqn: SBM enriched free energy} should satisfy the infinite-dimensional Hamilton-Jacobi equation 
\begin{equation}\label{eqn: SBM enriched free energy HJ equation}
\left\{
\begin{aligned}
\partial_t f(t,\mu)&=\C_\infty\big(D_\mu f(t,\mu)\big) & \text{on }& \Rpp\times \M_+,\\
f(0,\mu)&=\psi(\mu)& \text{on }& \M_+,
\end{aligned}
\right.
\end{equation}
where the initial condition $\smash{\psi:\M_+\to\R}$ is the limit of $\smash{F_N(0,\cdot)}$ and can be readily computed, see Lemma~\ref{SBM convergence of initial condition discrete}. 
The well-posedness of this equation is  established in \cite{TD_JC_HJ}, and leads to the conjecture that the enriched free energy converges to the solution to this equation. Remembering \eqref{eqn: SBM MI and FE}, this translates into a conjecture for the asymptotic mutual information. 

\begin{conjecture}\label{SBM main conjecture}
If $f$ denotes the unique viscosity solution to the infinite-dimensional Hamilton-Jacobi equation \eqref{eqn: SBM enriched free energy HJ equation}, then the limit of the free energy \eqref{eqn: SBM free energy} is given by
\begin{equation}
\lim_{N\to \infty}\overline{F}_N=f(1,0).
\end{equation}
In particular, the asymptotic value of the mutual information \eqref{eqn: SBM mutual information} is
\begin{align}
\lim_{N\to\infty}\frac{1}{N}I(\mathbf{G}_N;\sigma^*)=\frac{1}{2}\E\big(c+\Delta \sigma_1^*\sigma_2^*\big)\log\big(c+\Delta \sigma_1^*\sigma_2^*\big)-\frac{c}{2}-\frac{\Delta\m^2}{2}-f(1,0).
\end{align}
\end{conjecture}

The main result of this paper is a proof of the upper bound in \Cref{SBM main conjecture}.

\begin{theorem}\label{SBM main result}
Denote by $f$ the unique viscosity solution to the infinite-dimensional Hamilton-Jacobi equation \eqref{eqn: SBM enriched free energy HJ equation}. For every $t\geq 0$ and  $\mu\in \M_+$, the limit of the enriched free energy \eqref{eqn: SBM enriched free energy} satisfies the upper bound $\smash{\limsup_{N\to \infty}\overline{F}_N(t,\mu)\leq f(t,\mu)}$.
In particular, the free energy \eqref{eqn: SBM free energy} satisfies the upper bound
\begin{equation}
\label{e.SBM.main}
\limsup_{N\to \infty}\overline{F}_N\leq f(1,0).
\end{equation}
\end{theorem}

Although the matching lower bound still remains open, we give some support in favor of \Cref{SBM main conjecture} by proving that, in the disassortative regime, it matches the variational formula for the asymptotic free energy obtained in \cite{CKPZ}. (The condition $p = 1/2$ was also assumed in \cite{CKPZ}.) We state this formula using the notation in \cite{PanNote}, where a more direct proof is obtained using an interpolation argument and a cavity computation. Denote by
\begin{equation}
\M_p=\bigg\{\mu\in \Pr[-1,1]\mid \int_{-1}^1 x\ud \mu=\m\bigg\}
\end{equation}
the set of probability measures with mean $\m=2p-1$, and introduce the functional $\Par:\M_p\to \R$ defined by
\begin{equation}\label{eqn: SBM Parisi functional}
\Par(\mu)=\psi(\mu)+\frac{c}{2}+\frac{\Delta\m^2}{2}-\frac{1}{2}\E(c+\Delta x_1x_2)\log(c+\Delta x_1x_2),
\end{equation}
where $x_1$ and $x_2$ are independent samples from the probability measure $\mu$.
\begin{theorem}\label{SBM disassortative Hopf-Lax}
In the disassortative sparse stochastic block model with $\Delta\leq 0$, the limit of the free energy \eqref{eqn: SBM free energy} is given by
\begin{equation}\label{eqn: SBM disassortative Hopf-Lax}
\lim_{N\to \infty}\overline{F}_N=\sup_{\mu \in \M_p}\Par(\mu) = f(1,0).
\end{equation}
\end{theorem}
Our results generalize immediately to the case in which the measure $P^*$ is arbitrary with compact support, with the understanding that the link probabilities are given by \eqref{eqn: SBM edge probabilities}. We also believe that they generalize without much change to settings with more than two communities, although we have not checked every technical detail. 

Before closing this introduction, we discuss alternatives to the conjecture and the approach proposed in this paper. It will facilitate this discussion to point out that the proof of Theorem~\ref{SBM disassortative Hopf-Lax} also yields that when $\Delta \le 0$, we can identify the limit of $\bar F_N(t,\mu)$ for every $t \ge 0$ and $\mu \in \M_+$ as
\begin{equation}
\label{e.lim.barfn.t.mu}
\lim_{N\to \infty}\overline{F}_N(t,\mu) =  f(t,\mu) = \sup_{\nu \in \Pr[-1,1]}\bigg(\psi(\mu + t\nu) - \frac{t}{2} \int_{-1}^1 G_{\nu}(y) \ud \nu(y)\bigg),
\end{equation}
where we recall that $G_\nu$ is defined in \eqref{eqn: SBM g function}-\eqref{eqn: SBM cone of functions}. The identity \eqref{e.lim.barfn.t.mu} also remains valid if we take the supremum over all $\nu \in \M_+$, and it is at times convenient to operate over variables that can vary freely inside a cone.

Concerning the limit of the free energy, one may hope that the formulas given in \eqref{eqn: SBM disassortative Hopf-Lax} and~\eqref{e.lim.barfn.t.mu} in the case $\Delta \le 0$ remain valid in general. It seems difficult to identify the exact range of validity of these formulas. We would be surprised if they hold for arbitrary measures $P^*$, but we could not quickly find a counter-example. We are however confident that these formulas will not generalize to settings with more than two communities. 

To see this, we rely on the fact that the problem of identifying the limit of the free energy becomes simpler in the dense regime. Indeed, if the average degree of a node diverges as $N$ tends to infinity, then central-limit-theorem effects take place, and one can equivalently study a fully-connected model with Gaussian noise \cite{Deshpande, Lelarge}. Such models have been studied extensively \cite{barbier2016, barbier2019adaptive, barbier2017layered, JC_HB_FinRank, chen2022hamilton, chen2021limiting, chen2022hamilton2, kadmon2018statistical, Lelarge, lesieur2017statistical,  luneau2021mutual, luneau2020high, mayya2019mutual, miolane2017fundamental,  JC_matrix, JC_HJ, reeves2020information, reeves2019geometry}. In this setting, a formula analogous to \eqref{eqn: SBM disassortative Hopf-Lax}-\eqref{e.lim.barfn.t.mu} is known to be valid as long as the relevant non-linearity is convex; but in general, one needs to modify this formula into a ``sup-inf'' formulation. Possibly the simplest setting in which this happens is for the problem in which we observe a rank-one matrix of the form $X Y^{\mathsf T}$ plus noise, where $X$ and $Y$ are two vectors with i.i.d.\ coordinates. In this setting, the non-linearity replacing~$\C_\infty$ in \eqref{eqn: SBM enriched free energy HJ equation} is the mapping $(x,y) \mapsto xy$, which is non-convex. The functional to be optimized over as in~\eqref{e.lim.barfn.t.mu} would look like $\psi(x_0 + t x,y_0 + t y) - \frac t 2 x y$. Finding counter-examples to the formula is made relatively easy by considering candidates with, say, $x = 0$; in this case, the counter-term $xy$ vanishes, so we can freely choose $y$ as large as desired to maximize the $\psi$ functional and obtain a contradiction. A similar phenomenon also occurs in the context of spin glasses, and a more precise discussion of this point can be found in Subsection~6.2 of \cite{JC_NC}. 

Coming back to the sparse setting investigated in this paper, we can leverage this observation to demonstrate that the formulas \eqref{eqn: SBM disassortative Hopf-Lax} or \eqref{e.lim.barfn.t.mu} will also be invalid in general. To give a concrete example, consider the following scenario, which can be thought of as a problem with four communities, or as a bipartite version of the two-community problem. We first color the $N$ nodes in red or blue, say with groups of sizes about $N/2$. We think of this coloring as fixed, e.g. the red nodes are the first $\lfloor N/2 \rfloor$ indices in $\{1,\ldots, N\}$, and it is perfectly known to the statistician. Next, we attribute $\pm 1$ labels to each node independently, possibly with different biases according to the color of the node. Finally, we draw links between nodes $i$ and $j$ according to the formula in \eqref{eqn: SBM edge probabilities}, with the additional constraint that only links between nodes of different colors are allowed. The task is to study the asymptotic behavior of the mutual information between the $\pm 1$ labels and the observed graph. This problem is constructed in such a way that, in the limit of diverging average degree, it reduces to the problem of observing a noisy version of $X Y^{\mathsf T}$, as discussed in the previous paragraph --- the vectors $X$ and $Y$ contain the $\pm 1$ labels of the red and blue nodes respectively. Using the results of \cite{Deshpande, Lelarge} to justify the large-degree approximation, or possibly even directly, we are confident that we can then produce counter-examples to the formulas \eqref{eqn: SBM disassortative Hopf-Lax} and \eqref{e.lim.barfn.t.mu}.

For fully-connected models with possibly non-convex non-linearities such as the $X Y^{\mathsf T}$ example, the limit of the free energy was identified in the form of a ``sup-inf'' formula; see \cite{JC_HB_FinRank} for the most general results. Translating this result into our present context would suggest that the limit free energy might be given by
\begin{equation}
\label{e.hopf}
\sup_{\rho \in \M_+} \inf_{\nu \in \M_+} \bigg( \psi(\nu) + \int_{-1}^1 G_\rho(y) \ud (\mu - \nu)(y) + t \int_{-1}^1 G_{\rho}(y) \ud \rho(y) \bigg).
\end{equation}
The key ingredient for showing the validity of the corresponding formula in the dense regime is that the enriched free energy is a convex function of its parameters in this case. In our setting, the question would translate into whether the mapping $(t,\mu) \mapsto \bar F_N(t,\mu)$ is convex. However, it was shown in \cite{kireeva2023breakdown} that this mapping is in fact \emph{not} convex in the sparse regime, even in the limit of large~$N$. This non-convexity property not only breaks down the proof strategy of \cite{JC_HB_FinRank}; in fact, we can leverage it to assert that the quantity \eqref{e.hopf} can therefore \emph{not} be the limit of the free energy in this case. Indeed, the expression in \eqref{e.hopf} is a supremum over $\rho$ of affine functions of~$(t,\mu)$; it therefore follows that the whole expression is convex in $(t,\mu)$. By \cite{kireeva2023breakdown}, it is therefore not possible that the expression in \eqref{e.hopf} be the limit of the free energy. 

To sum up, if we aim for a formula that is robust to model changes, then both \eqref{e.lim.barfn.t.mu} and~\eqref{e.hopf} can be ruled out. We do not know of alternative candidate variational formulas for the limit of the free energy. This situation seems analogous to that encountered in the context of spin glasses with possibly non-convex interactions, as discussed in Section~6 of \cite{JC_NC}.

We now turn to a discussion of the recent works \cite{Abbe_assortative, gu2023uniqueness, kanade2016global, mossel2016belief, mossel2016local, yu2022ising}. Noticing that the graph $\mathbf G_N$ locally looks like a tree, these works aim to leverage a connection between community detection and a process of broadcasting on trees. We briefly describe the latter problem on a regular tree for convenience. We first attribute a random $\pm 1$ random variable $\sigma^*$ to the root node. Then, recursively and independently along each edge, we ``broadcast'' it to each child node, by flipping the sign of the spin with some fixed probability $\delta \in (0,1)$. One basic question is to determine the mutual information between the spin $\sigma^*$ attributed to the root node and the spins on all the nodes at a given depth, in the limit of large depth. A fruitful variant of this question consists in adding a ``survey'' of all nodes, by randomly revealing the spins attached to each node independently with some fixed probability $\epsilon$. If, in the limit of large depth, the knowledge of the spins on all the leaf vertices does not bring meaningful additional information on $\sigma^*$ on top of surveying compared with surveying alone, then one can directly relate the mutual information between $\sigma^*$ and the survey to the mutual information in the community detection problem; in this case, one may speak of ``boundary irrelevance''. To decide whether boundary irrelevance holds, one can study the evolution of the log-likelihood ratio between the two hypotheses $\sigma^* = \pm 1$ upon revealing the boundary information at a given depth. One can indeed calculate the law of this quantity recursively as the depth varies, by iterating a fixed map called the ``BP operator''. In order to establish the property of boundary irrelevance, it then essentially suffices to show that this BP operator admits a unique non-trivial fixed point. Building upon earlier works, it was recently established in \cite{yu2022ising} that this uniqueness property holds for the setting corresponding to the detection of two balanced communities, $p = 1/2$. As a byproduct, this yields a full identification of the limit of the mutual information \eqref{eqn: SBM mutual information} in this case. The uniqueness of a non-trivial fixed point to the BP operator has subsequently been shown to be false in general for models with more than two communities \cite{gu2023uniqueness}. 

We now point out some connections between the present paper and this series of works, and discuss how our approach might ultimately be able to circumvent the difficulties associated with the possible existence of multiple fixed points to the BP operator. To start with, recall that the function~$\psi$ is the limit of $\bar F_N(0,\cdot)$, which itself corresponds to a simple inference problem in which there is no interaction between the nodes $\{1,\ldots, N\}$. We can therefore identify this object explicitly, see Lemma~\ref{SBM convergence of initial condition discrete}. From Remark~\ref{r.calculation.derivative.psi}, we can also identify a mapping $\Gamma : \M_+ \to \Pr[-1,1]$ such that for every $\mu \in \M_+$, we have $D_\mu \psi(\mu,\cdot) = G_{\Gamma(\mu)}$. This mapping is closely related to the BP operator discussed above, and is described as follows. Let $\sigma^*$ be sampled according to $P^*$, and conditionally on $\sigma^*$, let $\Pi(\mu)$ denote a Poisson point process with intensity measure ${(c+\Delta \sigma^* x)} \ud \mu(x)$. Then the probability measure $\Gamma(\mu)$ is defined to be the law of the random variable
\begin{equation}
\frac{\int_{\Sigma_1} \sigma \exp(-\Delta \sigma \int_{-1}^1 x\ud \mu) \prod_{x\in \Pi(\mu)} (c+\Delta \sigma x)\ud P^*(\sigma)}{\int_{\Sigma_1} \exp(-\Delta \sigma \int_{-1}^1 x\ud \mu) \prod_{x\in \Pi(\mu)} (c+\Delta \sigma x)\ud P^*(\sigma)}.
\end{equation}
%
Notice next that the condition for the measure $\nu$ to be a critical point in the variational problem on the right side of \eqref{e.lim.barfn.t.mu} can be written as
\begin{equation}
\label{e.crit.point.1}
 G_\nu = D_\mu \psi(\mu+t\nu, \cdot).
\end{equation}
At least when $\Delta < 0$, the mapping $\nu \mapsto G_\nu$ is also injective, so the relation \eqref{e.crit.point.1} can be equivalently written as
\begin{equation}  
\label{e.fixed.point}
\nu = \Gamma(\mu + t \nu).
\end{equation}
Restricting to the case of $(t,\mu) = (1,0)$, this boils down to finding fixed points of the mapping~$\Gamma$. That there is a connection between the variational formula in Theorem~\ref{SBM disassortative Hopf-Lax} and some BP fixed point equation has already been observed in \cite{CKPZ, DKMZ} and elsewhere. The less classical question is to relate this to the Hamilton-Jacobi equation \eqref{eqn: SBM enriched free energy HJ equation} for arbitrary $\Delta$. In finite dimensions, Hamilton-Jacobi equations can be solved for short times using the method of characteristics. Moreover, the slope of the characteristic line is computed by evaluating the gradient of the non-linearity at the gradient of the initial condition. In our context, the characteristic line emanating from a measure $\nu \in \M_+$ is the trajectory 
\begin{equation}  
\label{e.charact.line}
t' \mapsto (t',\nu - t' \Gamma(\nu)),
\end{equation}
for $t'$ varying in $\Rp$. As long as characteristic lines emanating from different choices of $\nu$ do not intersect each other, we can then calculate the value of the solution along each characteristic line using the equation and the fact that the gradient of the solution remains constant along each line~\cite{Evans}. The condition~\eqref{e.fixed.point} turns out to be equivalent to asking that the characteristic line emanating from $\mu + t \nu$ passes through the point $(t,\mu)$, since the latter condition can be written as $\mu = \mu + t \nu - t \Gamma(\mu + t\nu)$.
In other words, for each fixed $(t,\mu)$, there is a simple one-to-one correspondence between the fixed points to \eqref{e.fixed.point} and the characteristic lines that pass through~$(t,\mu)$. The formula for prescribing the value of the solution along a characteristic line starting from $\mu + t\nu$ is then as in the supremum in~\eqref{e.lim.barfn.t.mu}. As long as $t$ is sufficiently small that the equation~\eqref{e.fixed.point} has a unique solution for each~$\mu$, this gives us a clear procedure for computing the solution to \eqref{eqn: SBME Gateaux derivative density}. Once characteristic lines start to intersect each other, the viscosity solution to \eqref{eqn: SBM enriched free energy HJ equation} aggregates these conflicting trajectories in a physically reasonable way, and our conjecture is that the free energy $\bar F_N$ is tracking this in the limit of large $N$.

Another alternative to the conjecture proposed here would be that the limit of the free energy is the maximal value one gets by plugging every possible solution of the fixed-point equation~\eqref{e.fixed.point}  into the functional inside the supremum in \eqref{e.lim.barfn.t.mu}. But in view of the discussion in the previous paragraph, counter-examples to the variational formula in \eqref{e.lim.barfn.t.mu} seem to produce counter-examples to this possibility as well.

To conclude this introduction, we give a brief outline of the paper. In \Cref{sec: SBM HJ eqn derivation} we show that, up to an error vanishing with $N$, the enriched free energy \eqref{eqn: SBM enriched free energy} satisfies the infinite-dimensional Hamilton-Jacobi equation \eqref{eqn: SBM enriched free energy HJ equation}, provided that all multi-overlaps concentrate. 
The derivative computations that lead to the Hamilton-Jacobi equation are similar in spirit to those in Lemma 6 of \cite{PanDIL}, with some new ideas required to compute the Gateaux derivative. \Cref{sec: SBM HJ eqn WP} is devoted to establishing the well-posedness of the infinite-dimensional Hamilton-Jacobi equation \eqref{eqn: SBM enriched free energy HJ equation} using the results in \cite{TD_JC_HJ}, which in turn follows ideas from \cite{HB_HJ, HB_cone, JC_upper, JC_NC}.
In \Cref{sec: SBM upper bound}, a finitary version of the multi-overlap concentration result in \cite{BarMOC} is combined with the strategy introduced in \cite{JC_upper,JC_NC} to prove \Cref{SBM main result}. The final section is devoted to the proof of \Cref{SBM disassortative Hopf-Lax}. Using the Hopf-Lax formula established in~\cite{TD_JC_HJ}, the variational expression in \eqref{eqn: SBM disassortative Hopf-Lax} is shown to coincide with the right side of \eqref{eqn: SBM disassortative Hopf-Lax}, and we can thus appeal to \Cref{SBM main result} to obtain an upper bound for the limit free energy. The matching lower bound is obtained through an interpolation argument taken from \cite{PanNote}. So as to not disrupt the flow of the paper, a number of technical arguments have been postponed to the appendices. In \Cref{SBM app equivalence}, it is shown that the free energy functionals \eqref{eqn: SBM free energy 0} and \eqref{eqn: SBM free energy} are asymptotically equivalent. The proof relies on the binomial-Poisson approximation. \Cref{SBM app FE concentration} is devoted to proving that a perturbed version of the enriched free energy~\eqref{eqn: SBM enriched free energy} is self-averaging, in the sense that the unaveraged free energy concentrates around its average value. This concentration result plays an important part in the proof of \Cref{SBM main result} and relies upon the generalized Efron-Stein inequality \cite{Boucheron}. In \Cref{SBM app multioverlap concentration}, a finitary version of the multi-overlap concentration result in \cite{BarMOC} is established. In addition to being finitary, the most notable difference between our multi-overlap result and that in \cite{BarMOC} is that we show multi-overlap concentration for any perturbation parameter satisfying a condition that may be verified in practice, as opposed to obtaining multi-overlap concentration on average over the set of perturbation parameters. This additional control on the choice of parameters is essential in the proof of \Cref{SBM main result}.


\begin{acknowledgements}
We would like to warmly thank Dmitry Panchenko and Jean Barbier for sharing their notes \cite{PanNote} on the free energy in the disassortative sparse stochastic block model with us, which provided us with a very useful starting point and helped us with many of the computations in \Cref{sec: SBM HJ eqn derivation}. 
\end{acknowledgements}

\section{The Hamilton-Jacobi equation}\label{sec: SBM HJ eqn derivation}

In this section, we compute the derivative of the enriched free energy \eqref{eqn: SBM enriched free energy} with respect to $t\geq 0$ and $\mu \in\M_+$. This will allow us to see that, up to an error vanishing with $N$, the enriched free energy heuristically satisfies \eqref{eqn: SBM enriched free energy HJ equation}. It will be convenient to write $\langle \cdot \rangle$ for the average with respect to the Gibbs measure associated with the Hamiltonian \eqref{eqn: SBM enriched Hamiltonian}. This means that for any bounded and measurable function $f=f(\sigma^1,\ldots,\sigma^n)$ of finitely many replicas, 
\begin{equation}\label{eqn: SBM Gibbs average}
 \langle f(\sigma^1, \ldots, \sigma^n) \rangle = \langle f\rangle   =\frac{\int_{\Sigma_N^n}f(\sigma^1,\ldots,\sigma^n)\prod_{\ell\leq n}\exp H_N^{t,\mu}(\sigma^\ell)\ud P_N^*(\sigma^\ell)}{\big(\int_{\Sigma_N} \exp H_N^{t,\mu}(\sigma)\ud P_N^*(\sigma)\big)^n}.
\end{equation}
In this notation, the replicas $\sigma^1,\ldots,\sigma^n$ represent i.i.d.\@ samples under the random measure $\langle \cdot\rangle$.
By construction, we have that
\begin{equation}
\label{eqn: SBM Gibbs conditional}
    \langle f(\sigma^1) \rangle = \E \big[ f(\sigma^*) \big \vert \D_N^{t,\mu} \big],
\end{equation}
where $\smash{\D_N^{t,\mu}=\widetilde{\D}_N^{t,\mu[-1,1],\bar{\mu}}}$ is the data defined in \eqref{eqn: SBM enriched data}.

Our computations will be considerably simplified by the Nishimori identity. This identity will allow us to freely interchange one replica $\sigma^\ell$ by the signal $\sigma^*$ when taking an average with respect to all sources of randomness, thus avoiding a cascade of new replicas as we differentiate the free energy. This identity states that, for every bounded and measurable function $f=f(\sigma^1,\ldots,\sigma^n,\D_N^{t,\mu})$ of finitely many replicas and the data, 
\begin{equation}\label{eqn: SBM Nishimori identity}
\E\big\langle f\big(\sigma^1,\sigma^2,\ldots,\sigma^n, \D_N^{t,\mu}\big)\big\rangle=\E\big\langle f\big(\sigma^*,\sigma^2,\ldots,\sigma^n, \D_N^{t,\mu}\big)\big\rangle.
\end{equation}
This can be first verified for functions of product form using \eqref{eqn: SBM Gibbs conditional}, and then extended to all bounded and measurable functions by a monotone class argument.

We now turn our attention to the computation of the time derivative of the enriched free energy~\eqref{eqn: SBM enriched free energy}. We fix a finite measure $\mu \in \M_+$ and proceed as in Lemma 6 of \cite{PanDIL}. For each parameter $\lambda>0$ and every integer $m\geq 0$, we denote by
\begin{equation}\label{eqn: SBM Poisson density}
\pi(\lambda,m)=\frac{\lambda^m}{m!}\exp(-\lambda)
\end{equation}
the mass attributed to the atom $m$ by a $\Poi(\lambda)$ distribution. It will be convenient to set the convention that $\pi(\lambda,-1) = 0$. We write
\begin{equation}\label{eqn: SBM conditioned Hamiltonian 1}
H_{N,m}(\sigma)=\sum_{k\leq m}\log \bigg[\big(c+\Delta \sigma_{i_k}\sigma_{j_k}\big)^{G^k_{i_k,j_k}}\Big(1-\frac{c+\Delta \sigma_{i_k}\sigma_{j_k}}{N}\Big)^{1-G^k_{i_k,j_k}}\bigg]
\end{equation}
for the Hamiltonian \eqref{eqn: SBM Hamiltonian} conditional on there being $m$ terms in the sum, and introduce the partition function
\begin{equation}
Z_{N,m}=\int_{\Sigma_N} \exp\Big(H_{N,m}(\sigma)+\widetilde{H}_N^{\mu[-1,1],\bar{\mu}}(\sigma)\Big)\ud P_N^*(\sigma).
\end{equation}
In this notation, the enriched free energy \eqref{eqn: SBM enriched free energy} may be expressed as
\begin{equation}\label{eqn: SBM enriched free energy rep1}
\overline{F}_N(t,\mu)=\frac{1}{N}\sum_{m\geq 0}\pi\Big(t\binom{N}{2},m\Big)\E \log Z_{N,m}.
\end{equation}
To take the time derivative of this expression, we will rely upon the simple fact that
\begin{equation}\label{eqn: SBM Poisson density derivative}
\partial_\lambda\pi(\lambda,m)=\pi(\lambda,m-1)-\pi(\lambda,m).
\end{equation}

\begin{lemma}\label{SBM enriched free energy time derivative}
For any $t>0$ and $\mu \in \M_+$,
\begin{equation}\label{eqn: SBM enriched free energy time derivative}
\partial_t\overline{F}_N(t,\mu)=\frac{1}{2}\E \big(c+\Delta \langle \sigma_i\sigma_j\rangle\big)\log \big( c+\Delta \langle \sigma_i\sigma_j\rangle\big)-\frac{\Delta \m^2}{2}-\frac{c}{2}+\BigO(N^{-1}),
\end{equation}
where the indices $i,j\in \{1,\ldots,N\}$ are uniformly sampled independently of all other sources of randomness.
\end{lemma}

\begin{proof}
To simplify notation, let $\lambda(t)=t\binom{N}{2}$. Leveraging \eqref{eqn: SBM Poisson density derivative} to differentiate the right-hand side of \eqref{eqn: SBM enriched free energy rep1} yields
\begin{align}\label{eqn: SBM enriched free energy time derivative key}
\partial_t\overline{F}_N(t,\mu)&=\frac{1}{N}\binom{N}{2}\sum_{m\geq 0} \big(\pi(\lambda(t),m-1)-\pi(\lambda(t),m)\big) \E \log Z_{N,m}\notag\\
&=\frac{1}{N}\binom{N}{2}\sum_{m\geq 0}\pi(\lambda(t),m)\E \log \frac{Z_{N,m+1}}{Z_{N,m}}.
\end{align}
Denote by $i,j\in \{1,\ldots,N\}$ uniformly sampled indices, and write $G_{i,j}$ for a random variable with conditional distribution \eqref{eqn: SBM G distribution}. These random variables are taken to be independent of all other sources of randomness and of each other. Since
$$Z_{N,m+1}\stackrel{d}{=} \int_{\Sigma_N} \big(c+\Delta \sigma_i\sigma_j\big)^{G_{i,j}}\Big(1-\frac{c+\Delta \sigma_i\sigma_j}{N}\Big)^{1-G_{i,j}}\exp \Big(H_{N,m}(\sigma)+H_N^{\mu[-1,1],\bar{\mu}}(\sigma)\Big)\ud P_N^*(\sigma),$$
it follows by \eqref{eqn: SBM enriched free energy time derivative key} and the definition of the Gibbs average in \eqref{eqn: SBM Gibbs average} that
$$\partial_t \overline{F}_N(t,\mu)=\frac{1}{N}\binom{N}{2}\E \log \bigg\langle \big(c+\Delta \sigma_i\sigma_j\big)^{G_{i,j}}\Big(1-\frac{c+\Delta \sigma_i\sigma_j}{N}\Big)^{1-G_{i,j}}\bigg\rangle.$$
Remembering the explicit form of the conditional distribution \eqref{eqn: SBM G distribution}, and averaging with respect to the randomness of $G_{i,j}$ reveals that
\begin{align*}
\partial_t \overline{F}_N(t,s,\mu)=\frac{1}{2}\E\big(c+\Delta \sigma_i^*\sigma_j^*\big)& \log \langle c+\Delta \sigma_i\sigma_j\rangle\\
&+\frac{N}{2}\E\bigg(1-\frac{c+\Delta \sigma_i^*\sigma_j^*}{N}\bigg)\log \Big\langle 1-\frac{c+\Delta \sigma_i\sigma_j}{N}\Big\rangle+\BigO(N^{-1}),
\end{align*}
Taylor expanding the logarithm and keeping only first order terms reduces this to
\begin{align*}
\partial_t \overline{F}_N(t,\mu)&=\frac{1}{2}\E \big(c+\Delta \sigma_i^*\sigma_j^*\big)\log \langle c+\Delta \sigma_i\sigma_j\rangle-\frac{\Delta}{2}\E \langle\sigma_i\sigma_j\rangle-\frac{c}{2}+\BigO(N^{-1})\\
&=\frac{1}{2}\E \big(c+\Delta \sigma_i^*\sigma_j^*\big)\log \langle c+\Delta \sigma_i\sigma_j\rangle-\frac{\Delta}{2}\E \sigma_i^*\sigma_j^*-\frac{c}{2}+\BigO(N^{-1})\\
&=\frac{1}{2}\E \big(c+\Delta \sigma_i^*\sigma_j^*\big)\log \big(c+\Delta \langle\sigma_i\sigma_j\rangle\big)-\frac{\Delta\m^2}{2}-\frac{c}{2}+\BigO(N^{-1}),
\end{align*}
where the second equality uses the Nishimori identity \eqref{eqn: SBM Nishimori identity} and the third equality uses the fact that $i$ and $j$ are distinct with overwhelming probability. Noticing that the Gibbs average $\langle \sigma_i\sigma_j\rangle$ is a measurable function of the data by \eqref{eqn: SBM Gibbs conditional}, and applying the Nishimori identity \eqref{eqn: SBM Nishimori identity} completes the proof.
\end{proof}

To compare \eqref{eqn: SBM enriched free energy time derivative} with the Gateaux derivative of the enriched free energy which we will compute below, it will be convenient to Taylor expand the logarithm. This will make the dependence of the time-derivative of the enriched free energy on the multi-overlaps 
\begin{equation}\label{eqn: SBM enriched multi-overlap}
R_{\ell_1,\ldots,\ell_n}=\frac{1}{N}\sum_{i\leq N}\sigma_i^{\ell_1}\cdots \sigma_i^{\ell_n}
\end{equation}
associated with the enriched Hamiltonian \eqref{eqn: SBM enriched Hamiltonian} explicit. Here $(\sigma^\ell)$ denotes a sequence of i.i.d.\@ replicas sampled from the Gibbs measure \eqref{eqn: SBM Gibbs average}. To simplify notation, we will write $R_{[n]}=R_{1,\ldots,n}$.

\begin{corollary}\label{SBM enriched free energy time derivative Taylor}
For any $t>0$ and $\mu \in \M_+$,
\begin{equation}
\partial_t \overline{F}_N(t,\mu)=\frac{1}{2}\big(c+\Delta\m^2\big)\log (c)+\frac{c}{2}\sum_{n\geq 2}\frac{(-\Delta/c)^n}{n(n-1)}\E\big\langle R_{[n]}^2\big\rangle-\frac{c}{2}+\BigO(N^{-1}).
\end{equation}
\end{corollary}

\begin{proof}
A Taylor expansion of the logarithm shows that 
$$\log\big( c+\Delta \langle\sigma_i\sigma_j\rangle\big)=\log(c)+\log \Big(1+\frac{\Delta}{c}\langle \sigma_i\sigma_j\rangle\Big)=\log(c)-\sum_{n\geq 1}\frac{(-\Delta/c)^n}{n}\langle \sigma_i\sigma_j\rangle^n.$$
Together with the Nishimori identity \eqref{eqn: SBM Nishimori identity} this implies that
\begin{align}\label{eqn: SBM enriched free energy time derivative rep2 key}
\E\big(c+\Delta \langle \sigma_i \sigma_j\rangle\big)\log\big( c+\Delta \langle\sigma_i\sigma_j\rangle\big)=\big(c+\Delta \m^2\big)&\log(c)\notag\\
&-\sum_{n\geq 1}\frac{(-\Delta/c)^n}{n}\E\big(c+\Delta\langle \sigma_i\sigma_j\rangle\big)\langle \sigma_i\sigma_j\rangle^n.
\end{align}
Averaging with respect to the randomness of the uniformly sampled indices $i,j\in \{1,\ldots,N\}$ reveals that
$$\E\big(c+\Delta\langle \sigma_i\sigma_j\rangle \big)\langle \sigma_i\sigma_j\rangle^n=c\E\big\langle R_{[n]}^2\big\rangle+\Delta \E\big\langle R_{[n+1]}^2\big\rangle.$$
Remembering that $\abs{\Delta}<c$ and noticing that $\E\langle R_1^2\rangle=\m^2+\BigO(N^{-1})$ by the Nishimori identity, it follows that
\begin{align*}
\sum_{n\geq 1}\frac{(-\Delta/c)^n}{n}\E\big(c+\Delta\langle \sigma_i\sigma_j\rangle\big)&\langle \sigma_i\sigma_j\rangle^n\\
&=-\Delta\E\big\langle R_1^2\big\rangle+c\sum_{n\geq 2}\bigg(\frac{(-\Delta/c)^n}{n}-\frac{(-\Delta/c)^n}{n-1}\bigg)\E\big\langle R_{[n]}^2\big\rangle +\BigO(N^{-1})\\
&=-\Delta \m^2-c\sum_{n\geq 2}\frac{(-\Delta/c)^n}{n(n-1)}\E\big\langle R_{[n]}^2\big\rangle +\BigO(N^{-1}).
\end{align*}
Substituting this into \eqref{eqn: SBM enriched free energy time derivative rep2 key} and invoking \Cref{SBM enriched free energy time derivative} completes the proof.
\end{proof}

The computation of the Gateaux derivative of the enriched free energy \eqref{eqn: SBM enriched free energy} at a measure $\smash{\mu \in \M_+}$ in the direction of a probability measure $\smash{\nu\in \Pr[-1,1]}$,
\begin{equation}\label{eqn: SBM free energy Gateaux derivative}
D_\mu \overline{F}_N(t,\mu; \nu)=\lim_{\epsilon \to 0}\frac{\overline{F}_N(t,\mu+\epsilon \nu)-\overline{F}_N(t,\mu)}{\epsilon},
\end{equation}
is slightly more involved. It will be useful to compute the derivative of the free energy \eqref{eqn: SBM enriched free energy s} with respect to the parameter $s\geq 0$ first. Fix a probability measure $\mu\in \Pr[-1,1]$ and a time $t\geq 0$. For each $i\leq N$ write
\begin{align}\label{eqn: SBM conditioned Hamiltonian 2}
\widetilde{H}_{N,m}^{s,i}(\sigma)=\sum_{j\neq i}\sum_{k\leq \Pi_{j,s}}\log&\bigg[\big(c+\Delta \sigma_jx_{j,k}\big)^{\tG_{j,k}^x}\Big(1-\frac{c+\Delta \sigma_jx_{j,k}}{N}\Big)^{1-\tG_{j,k}^x}\bigg]\notag\\
& +\sum_{k\leq m}\log\bigg[\big(c+\Delta \sigma_ix_{i,k}\big)^{\tG_{i,k}^x}\Big(1-\frac{c+\Delta \sigma_ix_{i,k}}{N}\Big)^{1-\tG_{i,k}^x}\bigg]
\end{align}
for the Hamiltonian \eqref{eqn: SBM Hamiltonian s} conditional on the $i$'th Poisson sum containing $m$ terms, and denote by
\begin{equation}
Z_{N,m}^{s,i}=\int_{\Sigma_N} \exp \big(H_{N}^t(\sigma)+\widetilde{H}^{s,i}_{N,m}(\sigma)\big)\ud P_N^*(\sigma)
\end{equation}
its associated partition function. In this notation, the free energy \eqref{eqn: SBM enriched free energy s} may be expressed as
\begin{equation}\label{eqn: SBM enriched free energy rep2}
\widetilde{F}_N(t,s,\mu)=\frac{1}{N}\sum_{m\geq 0}\pi(sN,m)\E \log Z_{N,m}^{s,i}.
\end{equation}

\begin{lemma}\label{SBM enriched free energy s derivative}
For any $t>0$, $s>0$ and $\mu \in \Pr[-1,1]$,
\begin{equation}\label{eqn: SBM enriched free energy s derivative}
\partial_s \widetilde{F}_N(t,s,\mu)=\E\big(c+\Delta \langle \sigma_i\rangle x_i\big)\log \big(c+\Delta \langle\sigma_i\rangle x_i\big)-c-\Delta \m \E x_1+\BigO(N^{-1}),
\end{equation}
where the index $i\in \{1,\ldots,N\}$ is uniformly sampled and the random variables $(x_i)$ are sampled from the measure $\mu$ independently of all other sources of randomness.
\end{lemma}

\begin{proof}
Conditioning on the number of terms in each of the Poisson sums that appear in the definition of the free energy \eqref{eqn: SBM enriched free energy s} and leveraging the product rule as well as equations \eqref{eqn: SBM enriched free energy rep2} and \eqref{eqn: SBM Poisson density derivative}, we see that
\begin{align}\label{eqn: SBM enriched free energy s derivative key}
\partial_s \widetilde{F}_N(t,s,\mu)&=\frac{1}{N}\sum_{i\leq N}\sum_{m\geq 0}\partial_s\pi(sN,m)\E \log Z_{N,m}^{s,i}=\sum_{i\leq N}\sum_{m\geq 0}\pi(sN,m)\E\log \frac{Z_{N,m+1}^{s,i}}{Z_{N,m}^{s,i}}.
\end{align}
For each $i\leq N$, denote by $x_i$ a sample from the measure $\mu$, and write $\tG_i^x$ for a random variable with conditional distribution \eqref{eqn: SBM tG distribution}. These random variables are taken to be independent for different values of $i\leq N$, and independent of all other sources of randomness. Since
\begin{equation}\label{eqn: SBM s derivative distributional equality}
Z_{N,m+1}^{s,i}\stackrel{d}{=}\int_{\Sigma_N} \big(c+\Delta \sigma_ix_i\big)^{\tG_i^x}\Big(1-\frac{c+\Delta \sigma_ix_i}{N}\Big)^{1-\tG_i^x}\exp \big(H_N^t(\sigma)+\widetilde{H}_{N,m}^{s,i}(\sigma)\big)\ud P_N^*(\sigma),
\end{equation}
it follows by \eqref{eqn: SBM enriched free energy s derivative key} and the definition of the Gibbs average in \eqref{eqn: SBM Gibbs average} that
\begin{equation}\label{eqn: SBM enriched free energy s derivative exact}
\partial_s \widetilde{F}_N(t,s,\mu)=\sum_{i\leq N}\E \log \bigg\langle\big(c+\Delta \sigma_ix_i\big)^{\tG_i^x}\Big(1-\frac{c+\Delta \sigma_ix_i}{N}\Big)^{1-\tG_i^x} \bigg\rangle.
\end{equation}
Remembering the explicit form of the conditional distribution \eqref{eqn: SBM G distribution} reveals that
\begin{align*}
\partial_s \widetilde{F}_N(t,s,\mu)=\frac{1}{N}\sum_{i\leq N}\E\big(c+\Delta \sigma_i^*x_i\big)\log \langle c&+\Delta \sigma_ix_i\rangle\\
&+\sum_{i\leq N}\E\Big(1-\frac{c+\Delta \sigma_i^* x_i}{N}\Big)\log \Big\langle 1-\frac{c+\Delta \sigma_ix_i}{N}\Big\rangle.
\end{align*}
Taylor expanding the logarithm and keeping only first order terms reduces this to 
\begin{align*}
\partial_s \widetilde{F}_N(t,s,\mu)&=\frac{1}{N}\sum_{i\leq N}\E\big(c+\Delta \sigma_i^*x_i\big)\log \langle c+\Delta \sigma_ix_i\rangle-c-\frac{\Delta}{N}\sum_{i\leq N}\E x_i\E\langle \sigma_i\rangle+\BigO(N^{-1})\\
&=\frac{1}{N}\sum_{i\leq N}\E\big(c+\Delta \sigma_i^* x_i\big)\log \big(c+\Delta \langle \sigma_i\rangle x_i\big)-c-\Delta \m \E x_1+\BigO(N^{-1}),
\end{align*}
where the second equality uses the Nishimori identity \eqref{eqn: SBM Nishimori identity}. Noticing that the Gibbs average $\langle \sigma_i\rangle$ is a measurable function of the data by \eqref{eqn: SBM Gibbs conditional} and applying the Nishimori identity \eqref{eqn: SBM Nishimori identity} completes the proof.
\end{proof}

Before leveraging this result to compute the Gateaux derivative \eqref{eqn: SBM free energy Gateaux derivative}, we begin with some distributional identities which will simplify the calculation. Fix a finite measure $\mu\in \M_+$ and a probability measure $\nu\in \Pr[-1,1]$. Let $s=\mu[-1,1]$ and fix $\epsilon>0$. Introduce the measure $\lambda=\mu+\epsilon \nu$ and observe that 
\begin{equation}\label{eqn: SBM sum of measures is mixture}
\bar{\lambda}=\frac{\lambda}{s+\epsilon}=\frac{s}{s+\epsilon}\bar{\mu}+\frac{\epsilon}{s+\epsilon}\nu.
\end{equation}
Denote by $(x_{i,k})$ i.i.d.\@ random variables sampled from the measure $\bar{\mu}$, and write $(y_{i,k})$ for i.i.d.\@ random variables sampled from the measure $\nu$. Given i.i.d.\@ random variables $(w_{i,k})$ with distribution $\smash{\Ber(\frac{s}{s+\epsilon})}$, notice that by \eqref{eqn: SBM sum of measures is mixture} the random variables
\begin{equation}
z_{i,k}=x_{i,k}^{w_{i,k}}y_{i,k}^{1-w_{i,k}}
\end{equation}
are i.i.d.\@ with distribution $\bar{\lambda}$. In particular, if $(\tG_{i,k}^z)$ are independent random variables with conditional distribution \eqref{eqn: SBM tG distribution}, the Hamiltonian \eqref{eqn: SBM Hamiltonian s} may be expressed as
\begin{equation}\label{eqn: SBM GD distributional identity}
\widetilde{H}_N^{s+\epsilon,\bar{\lambda}}(\sigma)\stackrel{d}{=}\sum_{i\leq N}\sum_{k\leq \Pi_{i,s+\epsilon}}\log\bigg[\big(c+\Delta \sigma_iz_{i,k}\big)^{\tG_{i,k}^z}\Big(1-\frac{c+\Delta \sigma_iz_{i,k}}{N}\Big)^{1-\tG_{i,k}^z}\bigg].
\end{equation}
This identity will allow us to linearize the enriched free energy \eqref{eqn: SBM enriched free energy} upon realizing that
\begin{equation}
\overline{F}_N(t,\mu+\epsilon \nu)=\widetilde{F}_N\big(t,s+\epsilon,\bar \lambda\big).
\end{equation}
To make the computation as clear as possible, it will be convenient to introduce additional notation. In the same spirit as \eqref{eqn: SBM conditioned Hamiltonian 2}, for each $i\leq N$, write
\begin{align}\label{eqn: SBM conditioned Hamiltonian 3}
\widetilde{H}_{N,m}^{s,i,+}(\sigma)&=\widetilde{H}_{N,m}^{s,i}+\log\bigg[\big(c+\Delta \sigma_iy_i\big)^{\tG_{i}^y}\Big(1-\frac{c+\Delta \sigma_iy_{i}}{N}\Big)^{1-\tG_{i}^y}\bigg]
\end{align}
for the Hamiltonian \eqref{eqn: SBM Hamiltonian s} conditional on the $i$'th Poisson sum containing $m+1$ terms one of which is sampled from the measure $\nu$. Denote by
\begin{equation}
Z_{N,m}^{s,i,+}=\int_{\Sigma_N} \exp\big(H_N^t(\sigma)+\widetilde{H}_{N,m}^{s,i,+}(\sigma)\big)\ud P_N^*
\end{equation}
its associated partition function. It will be useful to record the following consequence of Taylor's theorem,
\begin{equation}\label{eqn: SBM GD Taylor}
\Big(\frac{s}{s+\epsilon}\Big)^{\sum_{i\leq N}m_i}=1-\frac{\epsilon}{s+\epsilon}\sum_{i\leq N}m_i+o(\epsilon),
\end{equation}
as well as the elementary identity,
\begin{equation}\label{eqn: SBM Poisson density multiplication}
m\pi(\lambda,m)=\lambda\pi(\lambda,m-1).
\end{equation}

\begin{lemma}\label{SBM enriched free energy Gateaux derivative}
For any $t>0$, $\mu \in \M_+$ and $\nu \in \Pr[-1,1]$,
\begin{align}\label{eqn: SBM enriched free energy Gateaux derivative}
D_\mu \overline{F}_N(t,\mu; \nu)=\E\big(c+\Delta \langle \sigma_i\rangle y_i\big)\log\big(c&+\Delta\langle \sigma_i\rangle y_i\big)\notag\\
&+N\E\Big(1-\frac{c+\Delta \langle \sigma_i\rangle y_i}{N}\Big)\log\Big(1-\frac{c+\Delta\langle \sigma_i\rangle y_i}{N}\Big)
\end{align}
where the index $i\in \{1,\ldots,N\}$ is uniformly sampled and the random variables $(y_i)$ are sampled from the measure $\nu$ independently of all other sources of randomness.
\end{lemma}

\begin{proof}
Leveraging \eqref{eqn: SBM GD distributional identity}, conditioning on the number of random variables $(w_{i,k})$ that are equal to one and using \eqref{eqn: SBM GD Taylor}, we see that
\begin{align*}
\widetilde{F}_N(t,s+\epsilon,\bar \lambda)=\widetilde{F}_N(t,s+\epsilon, \bar \mu)&-\frac{\epsilon}{s+\epsilon}\sum_{i\leq N}\sum_{m\geq 0}m\pi((s+\epsilon)N,m)\frac{1}{N}\E\log Z_{n,m}^{s+\epsilon,i}\\
&+\frac{\epsilon}{s+\epsilon}\sum_{i\leq N}\sum_{m\geq 0}(m+1)\pi((s+\epsilon)N,m+1)\frac{1}{N}\E\log Z_{N,m}^{s+\epsilon,i,+}+o(\epsilon).
\end{align*}
Keeping in mind \eqref{eqn: SBM Poisson density multiplication}, this simplifies to
\begin{align}\label{eqn: SBM enriched free energy GD key}
\widetilde{F}_N\big(t,s+\epsilon,\bar{\lambda}\big)=\widetilde{F}_N&(t,s+\epsilon,\bar{\mu})\notag\\
&+\epsilon \sum_{i\leq N}\sum_{m\geq 0}\pi\big((s+\epsilon)N,m\big)\bigg(\E\log \frac{Z_{N,m}^{s+\epsilon,i,+}}{Z_{N,m}^{s+\epsilon,i}}-\E\log \frac{Z_{N,m+1}^{s+\epsilon,i}}{Z_{N,m}^{s+\epsilon,i}}\bigg)+o(\epsilon).    
\end{align}
For each $i\leq N$, denote by $x_i$ a sample from the measure $\bar{\mu}$ and by $y_i$ a sample from the measure $\nu$. Write $\tG_i^x$ and $\tG_i^y$ for random variables with conditional distribution \eqref{eqn: SBM tG distribution}. These random variables are taken to be independent for different values of $i\leq N$, and independent of all other sources of randomness. Combining \eqref{eqn: SBM enriched free energy GD key} with \eqref{eqn: SBM s derivative distributional equality} and the distributional identity
$$Z_{N,m}^{s+\epsilon,i,+}\stackrel{d}{=}\int_{\Sigma_N} \big(c+\Delta \sigma_i y_i\big)^{\tG_i^y}\Big(1-\frac{c+\Delta \sigma_i y_i}{N}\Big)^{1-\tG_i^y}\exp\big(H_N^t(\sigma)+\widetilde{H}_{N,m}^{s+\epsilon,i}(\sigma)\big)\ud P_N^*(\sigma)$$
yields
\begin{align}\label{eqn: SBM enriched free energy Gateaux derivative key}
\widetilde{F}_N(t,s+\epsilon,\bar{\lambda})=\widetilde{F}_N(t,s+&\epsilon,\bar{\mu})+\epsilon \sum_{i\leq N}\E\log \bigg\langle\big(c+\Delta \sigma_iy_i\big)^{\tG_i^y}\Big(1-\frac{c+\Delta \sigma_iy_i}{N}\Big)^{1-\tG_i^y} \bigg\rangle\notag\\
&\quad -\epsilon \sum_{i\leq N}\E\log \bigg\langle\big(c+\Delta \sigma_ix_i\big)^{\tG_i^x}\Big(1-\frac{c+\Delta \sigma_ix_i}{N}\Big)^{1-\tG_i^x} \bigg\rangle+o(\epsilon).
\end{align}
Together with \eqref{eqn: SBM enriched free energy s derivative exact}, this implies that
\begin{align*}
D_\mu \overline{F}_N(t,\mu;\nu)&=\lim_{\epsilon \to 0}\frac{\widetilde{F}_N(t,s+\epsilon,\bar{\lambda})-\widetilde{F}_N(t,s+\epsilon,\bar{\mu})}{\epsilon}+\partial_s \overline{F}_N(t,s,\bar{\mu})\\
&=\sum_{i\leq N}\E\log \bigg\langle \big(c+\Delta \sigma_iy_i\big)^{\tG_i^y}\Big(1-\frac{c+\Delta \sigma_iy_i}{N}\Big)^{1-\tG_i^y} \bigg\rangle.
\end{align*}
Notice that the Gibbs averages in \eqref{eqn: SBM enriched free energy Gateaux derivative key} depend on $\epsilon$, so in taking this limit we have implicitly used the fact that this dependence is continuous. Proceeding exactly as in the proof of \Cref{SBM enriched free energy s derivative} (after display \eqref{eqn: SBM enriched free energy s derivative exact}) completes the proof.
\end{proof}

To compare \eqref{eqn: SBM enriched free energy Gateaux derivative} with the time-derivative of the enriched free energy in \Cref{SBM enriched free energy time derivative Taylor} it will be convenient to once again Taylor expand the logarithm. We will write
\begin{align}\label{eqn: SBM Gateaux derivative density}
D_\mu \overline{F}_N(t,\mu,x)=\E\big(c+\Delta \langle \sigma_i\rangle x\big)\log\big(c&+\Delta\langle \sigma_i\rangle x\big)\notag\\
&+N\E\Big(1-\frac{c+\Delta \langle \sigma_i\rangle x}{N}\Big)\log\Big(1-\frac{c+\Delta\langle \sigma_i\rangle x}{N}\Big)
\end{align}
for the density of the Gateaux derivative $\smash{D_\mu \overline{F}_N(t,\mu)}$. Taylor expanding the logarithm shows that
\begin{equation}\label{eqn: SBM enriched free energy approximate Gateaux derivative}
D_\mu \overline{F}_N(t,\mu,x)=\E\big(c+\Delta \langle \sigma_i\rangle x\big)\log \big(c+\Delta \langle\sigma_i\rangle x\big)-c-\Delta \m x+\BigO(N^{-1}).
\end{equation}

\begin{corollary}\label{SBM enriched free energy Gateaux derivative Taylor}
For every $t>0$ and $\mu \in \M_+$,
\begin{equation}\label{eqn: SBM enriched free energy Gateaux derivative Taylor}
D_\mu \overline{F}_N(t,\mu,x)=\big(c+\Delta \m x\big)\log(c)+c\sum_{n\geq 2}\frac{(-\Delta/c)^n}{n(n-1)}\E\langle R_{[n]}\rangle x^n-c+\BigO(N^{-1}).
\end{equation}
\end{corollary}

\begin{proof}
Fix $\nu\in \Pr[-1,1]$. A Taylor expansion of the logarithm shows that
\begin{align}\label{eqn: SBM enriched free energy Gateaux derivative rep1 key}
\E\big(c+\Delta \langle \sigma_i\rangle y_i\big)\log\big(c+\Delta \langle \sigma_i\rangle y_i\big)=\big(c+\Delta \m \E y_1\big)&\log(c)\notag\\
&-\sum_{n\geq 1}\frac{(-\Delta/c)^n}{n}\E\big(c+\Delta\langle\sigma_i\rangle y_i\big)y_i^n\langle \sigma_i\rangle^n,
\end{align}
where we have used that $y_i$ is sampled from $\nu\in \Pr[-1,1]$ and is independent of all other sources of randomness. 
Since $\abs{\Delta}< c$ and $\E \langle \sigma_i\rangle=\m$ by the Nishimori identity,
\begin{align*}
\sum_{n\geq 1}\frac{(-\Delta/c)^n}{n}\E\big(c+\Delta \langle \sigma_i\rangle y_i\big)y_i^n\langle \sigma_i\rangle^n&=\sum_{n\geq 1}\frac{(-\Delta/c)^n}{n}\big(c\E y_i^n\E\langle \sigma_i\rangle^n+\Delta \E y_i^{n+1}\E \langle\sigma_i\rangle^{n+1}\big)\\
&=-\Delta \E y_i\E\langle \sigma_i\rangle+c\sum_{n\geq 2}\bigg(\frac{(-\Delta/c)^n}{n}-\frac{(-\Delta/c)^n}{n-1}\bigg)\E y_i^n\E\langle \sigma_i\rangle^n\\
&=-\Delta \m \E y_1-c\sum_{n\geq 2}\frac{(-\Delta/c)^n}{n(n-1)} \E\langle R_{[n]}\rangle\E y_i^n.
\end{align*}
Substituting this into \eqref{eqn: SBM enriched free energy Gateaux derivative rep1 key} and recalling \eqref{eqn: SBM enriched free energy approximate Gateaux derivative} completes the proof.
\end{proof}

This result implies that the Gateaux derivative density \eqref{eqn: SBM Gateaux derivative density} is close to an element in the cone of functions \eqref{eqn: SBM cone of functions}. Indeed, if $\mu^*=\L(\langle \sigma_i\rangle)$ denotes the law of the Gibbs average of a uniformly sampled spin coordinate, then \eqref{eqn: SBM enriched free energy approximate Gateaux derivative} may be {formally} rewritten as
\begin{equation}\label{eqn: SBM enriched free energy in cone}
D_\mu \overline{F}_N(t,\mu,x)\simeq \int_{-1}^1 g(xy)\ud \mu^*(y)=G_{\mu^*}(x),
\end{equation}
using the Nishimori identity \eqref{eqn: SBM Nishimori identity} to assert that $\E\langle \sigma_i \rangle = \m$. It follows by another application of the Nishimori identity that
\begin{equation}
\C_\infty\big(D_\mu \overline{F}_N(t,\mu)\big)\simeq \frac{1}{2}\big(c+\Delta \m^2\big)\log(c)+\frac{c}{2}\sum_{n\geq 2}\frac{(-\Delta/c)^n}{n(n-1)}\big(\E\langle R_{[n]}\rangle\big)^2-\frac{c}{2}.
\end{equation}
Comparing this with the expression in \Cref{SBM enriched free energy time derivative Taylor}, and assuming the approximate concentration of all the multi-overlaps,
\begin{equation}\label{eqn: SBM approximate multioverlap concentration}
\E\langle R_{[n]}^2\rangle\simeq \big(\E\langle R_{[n]}\rangle\big)^2,
\end{equation}
reveals that, up to a small error, the enriched free energy \eqref{eqn: SBM enriched free energy} formally satisfies the infinite-dimensional Hamilton-Jacobi equation \eqref{eqn: SBM enriched free energy HJ equation},
\begin{equation}
\partial_t f(t,\mu)=\C_\infty(D_\mu f(t,\mu)) \,\,\, \text{ on } \Rpp\times \M_+.
\end{equation}

As already mentioned in the introduction, the difficulty in making this informal derivation rigorous is that we do not expect the concentration of the multi-overlaps \eqref{eqn: SBM approximate multioverlap concentration} to be valid for each choice of the parameters $t$ and $\mu$. On the positive side, the arguments in \cite{BarMOC} reveal that the concentration of the multi-overlaps can be enforced through a small perturbation of the Hamiltonian which does not affect the limit of the free energy, for most values of the perturbation parameters. Yet, the solution theory to Hamilton-Jacobi equations is rather sensitive to details, and in particular, this control ``for most values'' or after a suitable local averaging is not sufficient to allow us to conclude. The next section will clarify the nature of this solution theory.

\section{Well-posedness of the Hamilton-Jacobi equation}\label{sec: SBM HJ eqn WP}

In this section, we leverage the main result in \cite{TD_JC_HJ} to establish the well-posedness of the infinite-dimensional Hamilton-Jacobi equation  \eqref{eqn: SBM enriched free energy HJ equation}. The first order of business will be to identify the initial condition in  \eqref{eqn: SBM enriched free energy HJ equation}. For each integer $N\geq 1$, let
\begin{equation}\label{eqn: SBM HJ initial condition}
\psi_N(\mu)=\overline{F}_N(0,\mu)=\widetilde{F}_N(0,\mu[-1,1],\bar \mu),
\end{equation}
and notice that the initial condition $\psi:\M_+\to \R$ in  \eqref{eqn: SBM enriched free energy HJ equation} should be the limit of the initial conditions $(\psi_N)$. Following \cite{PanNote}, we will first compute this limit for discrete measures $\mu \in \M_+$, and later show that the convergence may be extended to all measures in $\M_+$ through a density argument. Given a measure $\mu\in \M_+$, it will be convenient to write $\Pi_{\pm}(\mu)$ for the generalized Poisson point process with mean measure $(c\pm \Delta x)\ud \mu(x)$ on $[-1,1]$. By generalized Poisson process we mean that an atom $a\in [-1,1]$ with $\mu(\{a\})>0$ is counted with independent Poisson multiplicity having mean $(c\pm \Delta a)\mu(\{a\})$. It turns out that the limit of the initial conditions \eqref{eqn: SBM HJ initial condition} is given by an appropriate average with respect to the randomness of the generalized Poisson point processes $\Pi_\pm(\mu)$, which we denote by $\E$ as usual.  

\begin{lemma}\label{SBM convergence of initial condition discrete}
For any discrete measure $\mu \in \M_+$, the sequence $(\psi_N(\mu))$ converges to
\begin{align}\label{eqn: SBM asymptotic initial condition}
\psi(\mu)=-\mu[-1,1] c&+p\E \log \int_{\Sigma_1} \exp(-\mu[-1,1]\Delta \sigma \E x_{1})\prod_{x\in \Pi_+(\mu)}(c+\Delta \sigma x)\ud P^*(\sigma)\notag\\
&+(1-p)\E \log \int_{\Sigma_1} \exp(-\mu[-1,1]\Delta \sigma \E x_{1})\prod_{x\in \Pi_-(\mu)}(c+\Delta \sigma x)\ud P^*(\sigma),
\end{align}
where $x_{1}$ has law $\bar \mu$. 
\end{lemma}

\begin{remark}  
\label{r.calculation.derivative.psi}
From this lemma and its extension to any $\mu \in \M_+$ proved in Proposition~\ref{SBM convergence of initial condition} below, one can also show that for every $\mu \in \M_+$, the density of the Gateaux derivative $D_\mu \psi(\mu)$ is 
\begin{align}  
D_\mu \psi(\mu,x) = p\E \langle c+\Delta \sigma x\rangle_+ \log\langle c+\Delta &\sigma x\rangle_+ \notag\\
&+ (1-p)\E \langle c+\Delta \sigma x\rangle_- \log\langle c+\Delta \sigma x\rangle_- -c-\Delta \m x,
\end{align}
where $\langle \cdot \rangle_{\pm}$ denote the Gibbs averages given by
\begin{equation}
\langle f(\sigma)\rangle_{\pm} := \frac{\int_{\Sigma_1} f(\sigma)\exp(-\mu[-1,1]\Delta \sigma \E x_{1}) \prod_{x\in \Pi_{\pm}(\mu)} (c+\Delta \sigma x)\ud P^*(\sigma)}{\int_{\Sigma_1}\exp(-\mu[-1,1]\Delta \sigma \E x_{1}) \prod_{x\in \Pi_{\pm}(\mu)} (c+\Delta \sigma x)\ud P^*(\sigma)}.
\end{equation}
\end{remark}

\begin{proof}[Proof of Lemma~\ref{SBM convergence of initial condition discrete}]
Since $\mu\in\M_+$ is a discrete measure, it may be expressed as
$$\mu=\sum_{\ell\leq K}p_\ell\delta_{a_\ell}$$
for some integer $K\geq 1$, some atoms $a_\ell\in [-1,1]$ and some weights $p_\ell\geq 0$. Let $s=\mu[-1,1]$, and introduce independent Poisson random variables $\Pi_{i,s}\sim \Poi(sN)$ in such a way that
$$\psi_N(\mu)=\frac{1}{N}\sum_{i\leq N}\E \log \int_{\Sigma_1} \exp \sum_{k\leq \Pi_{i,s}}\log\bigg[(c+\Delta \sigma x_{i,k})^{\tG_{i,k}^x}\Big(1-\frac{c+\Delta \sigma x_{i,k}}{N}\Big)^{1-\tG_{i,k}^x}\bigg]\ud P^*(\sigma),$$
where $(x_{i,k})$ are i.i.d.\@ random variables with law $\bar \mu$. Since each of the expectations in this average is the same,
\begin{equation}\label{eqn: SBM asymptotic initial condition psi}
\psi_N(\mu)=\E \log \int_{\Sigma_1} \exp \sum_{k\leq \Pi_{1,s}}\log\bigg[(c+\Delta \sigma x_{k})^{\tG_{k}^x}\Big(1-\frac{c+\Delta \sigma x_{k}}{N}\Big)^{1-\tG_{k}^x}\bigg]\ud P^*(\sigma),
\end{equation}
where $(x_k)$ are i.i.d.\@ random variables with law $\bar \mu$ and $\tG_k^x$ has conditional distribution \eqref{eqn: SBM tG distribution} for $i=1$ and $x_{1,k}$ is replaced by $x_k$. To simplify this further, introduce the random index sets
$$\I_0=\big\{k\leq \Pi_{1,s}\mid \tG_{k}^x=0\big\} \quad \text{and} \quad \I_1=\big\{k\leq \Pi_{1,s}\mid \tG_{k}^x=1\big\}.$$
Decomposing the sum in \eqref{eqn: SBM asymptotic initial condition psi} according to the partition $\{k\leq \Pi_{1,s}\}=\I_0\sqcup \I_1$ and applying Taylor's theorem to the logarithm reveals that
$$\psi_N(\mu)=\E\log \int_{\Sigma_1}\prod_{k\in \I_1}(c+\Delta \sigma x_{k})\exp\Big(-\frac{\Delta \sigma}{N}\sum_{k\in \I_0}x_{k}\Big)\ud P^*(\sigma)-\Big(\frac{c}{N}+\BigO\big(N^{-2}\big)\Big)\E\lvert \I_0\rvert.$$
Conditionally on $\sigma^*$, $\Pi_{1,s}$ and $(x_{k})$, the random variable $\abs{\I_0}$ is a sum of Bernoulli random variables with probability of success $\smash{1-\frac{c+\Delta \sigma_1^*x_{k}}{N}}$. It therefore has mean
\begin{equation}\label{eqn: SBM asymptotic initial condition I0 mean}
\E\lvert \I_0\rvert=\E\Pi_{1,s}\E\Big(1-\frac{c+\Delta \sigma_1^*x_{1}}{N}\Big)=sN\Big(1-\frac{c+\Delta \m \E x_{1}}{N}\Big)
\end{equation}
and variance bounded by
\begin{equation}\label{eqn: SBM asymptotic initial condition I0 variance}
\Var\abs{\I_0}=\E\Pi_{1,s}\E\Big(1-\frac{c+\Delta \sigma_1^*x_{1}}{N}\Big)\Big(\frac{c+\Delta \sigma_1^*x_{1}}{N}\Big)\leq s(c+\abs{\Delta}).
\end{equation}
Using \eqref{eqn: SBM asymptotic initial condition I0 mean} and introducing the random index sets $\I_1^\ell = \{k\in \I_1\mid x_{k}=a_\ell\}$ reveals that
$$\psi_N(\mu)=\E\log \int_{\Sigma_1}\prod_{\ell\leq K}\prod_{k\in \I_1^\ell}(c+\Delta \sigma a_\ell)\exp\Big(-\frac{\Delta \sigma}{N}\sum_{k\in \I_0}x_{k}\Big)\ud P^*(\sigma)-cs+\BigO\big(N^{-1}\big).$$
Observe that for any $\sigma\in \Sigma_1$,
\begin{align*}
\E \Big\lvert -\frac{\Delta \sigma}{N}\sum_{k\in \I_0}x_{k}+\Delta \sigma s \E x_{1}\Big\rvert&\leq \abs{\Delta}s \E \Big\lvert \frac{1}{Ns}\sum_{k\in \I_0}x_{k}-\E x_{1}\Big\rvert\\
&\leq \abs{\Delta}s \E\Big\lvert \frac{1}{Ns}\sum_{k\leq Ns}x_{k}-\E x_{1}\Big\rvert+\frac{\Delta}{N}\big(\Var\abs{\I_0}+\lvert \E\abs{\I_0}-Ns\rvert\big)
\end{align*}
where we have used the fact that $\abs{x_{k}}\leq 1$ and Jensen's inequality in the second inequality. Recalling \eqref{eqn: SBM asymptotic initial condition I0 variance} and invoking the strong law of large numbers shows that
\begin{equation}\label{eqn: SBM asymptotic initial condition pre-limit}
\psi_N(\mu)=\E\log \int_{\Sigma_1}\exp(-\Delta \sigma s\E x_{1})\prod_{\ell\leq K}\prod_{k\in \I_1^\ell}(c+\Delta \sigma a_\ell)\ud P^*(\sigma)-cs+o_N(1).
\end{equation}
The Poisson coloring theorem (see Chapter 5 in \cite{Kingman}) implies that $\abs{\I_1^\ell}$ is a Poisson random variable with mean
$$\E\Pi_{1,s}\cdot \P\{\tG_{1}^x=1, x_{1}=a_\ell\}=sN\cdot \frac{c+\Delta \sigma_1^*a_\ell}{N}\cdot \bar \mu(a_\ell)=\big(c+\Delta \sigma_1^*a_\ell\big) \mu(a_\ell),$$
so averaging \eqref{eqn: SBM asymptotic initial condition pre-limit} over the randomness of $\sigma^*$ yields
\begin{align*}
\psi_N(\mu)=-&cs+p\E\log \int_{\Sigma_1}\exp(-\Delta \sigma s\E x_{1})\prod_{x\in \Pi_+(\mu)}(c+\Delta \sigma x)\ud P^*(\sigma)\\
&+(1-p)\E\log \int_{\Sigma_1}\exp(-\Delta \sigma s\E x_{1})\prod_{x\in \Pi_-(\mu)}(c+\Delta \sigma x)\ud P^*(\sigma)+o_N(1).
\end{align*}
This completes the proof.
\end{proof}

To extend this convergence to all measures in $\M_+$, we will rely upon the continuity of the functional \eqref{eqn: SBM asymptotic initial condition} with respect to the Wasserstein distance on the space of probability measures,
\begin{align}
W(\P,\Q)&=\sup\bigg\{\Big\lvert \int_{-1}^1 h(x)\ud \P(x)-\int_{-1}^1 h(x)\ud \Q(x)\Big\rvert \mid \norm{h}_{\text{Lip}}\leq 1\bigg\} \label{eqn: SBM Wasserstein}\\
&=\inf\bigg\{\int_{[-1,1]^2}\abs{x-y}\ud\nu(x,y) \mid \nu\in \Pr\big([-1,1]^2\big) \text{ has marginals } \P \text{ and } \Q\bigg\}.\label{eqn: SBM Wasserstein inf}
\end{align}
Here $\norm{\cdot}_{\mathrm{Lip}}$ denotes the Lipschitz semi-norm
\begin{equation}
\norm{h}_{\mathrm{Lip}}=\sup_{x\neq x'\in [-1,1]}\frac{\abs{h(x)-h(x')}}{\abs{x-x'}}
\end{equation}
defined on the space of functions $h:[-1,1]\to \R$. This continuity will be obtained as a consequence of the following uniform bound on the spatial derivatives of the Gateaux derivative density \eqref{eqn: SBM Gateaux derivative density}.

\begin{lemma}\label{SBM derivatives of Gateux density bound}
For every $N$ large enough (relative to $c$), $\mu\in \M_+$, $t\geq 0$ and $x\in [-1,1]$,
\begin{align}
\big\lvert D_\mu \overline{F}_N(t,\mu,x)\big\rvert&\leq 2c\big(2+\abs{\log(2c)}+\abs{\log(c-\abs{\Delta})}\big),\label{eqn: SBM Gateux density bound}\\
\big\lvert \partial_x D_\mu \overline{F}_N(t,\mu,x)\big\rvert&\leq c\big(1+\abs{\log(2c)}+\abs{\log(c-\abs{\Delta})}\big).\label{eqn: SBM derivative of Gateux density bound}
\end{align}
\end{lemma}

\begin{proof}
Recall from \eqref{eqn: SBM Gateaux derivative density} that 
\begin{align*}
D_\mu \overline{F}_N(t,\mu,x)=\E\big(c+\Delta \langle \sigma_i\rangle x\big)\log\big(c&+\Delta\langle \sigma_i\rangle x\big)\notag\\
&+N\E\Big(1-\frac{c+\Delta \langle \sigma_i\rangle x}{N}\Big)\log\Big(1-\frac{c+\Delta\langle \sigma_i\rangle x}{N}\Big).
\end{align*}
It follows by a direct computation that
$$\partial_x D_\mu \overline{F}_N(t,\mu,x)=\Delta \E \langle \sigma_i\rangle\log\big(c+\Delta\langle \sigma_i\rangle x\big)-\Delta \E\langle \sigma_i\rangle \log\Big(1-\frac{c+\Delta\langle \sigma_i\rangle x}{N}\Big).$$
Since all spin configuration coordinates are bounded by one and $\abs{\Delta}<c$, Taylor's theorem implies that for $N$ large enough,
\begin{align*}
\big\lvert D_\mu \overline{F}_N(t,\mu,x)\big\rvert&\leq 2c\big(2+\abs{\log(2c)}+\abs{\log(c-\abs{\Delta})}\big),\\
\big\lvert \partial_x D_\mu \overline{F}_N(t,\mu,x)\big\rvert&\leq c\big(1+\abs{\log(2c)}+\abs{\log(c-\abs{\Delta}\big)}\big).
\end{align*}
Notice that the choice of $N$ only depends on $c$ as $x\in [-1,1]$ and $\abs{\Delta}< c$. This completes the proof.
\end{proof}

\begin{lemma}\label{SBM initial condition W Lipschitz}
The initial condition $\psi_N$ satisfies the Lipschitz bound
\begin{equation}
\abs{\psi_N(\P)-\psi_N(\Q)}\leq c\big(1+\abs{\log(2c)}+\abs{\log(c-\abs{\Delta})}\big)\big)W(\P,\Q)
\end{equation}
for all probability measures $\P,\Q\in \Pr[-1,1]$.
\end{lemma}
\begin{proof}
The fundamental theorem of calculus and the definition of the Gateaux derivative in \eqref{eqn: SBME Gateaux derivative} imply that
$$\psi_N(\P)-\psi_N(\Q)=\int_0^1 \frac{\mathrm{d}}{\mathrm{d} t}\psi_N\big(\Q+t(\P-\Q)\big)\ud t=\int_0^1 D_\mu\psi_N \big(\Q+t(\P-\Q); \P-\Q\big)\ud t.$$
Since the Gateaux derivative of the initial condition admits a continuously differentiable density, 
$$\abs{\psi_N(\P)-\psi_N(\Q)}\leq \int_0^1\Big\lvert  \int_{-1}^1 f_t(x)\ud \P(x)-\int_{-1}^1 f_t(x)\ud \Q(x)\Big\rvert\ud t$$
for the continuously differentiable function $f_t(x)=D_\mu\psi_N(\Q+t(\P-\Q),x)$. The mean value theorem and \eqref{eqn: SBM derivative of Gateux density bound} reveal that $\smash{\norm{f_t}_{\text{Lip}}\leq c\big(1+\abs{\log(2c)}+\abs{\log(c-\abs{\Delta})}\big)}$. It follows by definition of the Wasserstein distance \eqref{eqn: SBM Wasserstein} that
$$\abs{\psi_N(\mu)-\psi_N(\nu)}\leq  c\big(1+\abs{\log(2c)}+\abs{\log(c-\abs{\Delta})}\big)W(\P,\Q).$$
This completes the proof.
\end{proof}

\begin{lemma}\label{SBM asymptotic initial condition continuous}
The functional $\psi:\M_+\to \R$ defined by \eqref{eqn: SBM asymptotic initial condition} is continuous with respect to the weak convergence of measures. This means that for any sequence of measures $(\mu_n)\subset \M_+$ converging weakly to a measure $\mu\in \M_+$, we have
\begin{equation}
\lim_{N\to \infty}\psi(\mu_n)=\psi(\mu).
\end{equation}
\end{lemma}

\begin{proof}
To alleviate the exposition, we will instead prove the continuity of the functional
$$\psi^1(\mu)=\E\log\int_{\Sigma_1}\exp(-\mu[-1,1]\Delta \sigma \E x_1)\prod_{x\in \Pi_+(\mu)}(c+\Delta \sigma x)\ud P^*(\sigma)$$
with respect to the weak convergence of measures. Up to an additive constant, the asymptotic initial condition $\psi(\mu)$ is the weighted average of $\psi^1(\mu)$ and another functional of the same form whose continuity can be established using an identical argument, so this suffices. For each measure $\mu\in \M_+$ introduce the Hamiltonian
$$H(\sigma,\mu)=-\Delta \sigma \int_{-1}^1 x\ud \mu(x)+\sum_{x\in \Pi_+(\mu)}(c+\Delta\sigma x)$$
in such a way that the asymptotic initial condition is its associated free energy,
$$\psi^1(\mu)=\E\log\int_{\Sigma_1}\exp H(\sigma,\mu)\ud P^*(\sigma).$$
Consider a sequence of measures $(\mu_n)\subset \M_+$ converging weakly to a measure $\mu\in \M_+$, and let $\Pi_n$ and $\Pi$ be independent Poisson random variables with means $\mu_n[-1,1]$ and $\mu[-1,1]$, respectively. Introduce a collection $\smash{(X_k^n,X_k)_{k\in \N}}$ of i.i.d.\@ random vectors with joint law $\smash{\nu\in \Pr\big([-1,1]^2\big)}$ having marginals $\smash{\bar{\mu}_n}$ and $\smash{\bar{\mu}}$. In this way, the coordinates $\smash{(X_k^n)_{k\in \N}}$ are i.i.d.\@ with law $\smash{\bar{\mu}_n}$, the coordinates $\smash{(X_k)_{k\in \N}}$ are i.i.d.\@ with law $\smash{\bar \mu}$, and we have the equalities in distribution
$$\sum_{x\in \Pi_+(\mu_n)}(c+\Delta \sigma x)\stackrel{d}{=}\sum_{k\leq \Pi_n}\big(c+\Delta \sigma X_k^n\big) \quad \text{and} \quad \sum_{x\in \Pi_+(\mu)}(c+\Delta x)\stackrel{d}{=}\sum_{k\leq \Pi}\big(c+\Delta \sigma X_k\big).$$
It follows that for any $\sigma\in \Sigma_1$,
$$\abs{H(\sigma,\mu_n)-H(\sigma,\mu)}\leq \abs{\Delta}\Big\lvert \int_{-1}^1 x\ud \mu_n(x)-\int_{-1}^1 x\ud \mu(x)\Big\rvert+c\big\lvert \Pi_n-\Pi\big\rvert+\abs{\Delta}\Big\lvert \sum_{k\leq \Pi_n}X_k^n-\sum_{k\leq \Pi}X_k\Big\rvert,$$
and therefore,
$$\big\lvert \psi^1(\mu_n)-\psi^1(\mu)\big\rvert\leq \abs{\Delta}\Big\lvert \int_{-1}^1 x\ud \mu_n(x)-\int_{-1}^1 x\ud \mu(x)\Big\rvert+c\E\big\lvert \Pi_n-\Pi\big\rvert+\abs{\Delta}\E\Big\lvert \sum_{k\leq \Pi_n}X_k^n-\sum_{k\leq \Pi}X_k\Big\rvert.$$
To simplify this further, define the random variable $\Pi'_n=\min(\Pi_n,\Pi)$, introduce a Poisson random variable $\smash{\Pi''_n}$ independent of all other sources of randomness with mean $\smash{\abs{\mu_n[-1,1]-\mu[-1,1]}}$, and define the collection of random variables $\smash{(Z_k^n)}$ by
$$Z_k^n=\begin{cases}
X_k^n& \text{if } \Pi_n'=\Pi\\
X_k& \text{otherwise}
\end{cases}.$$
The basic properties of Poisson random variables and the fact that $\smash{\abs{Z_k^n}\leq 1}$ imply that
\begin{align*}
\big\lvert \psi^1(\mu_n)-\psi^1(\mu)\big\rvert&\leq c\abs{\Delta}\Big\lvert \int_{-1}^1 x\ud(\mu_n-\mu)(x)\Big\rvert+c\E\Pi_n''+\abs{\Delta}\E\sum_{k\leq \Pi_n''}\big\lvert Z_k^n\big\rvert+\abs{\Delta}\E\sum_{k\leq \Pi_n'}\big\lvert X_k^n-X_k\big\rvert\\
&\leq c\bigg(\Big\lvert \int_{-1}^1 x\ud(\mu_n-\mu)(x)\Big\rvert+2\E\Pi_n''+\E\Pi_n'\int_{[-1,1]^2}\abs{x-y}\ud \nu(x,y)\rvert\bigg).
\end{align*}
Taking the infimum over all couplings $\smash{\nu\in \Pr\big([-1,1]^2\big)}$ with marginals $\smash{\bar{\mu}_n}$ and $\smash{\bar \mu}$ reveals that for $n$ large enough,
$$\big\lvert \psi^1(\mu_n)-\psi^1(\mu)\big\rvert\leq c\bigg(\Big\lvert \int_{-1}^1 x\ud(\mu_n-\mu)(x)\Big\rvert+2\big\lvert\mu_n[-1,1]-\mu[-1,1]\big\rvert+3\mu[-1,1]W\big(\bar{\mu}_n,\bar \mu\big)\bigg),$$
where we have used that $\E \Pi_n'\leq \E \Pi+\E\Pi_n\leq 3\mu[-1,1]$ for $n$ large enough as $\mu_n$ converges weakly to $\mu$. Letting $n$ tend to infinity and recalling that the Wasserstein distance \eqref{eqn: SBM Wasserstein inf} metrizes the weak convergence of probability measures completes the proof.
\end{proof}

\begin{proposition}\label{SBM convergence of initial condition}
For any measure $\mu \in \M_+$, the sequence $(\psi_N(\mu))$ converges to \eqref{eqn: SBM asymptotic initial condition}.
\end{proposition}
\begin{proof}
Consider a sequence $(\mu_n)_{n\geq 1}$ of discrete measures such that $\mu_n[-1,1]=\mu[-1,1]$ for all $n\geq 1$ and $\bar \mu_n\to \bar \mu$ with respect to the Wasserstein distance \eqref{eqn: SBM Wasserstein}. By the triangle inequality and \Cref{SBM initial condition W Lipschitz},
\begin{align*}
\big\lvert \psi(\mu)-\psi_N(\mu)\big\rvert&\leq \big\lvert \psi(\mu)-\psi(\mu_n)\big\rvert+\big\lvert \psi(\mu_n)-\psi_N(\mu_n)\big\rvert+\big\lvert \psi_N(\mu_n)-\psi_N(\mu)\big\rvert\notag\\
&\leq \big\lvert \psi(\mu)-\psi(\mu_n)\big\rvert+\big\lvert \psi(\mu_n)-\psi_N(\mu_n)\big\rvert\\
&\qquad\qquad\qquad\quad\,\,\,\,\,\,\,+c\big(1+\abs{\log(2c)}+\abs{\log(c-\abs{\Delta})}\big)W\big(\bar{\mu}_n,\bar \mu\big)
\end{align*}
where we have used the fact that $\mu_n[-1,1]=\mu[-1,1]$. Combining \Cref{SBM asymptotic initial condition continuous} with \Cref{SBM convergence of initial condition discrete} to let $N\to \infty$ and then $n\to \infty$ completes the proof.
\end{proof}

This result identifies the initial condition for the infinite-dimensional Hamilton-Jacobi equation \eqref{eqn: SBM enriched free energy HJ equation}.
We may now leverage the main result in \cite{TD_JC_HJ} to establish the well-posedness of this infinite-dimensional Hamilton-Jacobi equation. Before we do this, let us introduce the notation used in \cite{TD_JC_HJ}; for more details and motivation regarding this notation, we encourage the reader to consult \cite{TD_JC_HJ}.
Given an integer $K\geq 1$, we write
\begin{equation}\label{eqn: SBME dyadic rationals}
\D_K=\Big\{k=\frac{i}{2^K} \mid -2^K\leq i<  2^K\Big\}
\end{equation}
for the set of dyadic rationals on $[-1,1]$ at scale $K$. We index vectors using the set of dyadic rationals, writing
$\smash{x=(x_k)_{k\in \D_K}\in \Rp^{\D_K}}$ for a non-negative sequence with values in $\Rp$ indexed by the dyadic rationals $\D_K$. More generally, given two sets $\mathbb{A}$ and $\mathcal{B}$, we write $\mathbb{A}^\mathcal{B}$ for the set of functions from $\mathcal{B}$ to $\mathbb{A}$.
We denote the set of discrete measures supported on the dyadic rationals at scale $K$ in the interval $[-1,1]$ by
\begin{equation}
\M^{(K)}_+=\Big\{\mu\in \M_+\mid \mu=\frac{1}{\abs{\D_K}}\sum_{k\in \D_K}x_k\delta_k\text{ for some } x=(x_k)_{k\in \D_K}\in \Rp^{\D_K}\Big\},
\end{equation}
and we project a general measure $\mu \in \M_+$ onto $\M_+^{(K)}$ via the mapping
\begin{equation}\label{eqn: SBME weights of dyadic measure}
x^{(K)}(\mu)=\big(\abs{\D_K}\mu\big[k,k+2^{-K}\big)\big)_{k\in \D_K}
\in \Rp^{\D_K},
\end{equation}
whose inverse assigns to each $\smash{x \in \Rp^{\D_K}}$ the measure
\begin{equation}\label{eqn: SBME measure for given vector}
\mu^{(K)}_x=\frac{1}{\abs{\D_K}}\sum_{k\in \D_K}x_{k}\delta_k \in \M_+^{(K)}.
\end{equation}
We identify any real-valued function $\smash{f:\Rp\times \M_+^{(K)}\to \R}$ with the function 
\begin{equation}
f^{(K)}(t,x)=f\big(t,\mu_x^{(K)}\big)
\end{equation}
defined on $\smash{\Rp\times \Rp^{\D_K}}$, and we identify the projection of the initial condition with the function
\begin{equation}\label{eqn: SBM projected modified initial condition}
\psi^{(K)}(x)=\psi\big(\mu_x^{(K)}\big)
\end{equation}
defined on $\smash{\Rp^{\D_K}}$. The Gateaux derivative of a real-valued function $\smash{f:\Rp\times \M_+^{(K)}\to \R}$ at the measure $\smash{\mu\in \M_+^{(K)}}$ becomes the gradient $\smash{\abs{\D_K}\nabla f^{(K)}(t,x^{(K)}(\mu))}$ by duality. Indeed,
\begin{equation}\label{eqn: SBM Gateux derivative projection}
D_\mu f(t,\mu; \nu)=\frac{\mathrm{d}}{\mathrm{d} \epsilon}\Big|_{\epsilon=0}f^{(K)}(t,x^{(K)}(\mu)+\epsilon x^{(K)}(\nu))=\nabla f^{(K)}(t,x^{(K)}(\mu))\cdot x^{(K)}(\nu)
\end{equation}
for any direction $\smash{\nu\in \M_+^{(K)}}$. We fix $b>0$ large enough so the modified kernel
\begin{equation}
\label{e.def.tdg}
\td g_b (z) = g(z) + b
\end{equation}
is strictly positive, and we introduce the symmetric matrices
\begin{equation}\label{eqn: SBM matrix tGK}
G^{(K)}=\frac{1}{\abs{\D_K}^2}\big(g(kk')\big)_{k,k'\in \D_K} \quad \text{and} \quad \td G_b^{(K)}=\frac{1}{\abs{\D_K}^2}\big(\td g_b(kk')\big)_{k,k'\in \D_K}
\end{equation}
in $\smash{\R^{\D_K\times \D_K}}$. We also define the projected cone
\begin{equation}
\widetilde{\CC}_{b,K}=\Big\{\widetilde{G}_b^{(K)}x\in \R^{\D_K}\mid x\in \Rp^{\D_K}\Big\},
\end{equation}
the projected non-linearity $\widetilde{\C}_{b,K}:\widetilde{\CC}_{b,K}\to \R$ given by
\begin{equation}\label{eqn: projected non-linearity b}
\widetilde{\C}_{b,K}\big(\widetilde{G}^{(K)}_bx\big)=\frac{1}{2}\widetilde{G}_b^{(K)}x\cdot x=\frac{1}{2\abs{\D_K}^2}\sum_{k,k'\in \D_K}\td g_b(kk')x_kx_{k'},
\end{equation}
and the closed convex set
\begin{equation}
\K_{=1,K}=\Big\{G^{(K)}x\in \R^{\D_K}\mid x\in \Rp^{\D_K}\text{ and } \Norm{x}_1= a\Big\}.
\end{equation}
We will measure quantities in the normalized $\ell^1$ and $\ell^{1,*}$ norms,
\begin{equation}\label{eqn: SBME normalized norms}
\Norm{x}_1=\frac{1}{\abs{\D_K}}\sum_{k\in \D_K}\abs{x_{k}} \quad \text{and} \quad \dNorm{1}{y}=\max_{k\in \D_K}\abs{\D_K}\abs{y_{k}},
\end{equation}
and it will be convenient to write \begin{equation}
B_{K,R}=\big\{y\in \R^{\D_K}\mid \dNorm{1}{y}\leq R\big\}
\end{equation}
for the ball of radius $R>0$ centered at the origin with respect to the normalized-$\ell^{1,*}$ norm in $\smash{\R^{\D_K}}$. Recall that a function $\smash{h:\R^d\to \R}$ is said to be non-decreasing if $h(y)\leq h(y')$ for all $\smash{y,y'\in \R^d}$ with $\smash{y'-y\in \Rp^d}$. Proposition 2.3 in \cite{TD_JC_HJ} will give the existence of a uniformly Lipschitz continuous and non-decreasing non-linearity $\smash{\td \H_{b,K,R}}$ which agrees with $\smash{\td \C_{b,K}}$ on the ball $\smash{\td \CC_{b,K}\cap B_{K,R}}$, and we will define the solution to the infinite-dimensional Hamilton-Jacobi equation \eqref{eqn: SBM enriched free energy HJ equation} by
\begin{equation}\label{eqn: SBM inf HJ eqn solution}
f(t,\mu)=\lim_{K\to \infty}\bigg(\td f_{b,R}^{(K)}\big(t,x^{(K)}(\mu)\big)-b\Norm{x^{(K)}(\mu)}_1-\frac{bt}{2}\bigg),
\end{equation}
where $\td f^{(K)}_{b,R}:[0,\infty)\times \Rp^{\D_K}\to \R$ is the unique solution to the Hamilton-Jacobi equation
\begin{equation}\label{eqn: SBM projected HJ eqn b}
\partial_t \td f^{(K)}(t,x)=\widetilde{\H}_{b,K, R}\big(\nabla \td f^{(K)}(t,x)\big) \quad\text{on}\quad \Rpp\times \Rpp^{\D_K}
\end{equation}
subject to the projection of the initial condition $\td \psi_{b}:\M_+\to \R$ defined by
\begin{equation}\label{e.def.tdpsi}
\td \psi_{b}(\mu)=\psi(\mu)+b\int_{-1}^1 \ud \mu.
\end{equation}
The fact that the function $f$ defined by \eqref{eqn: SBM inf HJ eqn solution} does not depend on the choice of the constants $b \in \R$ and $R > 0$ sufficiently large is established in Proposition \ref{SBM modified HJ eqn WP}, while the existence and uniqueness of the appropriate notion of solution to the Hamilton-Jacobi equation \eqref{eqn: SBM projected HJ eqn b} is established in Proposition \ref{SBM WP of HJ eqn on Rpp b}. In order to prove these propositions, we will rely on some results from \cite{TD_JC_HJ}. To appeal to these results, we must start by verifying that the initial condition $\psi$ satisfies a certain number of hypotheses introduced there. To state these hypotheses, given a closed convex set $\smash{\K\subset \R^{\D_K}}$, write
\begin{equation}\label{eqn: SBM enlarged set}
\K'=\K+B_{K,2^{-K/2}}
\end{equation}
for the neighbourhood of radius $\smash{2^{-K/2}}$ around $\K$ in the normalized-$\ell^{1,*}$ norm, and denote by
\begin{equation}\label{eqn: SBM TV metric}
\TV(\mu,\nu)=\sup\big\{\abs{\mu(A)-\nu(A)}\mid A \text{ is a measurable subset of } [-1,1]\big\}
\end{equation}
the total variation distance on $\M_+$. A non-differential criterion for the gradient of a Lipschitz continuous function to lie in a closed convex set is given in Proposition B.2 of \cite{TD_JC_HJ}.

\begin{enumerate}[label = \textbf{H\arabic*}]
\item The initial condition $\psi:\M_+\to \R$ is Lipschitz continuous with respect to the total variation distance \eqref{eqn: SBM TV metric},
\begin{equation}
\abs{\psi(\mu)-\psi(\nu)}\leq \norm{\psi}_{\mathrm{Lip},\TV}\TV(\mu,\nu)
\end{equation}
for all measures $\nu,\mu\in \M_+$. \label{SBM H TV}
\item 
\label{SBM H a}
The initial condition $\smash{\psi:\M_+\to \R}$ has the property that each of the projected initial conditions has its gradient in the set $\K_{=1,K}'$,
\begin{equation}
\nabla \psi^{(K)}\in L^\infty\big(\Rp^d;\K_{=1,K}'\big).
\end{equation}
\item The initial condition $\psi:\Pr[-1,1]\to \R$ is Lipschitz continuous with respect to the Wasserstein distance \eqref{eqn: SBM Wasserstein},
\begin{equation}
\abs{\psi(\P)-\psi(\Q)}\leq \norm{\psi}_{\mathrm{Lip},W} W(\P,\Q)
\end{equation}
for all probability measures $\P,\Q\in \Pr[-1,1]$. \label{SBM H W}
\end{enumerate}

\begin{lemma}\label{SBM assumptions for WP of HJ}
The initial condition $\psi$ in \eqref{eqn: SBM asymptotic initial condition} satisfies \eqref{SBM H TV}-\eqref{SBM H W}.
\end{lemma}

\begin{proof}
Recall that \Cref{SBM derivatives of Gateux density bound} implies the existence of a constant $C>0$ which depends only on $c$ and $\Delta$ such that for every integer $N\geq 1$, $\mu \in \M_+$ and $t\geq 0$,
\begin{equation}\label{eqn: SBM Gateaux derivative bounds modified IC} 
\big\lvert D_\mu \overline{F}_N(t, \mu,x)\big\rvert\leq C\quad \text{and} \quad \big\lvert \partial_x D_\mu \overline{F}_N(t,\mu,x)\big\rvert\leq C.
\end{equation}
To establish \eqref{SBM H TV} notice that for every integer $N\geq 1$ and $\mu,\nu\in \M_+$,
$$\psi_N(\mu)-\psi_N(\nu)=\int_0^1 D_\mu\psi_N \big(\nu+t(\mu-\nu); \mu-\nu\big)\ud t=\int_0^1 \int_{-1}^1f_t(x)\ud \big( \mu-\nu\big)(x)\ud t$$
for the continuously differentiable function $\smash{f_t(x)=D_\mu\psi_N \big(\nu+t(\mu-\nu),x\big)}$. To bound this integral by the total variation distance, let $\eta=\mu-\nu\in \M_s$, and use the Hahn-Jordan decomposition to write $\smash{\eta=\eta^+-\eta^-}$ for measures $\smash{\eta^+,\eta^-\in \M_+}$ with the property that for some measurable set $D\subset[-1,1]$ and all measurable sets $E\subset [-1,1]$,
$$\eta^+(E)=\eta(E\cap D)\geq 0\quad \text{and} \quad \eta^-(E)=-\eta(E\cap D^c)\geq 0.$$
The triangle inequality and the first bound in \eqref{eqn: SBM Gateaux derivative bounds modified IC} imply that
\begin{align*}
\abs{\psi_N(\mu)-\psi_N(\nu)}&\leq \Big\lvert \int_0^1 \int_{-1}^1 f_t(x)\ud \eta^+(x)\Big\rvert+\Big\lvert\int_0^1 \int_{-1}^1 f_t(x)\ud \eta^-(x)\Big\rvert\\
&\leq C\big(\eta^+[-1,1]+\eta^-[-1,1]\big)\\
&\leq 2C\TV(\mu,\nu).
\end{align*}
Using \Cref{SBM convergence of initial condition} to let $N$ tend to infinity establishes \eqref{SBM H TV}. To prove \eqref{SBM H a} notice that by \eqref{eqn: SBM Gateux derivative projection} and \eqref{eqn: SBM enriched free energy approximate Gateaux derivative}, for every $\smash{y\in \Rp^{\D_K}}$, there exists some probability measure $\mu^*\in \Pr[-1,1]$ with
$$\partial_{x_k}\psi_N^{(K)}(y)=\frac{1}{\abs{\D_K}}D_\mu\psi_N\big(\mu_y^{(K)},k\big)=\frac{1}{\abs{\D_K}}G_{\mu^*}(k)+\BigO\big(N^{-1}\big).$$
If $\mu^*_K=\mu^{(K)}_{x^{(K)}(\mu^*)}\in \M_+^{(K)}$ denotes the projection of $\mu^*$ onto $\M_+^{(K)}$, then the mean value theorem implies that
$$\abs{G_{\mu^*}(k)-G_{\mu^*_K}(k)}\leq \sum_{k'\in \D_K}\int_{k'}^{k'+2^{-K}}\abs{g(ky)-g(kk')}\ud \mu^*(y)\leq \frac{\norm{g'}_\infty}{2^K},$$
where we have used that $\mu^*_K(k')=\mu^*[k',k'+2^{-K})$ for every dyadic $k'\in \D_K$. This means that
$$\partial_{x_k}\psi_N^{(K)}(y)=\frac{1}{\abs{\D_K}}G_{\mu^*_K}(k)+\BigO\big(2^{-2K}\big)+\BigO\big(N^{-1}\big)=G^{(K)}x^{(K)}(\mu^*_K)_k+\BigO\big(2^{-2K}\big)+\BigO\big(N^{-1}\big)$$
so, for $K$ large enough, we have $\smash{\nabla \psi^{(K)}_N(y)=w+\BigO\big(N^{-1}\big)}$ for some $\smash{w\in \K'_{=1,K}}$. At this point fix $c\in \R$ and $\smash{x,x'\in \R^d}$ with  $(x'-x)\cdot z\geq c$ for every $\smash{z\in \K_{=1,K}'}$. The fundamental theorem of calculus reveals that
$$\psi_N^{(K)}(x')-\psi_N^{(K)}(x)=\int_0^1 \nabla \psi_N^{(K)}\big(tx'+(1-t)x\big)\cdot (x'-x)\ud t\geq c+ \BigO\big(N^{-1}\big).$$
Using \Cref{SBM convergence of initial condition} to let $N$ tend to infinity shows that $\smash{\psi^{(K)}(x')-\psi^{(K)}(x)\geq c}$ and Proposition B.2 in \cite{TD_JC_HJ} gives \eqref{SBM H a}. Finally, \eqref{SBM H W} is a consequence of \Cref{SBM initial condition W Lipschitz} and \Cref{SBM convergence of initial condition}. This completes the proof.
\end{proof}

This result allows us to invoke Proposition 2.3 in \cite{TD_JC_HJ} to extend the non-linearity $\smash{\td \C_{b,K}}$ in \eqref{eqn: projected non-linearity b}. It will be convenient to write $M_b=\max_{[-1,1]}\td g_b$ and $m_b=\min_{[-1,1]}\td g_b>0$.

\begin{proposition}\label{SBM extending the non-linearity}
For every $R>0$, there exists a non-decreasing non-linearity $\smash{\widetilde{\H}_{b,K,R}:\R^{\D_K}\to \R}$ which agrees with $\smash{\widetilde{\C}_{b,K}}$ on $\smash{\widetilde{\CC}_{b,K}\cap B_{K,R}}$ and satisfies the Lipschitz continuity property
\begin{equation}\label{eqn: SBM extending the non-linearity}
\big\lvert \widetilde{\H}_{b,K,R}(y)-\widetilde{\H}_{b,K,R}(y')\big\rvert\leq \frac{8RM_b}{m_b^2}\dNorm{1}{y-y'}
\end{equation}
for all $y,y'\in \R^{\D_K}$.
\end{proposition}

The well-posedness of the Hamilton-Jacobi equation \eqref{eqn: SBM projected HJ eqn b} is the content of Theorem 1.1 in~\cite{TD_JC_HJ}. Before stating this result, let us remind the reader of the notion of a viscosity solution and introduce some more notation.

\begin{definition}
An upper semi-continuous function $\smash{u:[0,\infty)\times  \Rp^{\D_K}\to \R}$ is said to be a viscosity subsolution to \eqref{eqn: SBM projected HJ eqn b} if, given any $\smash{\phi\in C^\infty\big((0,\infty)\times  \Rpp^{\D_K}\big)}$ with the property that $u-\phi$ has a local maximum at $\smash{(t^*,x^*)\in (0,\infty)\times \Rpp^{\D_K}}$,
\begin{equation}\label{eqn: SBM subsolution conditions}
\big(\partial_t \phi-\td \H_{b,K,R}(\nabla \phi)\big)(t^*,x^*)\leq 0.
\end{equation}
\end{definition}

\begin{definition}
A lower semi-continuous function $\smash{v:[0,\infty)\times \Rp^{\D_K}\to \R}$ is said to be a viscosity supersolution to \eqref{eqn: SBM projected HJ eqn b} if, given any $\smash{\phi\in C^\infty\big((0,\infty)\times  \Rpp^{\D_K}\big)}$ with the property that $v-\phi$ has a local minimum at $\smash{(t^*,x^*)\in (0,\infty)\times \Rpp^{\D_K}}$,
\begin{equation}\label{eqn: SBM supersolution conditions}
\big(\partial_t \phi-\td \H_{b,K,R}(\nabla \phi)\big)(t^*,x^*)\geq 0.
\end{equation}
\end{definition}

\begin{definition}\label{SBM definition of viscosity solution}
A continuous function $\smash{f\in C\big([0,\infty)\times \Rp^{\D_K}\big)}$ is said to be a viscosity solution to \eqref{eqn: SBM projected HJ eqn b} if it is both a viscosity subsolution and a viscosity supersolution to \eqref{eqn: SBM projected HJ eqn b}.
\end{definition}

\noindent Given functions $h:\Rp^{\D_K}\to \R$ and $\smash{u:[0,\infty)\times \Rp^{\D_K}\to \R}$, define the semi-norms
\begin{equation}\label{eqn: SBM projected modified non-linearity Lipschitz}
\Norm{h}_{\mathrm{Lip},1}=\sup_{x\neq x'\in\Rp^{\D_K}}\frac{\abs{h(x)-h(x')}}{\Norm{x-x'}_1} \quad \text{and} \quad [u]_{0}=\sup_{\substack{t>0\\ x\in \Rp^{\D_K}}}\frac{\abs{u(t,x)-u(0,x)}}{t}.
\end{equation}
Introduce the space of functions with Lipschitz initial condition that grow at most linearly in time,
\begin{equation}
\mathfrak{L}=\big\{u:[0,\infty)\times \Rp^{\D_K}\to \R\mid u(0,\cdot) \text{ is Lipschitz continuous and } [u]_0<\infty\big\},
\end{equation}
and its subset of uniformly Lipschitz functions,
\begin{equation}
\mathfrak{L}_{\mathrm{unif}}=\Big\{u\in \mathfrak{L}\mid \sup_{t\geq 0}\Norm{u(t,\cdot)}_{\mathrm{Lip},1}<\infty\Big\}.
\end{equation}

Combining Theorem 1.1 with the arguments leading to the second conclusion of Lemma 6.1 in \cite{TD_JC_HJ} gives the following well-posedness result for the Hamilton-Jacobi equation \eqref{eqn: SBM projected HJ eqn b}.

\begin{proposition}\label{SBM WP of HJ eqn on Rpp b}
For every $R>0$, the Hamilton-Jacobi equation \eqref{eqn: SBM projected HJ eqn b} admits  a unique viscosity solution $\smash{\td f^{(K)}_{b,R}\in \mathfrak{L}_{\mathrm{unif}}}$ subject to the initial condition $\smash{\td \psi_b^{(K)}}$. Moreover, the solution has its gradient in the closed convex set $\smash{\td \K_{=1,K}'}$, 
\begin{equation}
\nabla \td f^{(K)}_{b,R} \in L^\infty\big([0,\infty)\times \Rp^{\D_K}; \td \K_{=1,K}'\big),
\end{equation}
and it satisfies the Lipschitz bound
\begin{equation}
\sup_{t>0}\Norm{\td f_{b,R}^{(K)}(t,\cdot)}_{\mathrm{Lip},1}=\Norm{\td \psi_b^{(K)}}_{\mathrm{Lip},1}\leq \norm{\td \psi_b}_{\mathrm{Lip},\TV}.
\end{equation}
\end{proposition}

\noindent In addition to this existence and uniqueness result, it will be important to record the following comparison principle for the projected Hamilton-Jacobi equation \eqref{eqn: SBM projected HJ eqn b}. This comparison principle is a consequence of Corollary A.12 in \cite{TD_JC_HJ}.

\begin{lemma}\label{SBM comparison principle corollary on Rpd}
If $\smash{u,v\in \mathfrak{L}_{\mathrm{unif}}}$ are respectively a continuous subsolution and a continuous supersolution to \eqref{eqn: SBM projected HJ eqn b}, then
\begin{equation}
\sup_{\Rp\times \Rp^{\D_K}}\big(u(t,x)-v(t,x)\big)=\sup_{\Rp^{\D_K}}\big(u(0,x)-v(0,x)\big).
\end{equation}
\end{lemma}

\noindent The existence of the limit \eqref{eqn: SBM inf HJ eqn solution} defining $f(t,\mu)$ is a consequence of Theorems 1.2 and 1.4 in \cite{TD_JC_HJ}.

\begin{proposition}\label{SBM modified HJ eqn WP}
Let $b \in \R$ be such that the kernel $\smash{\td g_b}$ defined in \eqref{e.def.tdg} is positive on $[-1,1]$, let $\smash{\td \psi_{b}}$ be defined by \eqref{e.def.tdpsi}, and for each integer $K\geq 1$ and $\smash{R>\norm{\td \psi_b}_{\mathrm{Lip},\TV}}$, denote by $\smash{\td f_{b,R}^{(K)}}$ the unique viscosity solution to the Hamilton-Jacobi equation \eqref{eqn: SBM projected HJ eqn b} subject to the initial condition $\smash{\td \psi_b^{(K)}}$. The limit \eqref{eqn: SBM inf HJ eqn solution} exists, is finite and is independent of $R$ and $b$. This limit is defined to be the solution to the infinite-dimensional Hamilton-Jacobi equation~\eqref{eqn: SBM enriched free energy HJ equation}. 
\end{proposition}

To establish \Cref{SBM main result}, the idea will be to show that the enriched free energy \eqref{eqn: SBM enriched free energy} is essentially a viscosity subsolution to the infinite-dimensional Hamilton-Jacobi equation \eqref{eqn: SBM enriched free energy HJ equation}.

\section{The free energy upper bound}\label{sec: SBM upper bound}

In this section, we combine the computations in \Cref{sec: SBM HJ eqn derivation} with the arguments in \cite{JC_NC} to essentially show that the enriched free energy \eqref{eqn: SBM enriched free energy} is a viscosity subsolution to the infinite-dimensional Hamilton-Jacobi equation \eqref{eqn: SBM enriched free energy HJ equation}. This equation is given a precise meaning by \Cref{SBM modified HJ eqn WP}. As alluded to in \Cref{sec: SBM HJ eqn derivation}, we will perturb the enriched Hamiltonian \eqref{eqn: SBM enriched Hamiltonian} to enforce the concentration of all the multi-overlaps \eqref{eqn: SBM enriched multi-overlap} without changing the limit of the associated free energy. The perturbation Hamiltonian we will now describe was introduced in \cite{BarMOC} to prove a general multi-overlap concentration result whose finitary version we establish in \Cref{SBM app multioverlap concentration}.

Fix an integer $K_+$ which will be chosen sufficiently large in the course of this section, and write $\lambda=(\lambda_0,\lambda_1,\ldots,\lambda_{K_+})$ for a perturbation parameter with $\lambda_0\in [1/2,1]$ and $\lambda_k\in [2^{-k-1},2^{-k}]$ for $1\leq k\leq K_+$. Given a sequence $(\epsilon_N)$ with $\epsilon_N=N^\gamma$ for some $-1/8<\gamma<0$ and a standard Gaussian vector $Z_0=(Z_{0,1},\ldots,Z_{0,N})$ in $\R^N$, introduce the Gaussian perturbation Hamiltonian
\begin{equation}\label{eqn: SBM gaussian perturbation}
H_N^{\mathrm{gauss}}(\sigma,\lambda_0)=\HH_0=\sum_{i\leq N}\big(\lambda_0\epsilon_N \sigma^*_i\sigma_i+\sqrt{\lambda_0\epsilon_N}Z_{0,i}\sigma_i\big)
\end{equation}
associated with the task of recovering the signal $\sigma^*$ from the data
\begin{equation}
Y^{\mathrm{gauss}}=\sqrt{\lambda_0\epsilon_N}\sigma^*+Z_0.
\end{equation}
Notice that
\begin{equation}\label{eqn: SBM epsilon sequence}
1\geq \epsilon_N \to 0 \quad \text{and} \quad N\epsilon_N\to \infty.
\end{equation}
Similarly, consider a sequence $(s_N)$ with $s_N=N^\eta$ for $4/5<\eta<1$ in such a way that
\begin{equation}\label{eqn: SBM s sequence}
\frac{s_N}{N}\to 0\quad \text{and} \quad \frac{s_N}{\sqrt{N}}\to \infty.
\end{equation}
Fix a sequence of i.i.d.\@ random variables $(\pi_k)$ with $\Poi(s_N)$ distribution as well as a sequence $e=(e_{jk})$ of random variables with $\Exp(1)$ distribution.
For every $j\leq \pi_k$, sample i.i.d.\@ random indices $i_{jk}$ uniformly from the set $\{1,\ldots,N\}$, and define the exponential perturbation Hamiltonian by
\begin{equation}\label{eqn: SBM exponential perturbation}
\HH_k=\sum_{j\leq \pi_k}\Big(\log(1+\lambda_k\sigma_{i_{jk}}\big)-\frac{\lambda_k e_{jk}\sigma_{i_{jk}}}{1+\lambda_k\sigma_{i_{jk}}^*}\Big) \quad \text{and}\quad H_N^{\mathrm{exp}}(\sigma)=\sum_{1\leq k\leq K_+} \HH_k.
\end{equation}
Observe that this is the Hamiltonian associated with the task of recovering the signal $\sigma^*$ from the independently generated data
\begin{equation}
Y_{jk}^{\mathrm{exp}}=\frac{e_{jk}}{1+\lambda_k\sigma^*_{i_{jk}}}
\end{equation}
for $j\leq \pi_k$ and $k\geq 1$. Introduce the perturbed Hamiltonian
\begin{equation}\label{eqn: SBM perturbed Hamiltonian}
H_N(\sigma,\lambda)=H_N^{t,\mu}(\sigma)+H_N^{\mathrm{gauss}}(\sigma,\lambda_0)+H_N^{\mathrm{exp}}(\sigma,\lambda)
\end{equation}
as well as its associated free energy
\begin{equation}\label{eqn: SBM perturbed free energy}
\overline{F}_N^{\text{pert}}(t,\mu,\lambda)=\frac{1}{N}\E\log\int_{\Sigma_N}\exp H_N(\sigma,\lambda)\ud P_N^*(\sigma).
\end{equation}
Since the Gibbs measure associated with the Hamiltonian in \eqref{eqn: SBM perturbed Hamiltonian} is still a conditional expectation as in \eqref{eqn: SBM Gibbs conditional}, it will still satisfy the Nishimori identity \eqref{eqn: SBM Nishimori identity}. An essential property of the perturbation Hamiltonians \eqref{eqn: SBM gaussian perturbation} and \eqref{eqn: SBM exponential perturbation} is that they do not affect the asymptotic behavior of the enriched free energy \eqref{eqn: SBM enriched free energy}.

\begin{lemma}\label{SBM perturbation effect on free energy}
For every $t>0$, $\mu \in \M_+$ and $\lambda$, the enriched free energy \eqref{eqn: SBM enriched free energy} and the perturbed free energy \eqref{eqn: SBM perturbed free energy} are asymptotically equivalent,
\begin{equation}
\lim_{N\to \infty}\big\lvert \overline{F}_N^{\mathrm{pert}}(t,\mu,\lambda)-\overline{F}_N(t,\mu)\big\rvert=0.
\end{equation}
\end{lemma}

\begin{proof}
A direct computation reveals that
$$\big\lvert \overline{F}_N^{\mathrm{pert}}(t,\mu,\lambda)-\overline{F}_N(t,\mu)\big\rvert\leq \frac{1}{N}\E \max_{\sigma \in \Sigma_N}\big\lvert H_N^{\mathrm{gauss}}(\sigma,\lambda_0)\big\rvert+\frac{1}{N}\E \max_{\sigma \in \Sigma_N}\big\lvert H_N^{\mathrm{exp}}(\sigma,\lambda)\big\rvert.$$
For any spin configuration $\sigma\in \Sigma_N$,
$$\big\lvert H_N^{\mathrm{gauss}}(\sigma,\lambda_0)\big\rvert\leq N\epsilon_N+\sqrt{\epsilon_N}\sum_{i\leq N}\abs{Z_{0,i}}$$
while
$$\big\lvert H_N^{\mathrm{exp}}(\sigma,\lambda)\big\rvert\leq \sum_{1\leq k\leq K'}\sum_{j\leq \pi_k}\Big(\log(1+\lambda_k)+\frac{\lambda_ke_{jk}}{1-\lambda_k}\Big).$$
Since these bounds are uniform in $\sigma$, it follows that
$$\big\lvert \overline{F}_N^{\mathrm{pert}}(t,\mu,\lambda)-\overline{F}_N(t,\mu)\big\rvert\leq \epsilon_N+\sqrt{\epsilon_N}\E Z_{0,1}+\frac{s_N}{N}\sum_{k\geq 1}\Big(\log(1+\lambda_k)+\frac{\lambda_k}{1-\lambda_k}\Big).$$
The third term was obtained by taking the expectation with respect to the randomness of $e$ first and then with respect to the randomness of $(\pi_k)$. Leveraging \eqref{eqn: SBM epsilon sequence} and \eqref{eqn: SBM s sequence} to let $N$ tend to infinity completes the proof.
\end{proof}

\noindent With this result in mind, we abuse notation and redefine the perturbed free energy \eqref{eqn: SBM perturbed free energy},
\begin{equation}
F_N(t,\mu,\lambda)=\frac{1}{N}\log\int_{\Sigma_N}\exp H_N(\sigma,\lambda)\ud P_N^*(\sigma) \quad \text{and} \quad \overline{F}_N(t,\mu,\lambda)=\E F_N(t,\mu,\lambda).
\end{equation}
For every integer $K\geq 1$ and $x\in \Rp^{\D_K}$, we denote by
\begin{equation}\label{eqn: SBM projected perturbed free energy}
F_N^{(K)}(t,x,\lambda)=F_N\big(t,\mu_x^{(K)},\lambda\big) \quad \text{and} \quad \overline{F}_N^{(K)}(t,x,\lambda)=\E F_N^{(K)}(t,x,\lambda)
\end{equation}
the finite-dimensional projections of these perturbed free energy functionals. In the same spirit as \eqref{eqn: SBM inf HJ eqn solution}, given $b\in \R$ such that the kernel $\td g_b$ defined in \eqref{e.def.tdg} is positive on $[-1,1]$, introduce translated versions of these free energy functionals,
\begin{equation}\label{eqn: SBM modified free energy}
F_N'(t,\mu, \lambda)=F_N(t,\mu,\lambda)+b\int_{-1}^1 \ud \mu+\frac{bt}{2} \quad \text{and} \quad \overline{F}_N'(t,\mu,\lambda)=\E F_N'(t,\mu,\lambda).
\end{equation}
For every integer $K\geq 1$ and $x\in \Rp^{\D_K}$, we denote by
\begin{equation}\label{eqn: SBM projected modified free energy}
F_N'^{(K)}(t,x,\lambda)=F_N'\big(t,\mu_x^{(K)},\lambda\big) \quad \text{and} \quad \overline{F}_N'^{(K)}(t,x,\lambda)=\E F_N'^{(K)}(t,x,\lambda)
\end{equation}
the finite-dimensional projections of these modified free energy functionals. Similarly, we write
\begin{equation}\label{eqn: SBM projected free energy}
\overline{F}_N^{(K)}(t,x)=\overline{F}_N\big(t,\mu_x^{(K)}\big) \quad \text{and} \quad \overline{F}_N'^{(K)}(t,x)=\overline{F}_N^{(K)}(t,x)+b\Norm{x}_1+\frac{bt}{2}
\end{equation}
for the finite-dimensional projections of the enriched free energy \eqref{eqn: SBM modified free energy} and its translation according to \eqref{eqn: SBM inf HJ eqn solution}. Combining Lemmas~\ref{SBM enriched free energy time derivative}, \ref{SBM enriched free energy Gateaux derivative} and \ref{SBM convergence of initial condition discrete} with the Arzela-Ascoli theorem, it is possible to extract a subsequential limit $\smash{\td F^{(K)}}$ from the sequence defined by \eqref{eqn: SBM projected free energy} for varying $N$. Passing to a further subsequence and using a diagonalization argument, it is also possible to ensure that
\begin{equation}\label{eqn: SBM modified FE limsup}
\td F^{(K)}(t,x)=\limsup_{N\to \infty}\overline{F}_N'^{(K)}(t,x)
\end{equation}
for all $\smash{(t,x)\in (0,\infty)\times \Rp^{\D_K}}$. The key to establishing \Cref{SBM main result} will be to show that, in some sense, the subsequential limit $\smash{\td F^{(K)}}$ is an approximate subsolution to the Hamilton-Jacobi equation \eqref{eqn: SBM projected HJ eqn b} for some $\smash{R>\norm{\td \psi_b}_{\mathrm{Lip},\TV}+\norm{\td g_b}_\infty+\norm{\td g_b'}_\infty+1}$ which will remain fixed throughout this section. 

We fix a smooth function $\smash{\phi \in C^\infty\big((0,\infty)\times \Rpp^{\D_K}\big)}$ with the property that the difference $\smash{\td F^{(K)}-\phi}$ achieves a local maximum at some point $\smash{(t_\infty, x_\infty)\in (0,\infty)\times \Rpp^{\D_K}}$. Recalling that the index $K_+$ controls the number of terms in the perturbation Hamiltonian \eqref{eqn: SBM exponential perturbation}, we introduce the parameter
\begin{equation}\label{eqn: SBM asymptotic perturbation parameter}
\lambda_\infty=\frac{\big(1, 2^{-1}, 2^{-2},\ldots,2^{-K_+}\big)+\big(2^{-1},2^{-2},2^{-3},\ldots,2^{-K_+-1}\big)}{2}
\end{equation}
as well as the smooth function
\begin{equation}
\td \phi(t,x,\lambda)=\phi(t,x)+(t-t_\infty)^2+\norm{x-x_\infty}_2^2+\norm{\lambda-\lambda_\infty}_2^2.
\end{equation}
It is clear that $\smash{(t,x,\lambda)\mapsto \td F^{(K)}(t,x)-\td \phi(t,x,\lambda)}$ has a strict local maximum at $(t_\infty,x_\infty,\lambda_\infty)$. Arguing as in the proof of \Cref{SBM perturbation effect on free energy} shows that $\smash{(t,x,\lambda)\mapsto \overline{F}_N'^{(K)}(t,x,\lambda)}$ converges to $\smash{(t,x,\lambda)\mapsto \td F^{(K)}(t,x)}$ locally uniformly. It is therefore possible to find a sequence $(t_N,x_N,\lambda_N)$ which converges to the point $(t_\infty,x_\infty,\lambda_\infty)$ and has the property that $\smash{(t,x,\lambda)\mapsto \overline{F}_N'^{(K)}(t,x,\lambda)-\td \phi(t,x,\lambda)}$ attains a local maximum at $(t_N,x_N,\lambda_N)$. More precisely, it is possible to find a constant $C<\infty$ which is allowed to depend on $K$, $K_+$, $t_\infty$, $x_\infty$ and the function $\phi$ such that
\begin{align}\label{eqn: SBM free energy subsolution maximum}
\Big(\overline{F}_N'^{(K)}-\td \phi\Big)&(t_N,x_N,\lambda_N)\notag\\
&=\sup\Big\{\Big(\overline{F}_N'^{(K)}-\td \phi\Big)(t,x,\lambda)\mid \abs{t-t_N}+\norm{x-x_N}_2+\norm{\lambda-\lambda_N}_2\leq C^{-1}\Big\}.
\end{align} 
We will use such a constant $C < \infty$ at various places in this proof, and we understand that its value may need to be increased as we proceed; the important point is that it does not depend on $N$.
The choices of $\lambda_\infty$ in \eqref{eqn: SBM asymptotic perturbation parameter} and $\smash{x_\infty\in \Rpp^{\D_K}}$ ensure that when $N$ is large enough $(\lambda_N)_k\in (2^{-k-1},2^{-k})$ for $0\leq k\leq K_+$ and $\smash{x_N\in \Rpp^{\D_K}}$. Increasing $C<\infty$ if necessary, it is therefore possible to guarantee that for $N$ large enough the supremum on the right-hand side of \eqref{eqn: SBM free energy subsolution maximum} is taken over triples $(t,x,\lambda)$ with $t>0$, $\smash{x\in \Rpp^{\D_K}}$ and $\lambda_k\in [2^{-k-1},2^{-k}]$ for $1\leq k\leq K_+$. It follows that
\begin{equation}\label{eqn: SBM free energy subsolution spacetime derivative}
\partial_t\Big(\overline{F}_N'^{(K)}-\td \phi\Big)(t_N,x_N,\lambda_N)=0, \qquad
\nabla_x\Big(\overline{F}_N'^{(K)}-\td \phi\Big)(t_N,x_N,\lambda_N)=0
\end{equation}
and
\begin{equation}\label{eqn: SBM free energy subsolution perturbation derivative}
\nabla_\lambda\Big(\overline{F}_N^{(K)}-\td\phi\Big)(t_N,x_N,\lambda_N)=\nabla_\lambda\Big(\overline{F}_N'^{(K)}-\td \phi\Big)(t_N,x_N,\lambda_N)=0.
\end{equation}
The majority of this section will be devoted to using the second equality in \eqref{eqn: SBM free energy subsolution perturbation derivative} in conjunction with the main result in \Cref{SBM app multioverlap concentration} to show the concentration of a finite but very large number of the multi-overlaps \eqref{eqn: SBM enriched multi-overlap}. We will then combine this finitary multi-overlap concentration result with the computations in \Cref{sec: SBM HJ eqn derivation} to establish the following crucial result.

\begin{lemma}\label{SBM free energy approximate solution at contact point}
Fix $\smash{R>\norm{\td \psi_b}_{\mathrm{Lip},\TV}+\norm{\td g_b}_\infty+\norm{\td g_b'}_\infty+1}$. For every $\epsilon>0$, there exists a choice of integer $K_+\geq 1$ in the perturbed Hamiltonian \eqref{eqn: SBM perturbed Hamiltonian} with the property that for any integer $K\geq 1$, it is possible to find a constant $\smash{\EE_{\epsilon,K}}$ with
\begin{equation}
\limsup_{N\to \infty}\Big\lvert \Big(\partial_t\overline{F}_N'^{(K)}-\widetilde{\H}_{b,K,R}\Big(\nabla_x \overline{F}_N'^{(K)}\Big)\Big)(t_N,x_N,\lambda_N)\Big\rvert\leq \EE_{\epsilon,K}
\end{equation}
and $\lim_{\epsilon\to 0}\lim_{K\to \infty}\EE_{\epsilon,K}=0$.
\end{lemma}

\noindent For the time being, let us prove \Cref{SBM main result} assuming \Cref{SBM free energy approximate solution at contact point}.

\begin{proof}[Proof of \Cref{SBM main result} assuming \Cref{SBM free energy approximate solution at contact point}.]
Given $\epsilon>0$, invoke \Cref{SBM free energy approximate solution at contact point} to find an integer $K_+\geq 1$ in the perturbed Hamiltonian \eqref{eqn: SBM perturbed Hamiltonian} with the property that for any integer $K\geq 1$, it is possible to find a constant $\smash{\EE_{\epsilon,K}}$ with
\begin{equation}\label{eqn: SBM free energy approximate solution at contact point}
\limsup_{N\to \infty}\Big\lvert \Big(\partial_t\overline{F}_N'^{(K)}-\widetilde{\H}_{b,K,R}\Big(\nabla_x \overline{F}_N'^{(K)}\Big)\Big)(t_N,x_N,\lambda_N)\Big\rvert\leq \EE_{\epsilon,K}
\end{equation}
and $\lim_{\epsilon\to 0}\lim_{K\to \infty}\EE_{\epsilon,K}=0$. Given an integer $K\geq 1$, the idea will be to show that the test function $\smash{\phi\in C^\infty\big((0,\infty)\times \Rpp^{\D_K}\big)}$ introduced above satisfies the subsolution condition in \eqref{eqn: SBM subsolution conditions} for the non-linearity $\smash{\widetilde{\H}_{b,K,R}}$ at the point of contact $(t_\infty,x_\infty)$ up to the small error $\EE_{\epsilon,K}$. This will mean that the subsequential limit $\smash{\td F^{(K)}}$ of the modified free energy \eqref{eqn: SBM projected free energy} is a viscosity subsolution to the Hamilton-Jacobi equation \eqref{eqn: SBM projected HJ eqn b} up to a small error. More precisely, the function
\begin{equation}\label{eqn: SBM main result proof corrected FE}
\td F^{(K)}_\epsilon=\td F^{(K)}-t\EE_{\epsilon,K}
\end{equation}
will be a viscosity subsolution to \eqref{eqn: SBM projected HJ eqn b}. This observation will allow us to leverage the comparison principle in \Cref{SBM comparison principle corollary on Rpd} to bound the limit superior of the enriched free energy \eqref{eqn: SBM enriched free energy} by the solution $f$ to the infinite-dimensional Hamilton-Jacobi equation \eqref{eqn: SBM enriched free energy HJ equation} constructed in \Cref{SBM modified HJ eqn WP}. We proceed in two steps.\\
\step{1: $\td F_\epsilon^{(K)}$ subsolution.}\\
Since $\smash{x_N\to x_\infty}$ assume without loss of generality that $\smash{(x_N)\subset \Rpp^{\D_K}}$. It follows by \eqref{eqn: SBM free energy subsolution spacetime derivative} that
$$\big(\partial_t\td \phi-\widetilde{\H}_{b,K,R}\big(\nabla_x\td \phi\big)\big)(t_N,x_N,\lambda_N)=\Big(\partial_t\overline{F}_N'^{(K)}-\widetilde{\H}_{b,K,R}\Big(\nabla_x \overline{F}_N'^{(K)}\Big)\Big)(t_N,x_N,\lambda_N),$$
so letting $N$ tend to infinity and combining the definition of $\td \phi$ with \eqref{eqn: SBM free energy approximate solution at contact point} yields
$$\big(\partial_t\td \phi-\widetilde{\H}_{b,K,R}\big(\nabla_x\td \phi\big)\big)(t_\infty, x_\infty)\leq \EE_{\epsilon,K}.$$
This shows that the function \eqref{eqn: SBM main result proof corrected FE} satisfies the subsolution condition in \eqref{eqn: SBM subsolution conditions}.\\
\step{2: comparison principle.}\\
The comparison principle in \Cref{SBM comparison principle corollary on Rpd} gives the upper bound
\begin{equation}\label{eqn: SBM main result comparison principle}
\td F^{(K)}(t,x)\leq \td f^{(K)}_{b,R}(t,x)+t\EE_{\epsilon,K}.
\end{equation}
We have implicitly used that $\smash{\td F^{(K)}}$ and $\smash{\td f_{b,R}^{(K)}}$ are continuous and have the same initial condition by \Cref{SBM convergence of initial condition}. We have also used that they both belong to the solution space $\smash{\mathfrak{L}_{\mathrm{unif}}}$ by \Cref{SBM WP of HJ eqn on Rpp b}, \eqref{SBM enriched free energy time derivative}, \eqref{SBM enriched free energy Gateaux derivative} and a simple application of the mean value theorem. With \eqref{eqn: SBM main result comparison principle} in mind, fix a finite measure $\mu\in \M_+$, and introduce the discrete measure $$\mu^{(K)}=\mu^{(K)}_{x^{(K)}(\mu)}$$
defined in \eqref{eqn: SBME measure for given vector}.
It is readily verified that $\smash{\bar{\mu}^{(K)}\to \bar{\mu}}$ with respect to the Wasserstein distance \eqref{eqn: SBM Wasserstein}. Moreover, an identical argument to that in \Cref{SBM initial condition W Lipschitz} leveraging the second bound in \eqref{eqn: SBM derivative of Gateux density bound} reveals that
$${\overline{F}_N}(t,\mu)\leq C'\mu[-1,1]W\big(\bar \mu,\bar\mu^{(K)}\big)+\overline{F}_N^{(K)}\big(t,x^{(K)}(\mu)\big)$$
for some constant $C'$ that depends only on $c$. We use the letter $C'$ instead of $C$ to emphasize that the constant $C'$ does not depend on $K$.
Letting $N$ tend to infinity, recalling \eqref{eqn: SBM projected free energy},  \eqref{eqn: SBM modified FE limsup} and leveraging \eqref{eqn: SBM main result comparison principle} yields
$$\limsup_{N\to \infty}\overline{F}_N(t,\mu)\leq C'\mu[-1,1]W\big(\bar \mu,\bar\mu^{(K)}\big)+\td f_{b,R}^{(K)}\big(t,x^{(K)}(\mu)\big)-b\Norm{x^{(K)}(\mu)}_1-\frac{bt}{2}+t\EE_{\epsilon,K}.$$
Invoking \Cref{SBM free energy approximate solution at contact point} and \Cref{SBM modified HJ eqn WP} to let $K$ tend to infinity and then $\epsilon$ tend to zero completes the proof.
\end{proof}

The rest of this section is devoted to the proof of \Cref{SBM free energy approximate solution at contact point} which will be obtained by combining the computations of \Cref{sec: SBM HJ eqn derivation} with the main result of \Cref{SBM app multioverlap concentration} to show the concentration of a finite but very large number of the multi-overlaps \eqref{eqn: SBM enriched multi-overlap}. In the notation of \cite{BarMOC}, for any perturbation parameter $\lambda$, let
\begin{equation}
\lambda_{0,N}=\epsilon_N\lambda_0,
\end{equation}
and introduce the quantities
\begin{align}
\L_0&=\frac{\HH_0'}{N\epsilon_N}, \quad \text{ where } \quad  \HH_0'=\partial_{\lambda_0} H_N^{\mathrm{gauss}}(\sigma,\lambda_0)=\epsilon_N\bigg(\sigma\cdot \sigma^*+\frac{\sigma\cdot Z_0}{2\sqrt{\lambda_{0,N}}}\bigg),\label{eqn: SBM lambda 0 derivative of Hamiltonian}\\
\L_k&=\frac{\HH_k'}{s_N}, \quad \text{ where } \quad  \HH_k'=\partial_{\lambda_k} H_N^{\mathrm{exp}}(\sigma,\lambda)=\sum_{j\leq \pi_k}\sigma_{i_{jk}}\bigg(\frac{1}{1+\lambda_k\sigma_{i_{jk}}}-\frac{e_{jk}}{(1+\lambda_k\sigma^*_{i_{jk}})^2}\bigg) \label{eqn: SBM lambda k derivative of Hamiltonian}
\end{align}
for $1\leq k\leq K_+$. The importance of these quantities stems from the fact that
\begin{align}
\partial_{\lambda_0} F_N^{(K)}(t,x,\lambda)=\frac{1}{N}\big\langle \HH_0'\big\rangle,&\,\,\,\,\,\,\, \partial^2_{\lambda_0} F_N^{(K)}(t,x,\lambda)=\frac{1}{N}\Big(\big\langle (\HH_0'-\langle \HH_0'\rangle)^2\big\rangle-\frac{\epsilon_N^2}{4\lambda_{0,N}^{3/2}}\langle \sigma\rangle \cdot Z_0\Big),\label{eqn: SBM lambda 0 derivatives}\\
\partial_{\lambda_k} F_N^{(K)}(t,x,\lambda)=\frac{1}{N}\big\langle \HH_k'\big\rangle, &\,\,\,\,\,\,\, \partial^2_{\lambda_k} F_N^{(K)}(t,x,\lambda)=\frac{1}{N}\big(\big\langle (\HH_k'-\langle \HH_k'\rangle)^2\big\rangle+\big\langle \HH_k''\big\rangle\big)\label{eqn: SBM lambda k derivatives}
\end{align}
for $1\leq k\leq K_+$ while
\begin{equation}\label{eqn: SBM lambda kj derivative}
\partial_{\lambda_k\lambda_j} F_N^{(K)}(t,x,\lambda)=\frac{1}{N}\big(\langle \HH'_j\HH'_k\rangle-\langle \HH'_j\rangle \langle \HH'_k\rangle\big)
\end{equation}
for $0\leq j\neq k\leq K_+$. 
Here, and for the remainder of this section, the Gibbs average $\langle \cdot \rangle$ will always be associated with the perturbed Hamiltonian \eqref{eqn: SBM perturbed Hamiltonian} evaluated at a triple $(t,x,\lambda)$ which will be clear from the context. It will also be convenient to record that for $1\leq k\leq K_+$,
\begin{equation}\label{eqn: SBM Hk second derivative}
\HH_k''=\partial_{\lambda_k}^2 H_N^{\mathrm{exp}}(\sigma,\lambda)=\sum_{j\leq \pi_k}\bigg(-\frac{1}{(1+\lambda_k\sigma_{i_{jk}})^2}+2\frac{\sigma_{i_{jk}}\sigma^*_{i_{jk}}e_{jk}}{(1+\lambda_k\sigma^*_{i_{jk}})^3}\bigg) \text{ and }  \big\lvert \E \big\langle \HH_k''\big\rangle\big\rvert\leq Cs_N.
\end{equation}

To obtain the concentration of the multi-overlaps \eqref{eqn: SBM enriched multi-overlap} we will show the concentration \eqref{eqn: SBMA limits to get FDS} of the quantities $\L_k$ for the Gibbs measure with parameters given by the contact point $(t_N,x_N,\lambda_N)$. This concentration will be deduced from the fact that the averaged free energy is being ``touched from above'' by a smooth function at the contact point, thereby constraining its Hessian at this point, together with the concentration of the free energy $\smash{F_N^{(K)}}$ about its average $\smash{\overline{F}_N^{(K)}}$. Due to the constraint on the Hessian at the contact point, we will be able to extend the concentration result on the free energy into an estimate on the concentration of its gradient. We decompose this argument into a series of four lemmas: the first two essentially bound the Hessian of the perturbed free energy \eqref{eqn: SBM projected perturbed free energy} from above and from below; the third leverages the free energy concentration result in \Cref{SBM app FE concentration} to estimate the uniform $L^p$-distance between the quenched and averaged free energies \eqref{eqn: SBM projected perturbed free energy}; while the fourth extends this to a control on the gradient of the free energy.

\begin{lemma}\label{SBM perturbed free energy Hessian bounded above}
For any perturbation parameter $\lambda$ with $\norm{\lambda}_2\leq C^{-1}$,
\begin{equation}\label{eqn: SBM perturbed free energy Hessian bounded above}
\overline{F}_N^{(K)}(t_N,x_N,\lambda_N+\lambda)-\overline{F}_N^{(K)}(t_N,x_N,\lambda_N)-\lambda\cdot \nabla_{\lambda} \overline{F}_N^{(K)}(t_N,x_N,\lambda_N)\leq C\norm{\lambda}_2^2.
\end{equation}
\end{lemma}

\begin{proof}
Fix a perturbation parameter $\lambda$ with $\norm{\lambda}_2\leq C^{-1}$, and notice that \eqref{eqn: SBM free energy subsolution maximum} gives
$$\overline{F}_N'^{(K)}(t_N,x_N,\lambda_N+\lambda)-\overline{F}_N'^{(K)}(t_N,x_N,\lambda_N)\leq \td \phi(t_N,x_N,\lambda_N+\lambda)-\td \phi(t_N,x_N,\lambda_N).$$
On the other hand, Taylor's formula with integral remainder implies that
\begin{align}\label{eqn: SBM perturbed free energy Hessian bounded above key} 
\overline{F}_N'^{(K)}(t_N,&x_N,\lambda_N+\lambda)-\overline{F}_N'^{(K)}(t_N,x_N,\lambda_N) \notag\\
&=\lambda\cdot \nabla_{\lambda} \overline{F}_N'^{(K)}(t_N,x_N,\lambda_N)+\int_0^1 (1-s)\lambda\cdot \nabla_{\lambda}^2\overline{F}_N'^{(K)}(t_N,x_N,\lambda_N+s\lambda)\lambda \ud s, 
\end{align}
and similarly,
\begin{align*}
\td \phi(t_N,x_N,\lambda_N&+\lambda)-\td \phi(t_N,x_N,\lambda_N)\\
&=\lambda\cdot \nabla_\lambda \td \phi(t_N,x_N,\lambda_N) +\int_0^1 (1-s)\lambda\cdot \nabla_\lambda^2\td \phi(t_N,x_N,\lambda_N+s\lambda)\lambda \ud s.
\end{align*}
Combining \eqref{eqn: SBM free energy subsolution spacetime derivative} with the chain rule shows that $\lambda\cdot \nabla_\lambda \widetilde{\phi}(t_N,x_N,\lambda_N)=\lambda\cdot \nabla_{\lambda}\overline{F}_N'^{(K)}(t_N,x_N,\lambda_N)$, and therefore
\begin{align*}
\int_0^1 (1-s)\lambda\cdot \nabla_{\lambda}^2\overline{F}_N'^{(K)}(t_N,x_N,\lambda_N+s\lambda)\lambda \ud s&\leq \int_0^1 (1-s)\lambda\cdot \nabla_\lambda^2\td \phi(t_N,x_N,\lambda_N+s\lambda)\lambda \ud s\\
&\leq C\norm{\lambda}_2^2.
\end{align*}
Substituting this into \eqref{eqn: SBM perturbed free energy Hessian bounded above key} gives
\begin{align*}
\overline{F}_N^{(K)}(t_N,x_N,\lambda_N+\lambda)-\overline{F}_N^{(K)}(t_N,x_N,\lambda_N)&=\overline{F}_N'^{(K)}(t_N,x_N,\lambda_N+\lambda)-\overline{F}_N'^{(K)}(t_N,x_N,\lambda_N)\\
&\leq \lambda\cdot \nabla_{\lambda} \overline{F}_N'^{(K)}(t_N,x_N,\lambda_N)+C\norm{\lambda}_2^2\\
&=\lambda\cdot \nabla_{\lambda} \overline{F}_N^{(K)}(t_N,x_N,\lambda_N)+C\norm{\lambda}_2^2.
\end{align*}
This completes the proof.
\end{proof}

\begin{lemma}\label{SBM perturbed free energy Hessian bounded below}
There exists a random variable $X$ with $\E X^2\leq C$ such that, for all perturbation parameters $\lambda$ with $\norm{\lambda}_2\leq C^{-1}$,
\begin{equation}\label{eqn: SBM perturbed free energy Hessian bounded below}
F_N^{(K)}(t_N,x_N,\lambda_N+\lambda)-F_N^{(K)}(t_N,x_N,\lambda_N)-\lambda\cdot \nabla_{\lambda} F_N^{(K)}(t_N,x_N,\lambda_N)\geq -X\norm{\lambda}_2^2.
\end{equation}
\end{lemma}

\begin{proof}
Since $t_N$ and $x_N$ remain fixed throughout, write $\smash{F_N^{(K)}(\lambda)}$ for $\smash{F_N^{(K)}(t_N,x_N,\lambda)}$. Introduce the function
\begin{equation}\label{eqn: SBM perturbed free energy Hessian bounded below convexification}
h(\lambda)=F_N^{(K)}(\lambda)-\frac{\sqrt{\lambda_{0,N}}}{N}\sum_{i\leq N}\abs{Z_{0,i}}+\frac{1}{N}\sum_{1\leq k\leq K_+}\sum_{j\leq \pi_k}\big(8\lambda_k^2e_{jk}-\log(1-\lambda_k)\big).
\end{equation}
Leveraging \eqref{eqn: SBM lambda 0 derivatives} and Hölder's inequality, one can see that
$$\partial_{\lambda_0}^2h(\lambda)=\frac{1}{N}\big\langle (\HH_0'-\langle \HH_0'\rangle)^2\big\rangle-\frac{\epsilon_N^2}{4N\lambda_{0,N}^{3/2}}\langle \sigma\rangle \cdot Z_0+\frac{\epsilon_N^2}{4N\lambda_{0,N}^{3/2}}\sum_{i\leq N}\abs{Z_{0,i}}\geq \frac{1}{N}\big\langle (\HH_0'-\langle \HH_0'\rangle)^2\big\rangle.$$
Using \eqref{eqn: SBM lambda k derivatives} and \eqref{eqn: SBM Hk second derivative} reveals that for $1\leq k\leq K_+$,
\begin{align*}
\partial_{\lambda_k}^2h(\lambda)=\frac{1}{N}\big\langle &(\HH_k'-\langle \HH_k'\rangle)^2\big\rangle\\
&+\frac{1}{N}\bigg\langle \sum_{j\leq \pi_k}\Big(-\frac{1}{(1+\lambda_k\sigma_{i_{jk}})^2}+2\frac{\sigma_{i_{jk}}\sigma^*_{i_{jk}}e_{jk}}{(1+\lambda_k\sigma^*_{i_{jk}})^3}+16e_{jk}+\frac{1}{(1-\lambda_k)^2}\Big)\bigg\rangle. 
\end{align*}
Since $\lambda_k\leq 1/2$ and all spin configuration coordinates are bounded by one, it is actually the case that
$$\partial_{\lambda_k}^2h(\lambda)\geq \frac{1}{N}\big\langle (\HH_k'-\langle \HH_k'\rangle)^2\big\rangle.$$
Together with \eqref{eqn: SBM lambda kj derivative}, this shows that $\nabla_\lambda^2h$ is positive definite and therefore $h$ is convex. It follows that for any perturbation parameter $\lambda$ with $\norm{\lambda}_2\leq C^{-1}$,
$$h(\lambda_N+\lambda)\geq h(\lambda_N)+\lambda \cdot\nabla_{\lambda} h(\lambda_N).$$
Remembering the definition of $h$ in \eqref{eqn: SBM perturbed free energy Hessian bounded below convexification}, this shows that the left-hand side of \eqref{eqn: SBM perturbed free energy Hessian bounded below} is bounded from below by
\begin{align*}
\frac{1}{N}& \bigg(\sqrt{(\lambda_N)_{0,N}+\lambda_{0,N}}-\sqrt{(\lambda_N)_{0,N}}-\frac{\lambda_{0,N}}{2\sqrt{(\lambda_N)_{0,N}}}\bigg)\sum_{i\leq N}\abs{Z_{0,i}}\\
&\qquad\qquad+\frac{1}{N}\sum_{1\leq k\leq K_+}\sum_{j\leq \pi_k}8e_{jk}\Big((\lambda_N)_k^2+2\lambda_k(\lambda_N)_k-\big((\lambda_N)_k+\lambda_k\big)^2\Big)\\
&\qquad\qquad\qquad\qquad +\frac{1}{N}\sum_{1\leq k\leq K_+}\sum_{j\leq \pi_k}\bigg(\log\bigg(\frac{1-(\lambda_N)_k-\lambda_k}{1-(\lambda_N)_k}\bigg)+\frac{\lambda_k}{1-(\lambda_N)_k}\bigg).
\end{align*}
Increasing $C$ if necessary, Taylor's theorem with differential remainder gives a perturbation parameter $\tilde{\lambda}$ with $\tilde{\lambda}_k\in [2^{-k-1},2^{-k}]$ for $0\leq k\leq K_+$ whose value might not be the same at each occurrence such that
\begin{align*}
\sqrt{(\lambda_N)_{0,N}+\lambda_{0,N}}-\sqrt{(\lambda_N)_{0,N}}-\frac{\lambda_{0,N}}{2\sqrt{(\lambda_N)_{0,N}}}&=-\frac{(\lambda_{0,N})^2}{8\tilde{\lambda}_{0,N}^{3/2}}\geq -\sqrt{\epsilon_N}\lambda_0^2\geq -\lambda_0^2\\
(\lambda_N)_k^2+2\lambda_k(\lambda_N)_k-\big((\lambda_N)_k+\lambda_k\big)^2&=-\lambda_k^2\\
\log\bigg(\frac{1-(\lambda_N)_k-\lambda_k}{1-(\lambda_N)_k}\bigg)+\frac{\lambda_k}{1-(\lambda_N)_k}&=-\frac{\lambda_k^2}{2(1-\tilde{\lambda}_k)^2}\geq -2\lambda_k^2.
\end{align*}
It follows that the left-hand side of \eqref{eqn: SBM perturbed free energy Hessian bounded below} is bounded from below by
$$-\frac{\lambda_0^2}{N}\sum_{i\leq N}\abs{Z_{0,i}}-\frac{1}{N}\sum_{1\leq k\leq K_+}\lambda_k^2\sum_{j\leq \pi_k}(8e_{jk}+2)\geq -X\norm{\lambda}_2^2$$
for the random variable
$$X=\frac{1}{N}\sum_{i\leq N}\abs{Z_{0,i}}+\frac{1}{N}\sum_{1\leq k\leq K_+}\sum_{j\leq \pi_k}(8e_{jk}+2).$$
Using the Cauchy-Schwarz inequality, taking the average with respect to the randomness of $(e_{jk})$ before the average with respect to the randomness of $(\pi_k)$ and remembering \eqref{eqn: SBM s sequence} shows that
\begin{align*}
\E X^2&\leq \frac{C}{N^2}\bigg(\E\Big(\sum_{i\leq N}\abs{Z_{0,i}}\Big)^2+\sum_{1\leq k\leq K_+}\E\Big(\sum_{j\leq \pi_k}(8 e_{jk}+2)\Big)^2\bigg)\\
&\leq \frac{C}{N^2}\Big(N\E\abs{Z_{0,1}}+(N^2-N)\E\abs{Z_{0,1}Z_{0,2}}+\sum_{1\leq k\leq K_+} \E\pi_k\sum_{j\leq \pi_k}(8e_{jk}+2)^2\Big)\\
&\leq \frac{C}{N^2}\big(N^2+s_N^2+s_N\big)\leq C.
\end{align*}
This completes the proof.
\end{proof}

\begin{lemma}\label{SBM free energy uniform concentration}
For every $M > 0$ small enough, $p \in [1,\infty)$ and $\epsilon>0$, there exists $C < \infty$ not depending on $N$ such that
\begin{equation}
\bigg(\E\sup_{\norm{\lambda}_\infty\leq M}\Big\lvert\Big(F_N^{(K)}-\overline{F}_N^{(K)}\Big)(t_N,x_N,\lambda_N+\lambda)\Big\rvert^p\bigg)^{\frac{1}{p}}\leq CN^{-\frac{1}{2}+\epsilon}.
\end{equation}
\end{lemma}

\begin{proof}
Let $0<M<1/2$ be small enough so that $(\lambda_N+\lambda)_k\in [2^{-k-1},2^{-k}]$ for $0\leq k\leq K_+$ whenever $\norm{\lambda}_\infty\leq M$, and for each perturbation parameter $\lambda$ introduce the random variable
$$Y(\lambda)=\frac{1}{N}\sum_{0\leq k\leq K_+}\abs{\langle \HH_k'\rangle},$$
where the Gibbs average is associated with the perturbed Hamiltonian \eqref{eqn: SBM perturbed Hamiltonian} evaluated at the triple $(t_N,x_N,\lambda_N+\lambda)$. The relevance of these random variables stems from the fact that by the mean value theorem, \eqref{eqn: SBM lambda 0 derivatives} and \eqref{eqn: SBM lambda k derivatives}, for every $\lambda,\lambda'$ in the $\ell^\infty$-ball of radius $M$,
$$\big\lvert F_N^{(K)}(t_N,x_N,\lambda_N+\lambda)-F_N^{(K)}(t_N,x_N,\lambda_N+\lambda')\big\rvert\leq C\sup_{\norm{\eta}_\infty\leq M}Y(\eta)\norm{\lambda-\lambda'}_1.$$
Averaging this inequality also shows that for every $\lambda,\lambda'$ in the $\ell^\infty$-ball of radius $M$,
$$\big\lvert \overline{F}_N^{(K)}(t_N,x_N,\lambda_N+\lambda)-\overline{F}_N^{(K)}(t_N,x_N,\lambda_N+\lambda')\big\rvert\leq\E\sup_{\norm{\eta}_\infty\leq M}Y(\eta)\norm{\lambda-\lambda'}_1$$
These two bounds imply that for any even integer $q\geq 2$,
\begin{align*}
\E\sup_{\norm{\lambda}_\infty\leq M}\Big\lvert\Big(F_N^{(K)}-\overline{F}_N^{(K)}\Big)(t_N,x_N,\lambda_N+\lambda)\Big\rvert^q\leq \E\sup_{\lambda \in A_\epsilon}\Big\lvert\Big(F_N^{(K)}-\overline{F}_N^{(K)}&\Big)(t_N,x_N,\lambda_N+\lambda)\Big\rvert^q\\
&\quad +C\epsilon^{q}\E\sup_{\norm{\lambda}_\infty\leq M}Y(\lambda)^q
\end{align*}
for $A_\epsilon=\epsilon \Z^{1+K_+}\cap\{\norm{\lambda}_\infty\leq M\}$. Indeed, every $\lambda$ is at most at distance $\epsilon (K_++1)$ from an element in $A_\epsilon$ with respect to the $\ell^1$-norm. Bounding the supremum over $A_\epsilon$ by the sum over $A_\epsilon$ and invoking the free energy concentration result in \Cref{SBM free energy concentration} shows that
\begin{equation}\label{eqn: SBM free energy uniform concentration key}
\E\sup_{\norm{\lambda}_\infty\leq M}\Big\lvert\Big(F_N^{(K)}-\overline{F}_N^{(K)}\Big)(t_N,x_N,\lambda_N+\lambda)\Big\rvert^q\leq C \abs{A_\epsilon}N^{-\frac{q}{2}}+C\epsilon^q\E\sup_{\norm{\lambda}_\infty\leq M}\abs{Y(\lambda)}^q.
\end{equation}
To bound the moments of $\smash{\sup_{\norm{\lambda}_\infty\leq M}\abs{Y(\lambda)}}$ fix $1\leq k\leq K_+$. Hölder's inequality and \eqref{eqn: SBM lambda k derivative of Hamiltonian} reveal that
\begin{align*}
\E \sup_{\norm{\lambda}_\infty\leq M}\abs{\langle \HH_k'\rangle}^q&\leq \E\sup_{\norm{\lambda}_\infty\leq M} \bigg\lvert\sum_{j\leq \pi_k}\frac{1}{1-(\lambda_N+\lambda)_k}+\frac{e_{jk}}{(1-(\lambda_N+\lambda)_k)^2}\bigg\rvert^q\\
&\leq \E \bigg\lvert \sum_{j\leq \pi_k}\frac{4(1+e_{jk})}{(1-2M)^2}\bigg\rvert^q\leq \E \pi_k^{q-1}\sum_{j\leq \pi_k}\sum_{j\leq \pi_k}\bigg\lvert \frac{4(1+e_{jk})}{(1-2M)^2}\bigg\rvert^q\\
&\leq C\E \pi_k^q,
\end{align*}
where the last inequality is found by averaging over the randomness of $(e_{jk})$. Similarly \eqref{eqn: SBM lambda 0 derivative of Hamiltonian} and Hölder's inequality give
\begin{align*}
\E \sup_{\norm{\lambda}_\infty\leq M}\abs{\langle \HH_0'\rangle}^q&\leq\epsilon_N^q \E\sup_{\norm{\lambda}_\infty\leq M}\bigg\lvert \abs{\sigma\cdot \sigma^*}+\frac{\abs{\sigma\cdot Z_0}}{2\epsilon_N^{q/2}((\lambda_N)_0+\lambda_0)^{q/2}}\bigg\rvert^q\\
&\leq C\epsilon_N^{q}\bigg(N^q+\frac{2^q\E\abs{\sigma\cdot Z_0}^q}{2\epsilon_N^{q/2}(1-2M)^q}\bigg)\leq C \epsilon_N^{\frac{q}{2}}N^q.
\end{align*}
Combining these two inequalities with \Cref{SBMA Poisson moment bounds} and recalling the choices \eqref{eqn: SBM epsilon sequence} and \eqref{eqn: SBM s sequence} of $\epsilon_N$ and $s_N$ shows that
\begin{equation}\label{eqn: SBM Lipschitz constant of unaveraged free energy finite}
\E \sup_{\norm{\lambda}_\infty\leq M}\abs{Y(\lambda)}^q\leq C.
\end{equation}
Substituting this into \eqref{eqn: SBM free energy uniform concentration key} and noticing that $\abs{A_\epsilon}$ is of order $\epsilon^{-(K_++1)}$ yields
$$\E\sup_{\norm{\lambda}_\infty\leq M}\Big\lvert\Big(F_N^{(K)}-\overline{F}_N^{(K)}\Big)(t_N,x_N,\lambda_N+\lambda)\Big\rvert^q\leq C\big(\epsilon^{-(K_++1)}N^{-\frac{q}{2}}+\epsilon^q\big).$$
Taking $1/q$'th powers and choosing $\epsilon=N^{-\frac{q}{2(q+K_++1)}}$ gives
$$\bigg(\E\sup_{\norm{\lambda}_\infty\leq M}\Big\lvert\Big(F_N^{(K)}-\overline{F}_N^{(K)}\Big)(t_N,x_N,\lambda_N+\lambda)\Big\rvert^q\bigg)^{\frac{1}{q}}\leq CN^{-\frac{q}{2(q+K_++1)}}.$$
Notice that the power on the right-hand side can be made arbitrarily close to $-\frac{1}{2}$ by taking $q$ large enough. Invoking Jensen's inequality completes the proof. 
\end{proof}

\begin{lemma}\label{perturbed free energy concentration of gradient}
For every $\epsilon>0$, there exists a constant $C<\infty$ not depending on $N$ such that
\begin{equation}
\E\Big\lVert\nabla_{\lambda}\Big(F_N^{(K)}-\overline{F}_N^{(K)}\Big)(t_N,x_N,\lambda_N)\Big\rVert_2^2\leq CN^{-\frac{1}{2}+\epsilon}.
\end{equation}
\end{lemma}

\begin{proof}
Given $\mu \in [0,C^{-1}]$, consider the random perturbation parameter
$$\lambda=\mu \cdot \frac{\nabla_{\lambda}\Big(F_N^{(K)}-\overline{F}_N^{(K)}\Big)(t_N,x_N,\lambda_N)}{\Big\lVert \nabla_{\lambda}\Big(F_N^{(K)}-\overline{F}_N^{(K)}\Big)(t_N,x_N,\lambda_N)\Big\rVert_2}.$$
Combining \Cref{SBM perturbed free energy Hessian bounded above} with \Cref{SBM perturbed free energy Hessian bounded below} shows that
\begin{align*}
\Big(F_N^{(K)}-\overline{F}_N^{(K)}\Big)(t_N,x_N,\lambda_N+\lambda)-\Big(&F_N^{(K)}-\overline{F}_N^{(K)}\Big)(t_N,x_N,\lambda_N)\\
&\geq \mu\Big\lVert \nabla_{\lambda}\Big(F_N^{(K)}-\overline{F}_N^{(K)}\Big)(t_N,x_N,\lambda_N)\Big\rVert_2-(C+X)\norm{\lambda}_2^2.
\end{align*}
Rearranging, squaring and taking expectations yields
\begin{align*}
\mu^2\E\Big\lVert\nabla_{\lambda}\Big(F_N^{(K)}-\overline{F}_N^{(K)}\Big)(t_N,x_N,\lambda_N)\Big\rVert_2^2\leq C\bigg(\E\sup_{\norm{\lambda}_{\infty}\leq C^{-1}}\Big\rvert\Big(F_N^{(K)}-\overline{F}_N^{(K)}\Big)(t_N,x_N,\lambda_N&+\lambda)\Big\rvert^2+\mu^4\bigg),
\end{align*}
where we have used the fact that $\E X^2\leq C$ and $\norm{\lambda}_2=\mu$. Invoking \Cref{SBM free energy uniform concentration} gives
$$\E\Big\lVert\nabla_{\lambda}\Big(F_N^{(K)}-\overline{F}_N^{(K)}\Big)(t_N,x_N,\lambda_N)\Big\rVert_2^2\leq C\bigg(\frac{1}{N^{1-2\epsilon}\mu^2}+\mu^2\bigg).$$
Optimizing over $\mu$ leads to the choice $\mu=N^{-\frac{1}{4}+\frac{\epsilon}{2}}$ and completes the proof.
\end{proof}

\begin{lemma}\label{SBM concentration of FDS Hamiltonians}
For any $1\leq k\leq K_+$, there exists a constant $C<\infty$ not depending on $N$ such that
\begin{equation}\label{eqn: SBM concentration of FDS Hamiltonians}
\E \big\langle (\L_0-\E\langle \L_0\rangle)^2\big\rangle\leq CN^{-\frac{1}{4}} \quad \text{and} \quad \E \big\langle (\L_k-\E\langle \L_k\rangle)^2\big\rangle\leq CN^{-\frac{1}{20}}.
\end{equation}
Here, the Gibbs average $\langle \cdot \rangle$ is associated with the perturbed Hamiltonian \eqref{eqn: SBM perturbed Hamiltonian} evaluated at the contact point $(t_N,x_N,\lambda_N)$. 
\end{lemma}

\begin{proof}
A direct computation using \eqref{eqn: SBM lambda 0 derivatives} shows that
\begin{align*}
N^2\epsilon_N^2\E \big\langle (\L_0-\E\langle \L_0\rangle)^2\big\rangle&=\E \big\langle (\HH_0'-\langle \HH_0'\rangle)^2\big\rangle+\E \big(\langle \HH_0'\rangle-\E\langle \HH'_0\rangle\big)^2\\
&=N\partial_{\lambda_0}^2\overline{F}_N^{(K)}(t_N,x_N,\lambda_N)+\frac{\epsilon_N^2}{4\lambda_{0,N}^{3/2}}\E\langle \sigma\rangle \cdot Z_0\\
&\qquad\qquad\qquad\qquad\quad+N^2\E\Big(\partial_{\lambda_0}\Big(F_N^{(K)}-\overline{F}_N^{(K)} \Big)(t_N,x_N,\lambda_N)\Big)^2.
\end{align*}
It follows by  \Cref{SBM perturbed free energy Hessian bounded above} and \Cref{perturbed free energy concentration of gradient} that for any $\epsilon>0$,
$$\E \big\langle (\L_0-\E\langle \L_0\rangle)^2\big\rangle\leq \frac{C}{N^2\epsilon_N^2}\Big(N+N\epsilon_N^{2-\frac{3}{2}}+N^{2-\frac{1}{2}+\epsilon}\Big)=C\Big(N^{2\abs{\gamma}-1}+N^{\frac{3}{2}\abs{\gamma}-1}+N^{2\abs{\gamma}+\epsilon-\frac{1}{2}}\Big).$$
Remembering that $-1/8<\gamma<0$ gives the first bound in \eqref{eqn: SBM concentration of FDS Hamiltonians}. To establish the second bound, fix $1\leq k\leq K_+$. A direct computation using \eqref{eqn: SBM lambda k derivatives} yields
\begin{align*}
s_N^2\E \big\langle (\L_k-\E\langle \L_k\rangle)^2\big\rangle&=\E \big\langle (\HH_k'-\langle \HH_k'\rangle)^2\big\rangle+\E \big(\langle \HH_k'\rangle-\E\langle \HH'_k\rangle\big)^2\\
&=N\partial_{\lambda_k}^2 \overline{F}_N^{(K)}(t_N,x_N,\lambda_N)-\E\langle \HH_k''\rangle\\
&\qquad\qquad\qquad\qquad\quad+N^2\E\Big(\partial_{\lambda_k}\Big(F_N^{(K)}-\overline{F}_N^{(K)} \Big)(t_N,x_N,\lambda_N)\Big)^2.
\end{align*}
Invoking \eqref{eqn: SBM Hk second derivative}, \Cref{SBM perturbed free energy Hessian bounded above} and \Cref{perturbed free energy concentration of gradient} reveals that for any $\epsilon>0$,
$$\E \big\langle (\L_k-\E\langle \L_k\rangle)^2\big\rangle\leq \frac{C}{s_N^2}\Big(N+s_N+N^{\frac{3}{2}+\epsilon}\Big)=C\Big(N^{1-2\eta}+N^{-\eta}+N^{\frac{3}{2}+\epsilon-2\eta}\Big).$$
Choosing $\epsilon=1/20$, and recalling that $-1/8<\gamma<0$ and $4/5<\eta<1$ completes the proof.
\end{proof}

This result implies the fundamental assumption \eqref{eqn: SBMA limits to get FDS} in \Cref{SBM app multioverlap concentration}. Combining this with \Cref{SBMA concentration of R1}, \Cref{SBMA concentration of R2} and \Cref{SBMA Franz de Sanctis} and fixing $\epsilon>0$, it is possible to find $\delta>0$ so the statement of \Cref{SBMA finitary MOC} holds. In particular, setting $K_+=\lfloor \delta^{-1}\rfloor$ in the perturbed Hamiltonian \eqref{eqn: SBM perturbed Hamiltonian} ensures that
\begin{equation}\label{eqn: SBM multi-overlap concentration}
\E\big\langle (R_{[m]}-\E\langle R_{[m]}\rangle)^2\big\rangle\leq \epsilon
\end{equation}
for $1\leq m\leq \lfloor \epsilon^{-1}\rfloor$. Together with the computations in \Cref{sec: SBM HJ eqn derivation}, this multi-overlap concentration allows us to finally give a proof of \Cref{SBM free energy approximate solution at contact point}.

\begin{proof}[Proof of \Cref{SBM free energy approximate solution at contact point}.]
To alleviate notation, we always implicitly assume that $\overline{F}_N^{(K)}$ and its derivatives are evaluated at the contact point $(t_N,x_N,\lambda_N)$. The definition of the modified free energy in \eqref{eqn: SBM modified free energy} and \Cref{SBM enriched free energy time derivative Taylor} reveal that
\begin{equation}\label{eqn: SBM free energy approximate solution at contact time-derivative}
\partial_t\overline{F}_N'^{(K)}=\frac{1}{2}\big(c+\Delta\m^2\big)\log (c)+\frac{c}{2}\sum_{n\geq 2}\frac{(-\Delta/c)^n}{n(n-1)}\E\big\langle R_{[n]}^2\big\rangle-\frac{c}{2}+\frac{b}{2}+\BigO(N^{-1}).
\end{equation}
On the other hand, the duality relation \eqref{eqn: SBM Gateux derivative projection}, the definition of the modified free energy in \eqref{eqn: SBM modified free energy} and \Cref{SBM enriched free energy Gateaux derivative Taylor} imply that for any $k\in \D_K$,
\begin{align*}
\partial_{x_k}\overline{F}_N'^{(K)}&=\frac{1}{\abs{\D_K}}\bigg(\big(c+\Delta \m k\big)\log(c)+c\sum_{n\geq 2}\frac{(-\Delta/c)^n}{n(n-1)}\E\langle R_{[n]}\rangle k^n-c+b\bigg)+\BigO(N^{-1}).
\end{align*}
If we denote by $\mu^*=\L(\langle \sigma_i\rangle)$ the law of the Gibbs average of a uniformly sampled spin coordinate, then the Nishimori identity \eqref{eqn: SBM Nishimori identity} and the definition of $\td g_b$ in \eqref{e.def.tdg} allow us to rewrite this as 
$$\partial_{x_{k}}\overline{F}_N'^{(K)}=\frac{1}{\abs{\D_K}}\int_{-1}^1 \td g_b(ky)\ud \mu^*(y)+\BigO\big(N^{-1}\big).$$
The mean value theorem shows that
\begin{align*}
\abs{\D_K}\Big\lvert \frac{1}{\abs{\D_K}}\int_{-1}^1 \td g_b(ky)\ud \mu^*(y) - \widetilde{G}_b^{(K)}x^{(K)}(\mu^*)_{k}\Big\rvert
&\leq \sum_{k'\in \D_K}\int_{k'}^{k'+2^{-K}}\abs{\td g_b(ky)-\td g_b(kk')}\ud \mu^*(y)\\
&\leq \frac{\norm{\td g_b'}_\infty}{2^K},
\end{align*}
which means that
$$\dNorm{1}{\nabla_x \overline{F}_N'^{(K)}-\widetilde{G}_b^{(K)}x^{(K)}(\mu^*)}\leq \frac{\norm{\td g'_b}_\infty}{2^K}+\BigO\big(N^{-1}\big).$$
In particular,
$$\dNorm{1}{\widetilde{G}_b^{(K)}x^{(K)}(\mu^*)}\leq \dNorm{1}{\nabla_x \overline{F}_N'^{(K)}}+\frac{\norm{\td g_b'}_\infty}{2^K}+\BigO\big(N^{-1}\big)\leq \norm{\td g_b}_\infty+\norm{\td g_b'}_\infty+\BigO\big(N^{-1}\big)\leq R$$
for $N$ large enough.
Remembering that $\smash{\widetilde{\H}_{b,K,R}}$ coincides with $\smash{\widetilde{\C}_{b,K}}$ on $\smash{\td \CC_{b,K}\cap B_{K,R}}$ and leveraging the Lipschitz continuity of $\smash{\widetilde{\H}_{b,K,R}}$ in \Cref{SBM extending the non-linearity} gives
$$\Big\lvert \widetilde{\H}_{b,K,R}\Big(\nabla_x \overline{F}_N'^{(K)}\Big)-\widetilde{\C}_{b,K}\big(\widetilde{G}_b^{(K)}x^{(K)}(\mu^*)\big)\Big\rvert\leq \frac{8RM_b\norm{\td g_b'}_\infty}{2^Km_b^2}+\BigO\big(N^{-1}\big).$$
Another application of the mean value theorem shows that
$$\Big\lvert \widetilde{\C}_{b,K}\big(\widetilde{G}_b^{(K)}x^{(K)}(\mu^*)\big)-\frac{1}{2}\int_{-1}^1 \int_{-1}^1 \td g_b(xy)\ud \mu^*(y)\ud \mu^*(x)\Big\rvert\leq \frac{\norm{\td g_b'}_\infty}{2^K},
$$
while a direct computation using the Nishimori identity \eqref{eqn: SBM Nishimori identity} reveals that
$$\frac{1}{2}\int_{-1}^1 \td g_b(xy)\ud \mu^*(y)\ud\mu^*(x)=\frac{1}{2}\big(c+\Delta \m^2\big)\log(c)+\frac{c}{2}\sum_{n\geq 2}\frac{(-\Delta/c)^n}{n(n-1)}\big(\E\langle R_{[n]}\rangle\big)^2-\frac{c}{2}+\frac{b}{2}.$$
It follows by \eqref{eqn: SBM free energy approximate solution at contact time-derivative} that up to an error vanishing with $N$,
$$\Big\lvert \partial_t\overline{F}_N'^{(K)}-\widetilde{\H}_{K,R}\Big(\nabla_x \overline{F}_N'^{(K)}\Big)\Big\rvert\leq \Big\lvert\frac{c}{2}\sum_{n\geq 2}\frac{(-\Delta/c)^n}{n(n-1)}\E\big\langle (R_{[n]}-\E\langle R_{[n]}\rangle)^2\big\rangle\Big\rvert+\frac{8RM_b\norm{\td g_b'}_\infty}{2^Km_b^2}+\frac{\norm{\td g'}_\infty}{2^K}.$$
Invoking the multi-overlap concentration \eqref{eqn: SBM multi-overlap concentration}, noticing that the multi-overlaps are bounded by one and using the formula for the sum of a geometric series implies that, up to an error vanishing with $N$,
$$\Big\lvert \partial_t\overline{F}_N'^{(K)}-\widetilde{\H}_{K,R}\Big(\nabla_x \overline{F}_N'^{(K)}\Big)\Big\rvert\leq \frac{\epsilon c^2}{2(c-\abs{\Delta})}+\frac{c}{2}\sum_{n\geq \lfloor\epsilon^{-1}\rfloor}\big(\abs{\Delta}/c\big)^n+\frac{8RM_b\norm{\td g_b'}_\infty}{2^Km_b^2}+\frac{\norm{\td g'}_\infty}{2^K}.$$
Defining $\EE_{\epsilon,K}$ to be the right-hand side of this expression completes the proof.
\end{proof}

\section{The disassortative sparse stochastic block model}\label{sec: SBM disassortative Hopf-Lax}
In this section, we leverage \Cref{SBM main result} to recover the known variational formula for the sparse stochastic block model in the disassortative regime, $\Delta\leq 0$. The idea will be to apply Theorem 1.5 in~\cite{TD_JC_HJ} to the infinite-dimensional Hamilton-Jacobi equation \eqref{eqn: SBM enriched free energy HJ equation} and obtain an infinite-dimensional Hopf-Lax formula. This Hopf-Lax formula will coincide with the standard variational formula obtained in \cite{PanNote, CKPZ} and stated in \Cref{SBM main result}. The key insight will be that in the disassortative setting, for $b$ large enough, the function $\td g_b$ in \eqref{e.def.tdg} is such that
\begin{equation}\label{eqn: SBME tg non-negative definite}
\int_{-1}^1 \int_{-1}^1\td g_b(xy) \ud \mu(x)\ud \mu(y)\geq 0
\end{equation}
for every signed measures $\mu \in \M_s$. This assumption is equivalent to the non-negative definiteness of each of the matrices $\smash{\td G_b^{(K)}}$ and to the convexity of each of the projected non-linearities \eqref{eqn: projected non-linearity b}.

\begin{lemma}
If $\Delta\leq0$ and $b$ is large enough, then the function $\td g_b:[-1,1]\to \R$ defined in \eqref{e.def.tdg} satisfies \eqref{eqn: SBME tg non-negative definite}.
\end{lemma}

\begin{proof}
By a simple approximation argument, it suffices to establish \eqref{eqn: SBME tg non-negative definite} for a discrete signed measure of the form
$$\mu=\frac{1}{\abs{\D_K}}\sum_{k\in \D_K} x_k \delta_k$$
for some $\smash{x\in \R^{\D_K}}$. For such a measure,
$$\int_{-1}^1 \int_{-1}^1\td g_b(xy) \ud \mu(x)\ud \mu(y)=\frac{1}{\abs{\D_K}^2}\sum_{k,k'\in\D_K}\td g_b(kk')x_kx_{k'}=x^T\widetilde{G}_b^{(K)}x,$$
so \eqref{eqn: SBME tg non-negative definite} is equivalent to the non-negative definiteness of each of the matrices $\smash{\widetilde{G}_b^{(K)}}$. Observe that for any $k,k'\in \D_K$,
$$\big(\widetilde{G}_b^{(K)}\big)_{kk'}=\frac{1}{\abs{\D_K}^2}\td g_b(kk')=\frac{1}{\abs{\D_K}^2}\bigg(b+c\log(c)-c+\Delta kk'\log(c)+c\sum_{n\geq 2}\frac{(-\Delta/c)^n}{n(n-1)}(kk')^n\bigg).$$
If we introduce the vectors $\mathbf{k}=(k)_{k\in \D_K}$ and $\mathbf{\iota}=(1)_{k\in \D_K}$, and write $\odot n$ for the $n$-fold Hadamard product on the space of $\D_K\times \D_K$ matrices, this implies that
$$\widetilde{G}_b^{(K)}=\frac{1}{\abs{\D_K}^2}\bigg((b+c\log(c)-c)\mathbf{\iota}\mathbf{\iota}^T+\Delta\log(c)\mathbf{k}\mathbf{k}^T+\sum_{n\geq 2}\frac{(-\Delta/c)^n}{n(n-1)}\big(\mathbf{k}\mathbf{k}^T\big)^{\odot n}\bigg)=\lim_{M\to \infty}\widetilde{G}^{(K)}_{b,M}$$
for the matrix
$$\widetilde{G}^{(K)}_{b,M}=\frac{1}{\abs{\D_K}^2}\bigg((b+c\log(c)-c)\mathbf{\iota}\mathbf{\iota}^T+\Delta\log(c)\mathbf{k}\mathbf{k}^T+\sum_{2\leq n\leq M}\frac{(-\Delta/c)^n}{n(n-1)}\big(\mathbf{k}\mathbf{k}^T\big)^{\odot n}\bigg).$$
Choosing $b>2c\abs{\log(c)}+c$ ensures that the first two terms in this sum define a non-negative definite matrix. Using that $\Delta \le 0$ and the Schur product theorem, we see that the matrix $\smash{\widetilde{G}^{(K)}_{b,M}}$ is a positive linear combination of non-negative definite matrices, and is therefore non-negative definite. Noticing that the limit of non-negative definite matrices is again non-negative definite completes the proof.
\end{proof}

This result allows us to apply Theorem 1.5 in \cite{TD_JC_HJ} to the infinite-dimensional Hamilton-Jacobi equation~\eqref{eqn: SBM enriched free energy HJ equation} and obtain a variational formula for its solution. To state this formula concisely, through a slight abuse of notation, 
introduce the functional $\smash{\Par:\Rpp\times \M_+\times \M_+\to \R}$ defined by
\begin{equation}
\Par(t,\mu,\nu)=\psi(\mu+t\nu)-\frac{t}{2}\int_{-1}^1 G_\nu(y)\ud \nu(y).
\end{equation}
Theorem 1.5 in \cite{TD_JC_HJ} and \eqref{SBM H a} imply that the unique solution to the infinite-dimensional Hamilton-Jacobi equation~\eqref{eqn: SBM enriched free energy HJ equation} is given by the Hopf-Lax formula
\begin{equation}\label{eqn: SBM modified FE Hopf-Lax}
f(t,\mu)=\sup_{\nu\in \Pr[-1,1]}\Par(t,\mu,\nu)
\end{equation}
for every $t>0$ and $\mu \in \M_+$. Combining this representation formula with \Cref{SBM main result} and a simple interpolation argument taken from \cite{PanNote}, we now prove \Cref{SBM disassortative Hopf-Lax}.

\begin{proof}[Proof of \Cref{SBM disassortative Hopf-Lax}.]
By \Cref{SBM main result} and the Hopf-Lax formula \eqref{eqn: SBM modified FE Hopf-Lax}, the limit of the free energy~\eqref{eqn: SBM free energy} satisfies the upper bound
\begin{equation}\label{eqn: SBM disassortative Hopf-Lax upper bound key}
\limsup_{N\to \infty}\overline{F}_N\leq \sup_{\nu\in \Pr[-1,1]}\Par(1,0,\nu),
\end{equation}
where $0$ denotes the zero measure. To show that the right-hand side of this expression may be bounded by the supremum of the functional \eqref{eqn: SBM Parisi functional} over measures $\mu\in \M_p$, for $b>1$ large enough so the function $\td g_b$ defined in \eqref{e.def.tdg} is strictly positive on $[-1,1]$ and each integer $N\geq 1$, introduce the functional $\smash{\td\Par_{b,N}: \Rp\times \M_+\times \M_+\to \R}$ defined by
$$\td \Par_{b,N}(t,\mu,\nu)=\td \psi_{b,N}(\mu+t\nu)-\frac{t}{2}\int_{-1}^1 \td G_{b,\nu}(y)\ud \nu(y).$$
Here $\td \psi_{b,N}:\M_+\to \R$ denotes the initial condition
$$\td \psi_{b,N}(\mu)=\psi_N(\mu)+b\int_{-1}^1 \ud \mu$$
and $\td G_{b,\nu}:[-1,1]\to \R$ denotes the function
$$\td G_{b,\nu}(x)=\int_{-1}^1 \td g_b(xy)\ud \nu(y).$$
By Theorem 1.5 in \cite{TD_JC_HJ}, there exists a probability measure $\td \nu\in \Pr[-1,1]$ which maximizes the right-hand side of \eqref{eqn: SBM disassortative Hopf-Lax upper bound key}. It follows by \Cref{SBM convergence of initial condition} that
\begin{equation}\label{eqn: SBM disassortative Hopf-Lax upper bound P pre}
\sup_{\nu\in \Pr[-1,1]}\Par(1,0,\nu)=\lim_{N\to \infty}\td \Par_{b,N}(1,0,\td \nu)-\frac{b}{2}\leq \limsup_{N\to \infty}\sup_{\nu\in \M_+}\td \Par_{b,N}(1,0,\nu)-\frac{b}{2}.
\end{equation}
An identical argument to that in Lemma 5.1 of \cite{TD_JC_HJ} gives a sequence of maximizing measures $(\nu_N)\subset \M_+$ with
\begin{equation}\label{eqn: SBM disassortative Hopf-Lax upper bound P maximizer}
\sup_{\nu\in \M_+}\td \Par_{b,N}(1,0,\nu)=\td \Par_{b,N}(1,0,\nu_N).
\end{equation}
By \Cref{SBM enriched free energy Gateaux derivative Taylor} and \eqref{eqn: SBM enriched free energy in cone}, the Gateaux derivative density of the initial condition $\smash{\td \psi_{b,N}}$ at the measure $\nu_N$ is given by
$\smash{D_\mu \td \psi_{b,N}(\nu_N,\cdot)=\td G_{b,\nu_N^*}+\BigO\big(N^{-1}\big)}$ for some measure $\nu_N^*\in \M_p$. Up to adding errors of $\BigO(N^{-1})$ throughout, the proof of Theorem 1.3 in \cite{TD_JC_HJ} applies and reveals that each maximizer $\nu_N\in \M_+$ satisfies the approximate first order condition
$$D_\mu \td \psi_{b,N}(\nu_N,\cdot)=\td G_{b,\nu_N}+\BigO\big(N^{-1}\big).$$
This means that
$$\td G_{b,\nu_N^*}=D_\mu\td \psi_{b,N}(\nu_N,\cdot)+\BigO\big(N^{-1}\big)=\td G_{b,\nu_N}+\BigO\big(N^{-1}\big).$$
Together with the definition of $\smash{\td g_b}$ in \eqref{e.def.tdg}, this implies that
\begin{align*}
(b+c\log(c)-c)&\int_{-1}^1 \ud \nu_N(y)+\Delta\log(c)\int_{-1}^1 y\ud \nu_N(y)x+c\sum_{n\geq 2}\frac{(-\Delta/c)^n}{n(n-1)}\int_{-1}^1 y^n\ud \nu_N(y)x^n\\
&=(b+c\log(c)-c)+\Delta \log(c)\m x+c\sum_{n\geq 2}\frac{(-\Delta/c)^n}{n(n-1)}\int_{-1}^1 y^n\ud \nu_N^*(y)x^n+\BigO\big(N^{-1}\big).
\end{align*}
Since $b>1$ and $c\log(c)-c\geq -1$, we must have
$$\int_{-1}^1 \ud \nu_N(y)=1+\BigO\big(N^{-1}\big) \quad \text{and} \quad \int_{-1}^1 y^n\ud \nu_N(y)=\int_{-1}^1 y^n\ud \nu_N^*(y)+\BigO\big(N^{-1}\big)$$
for all $n\geq 2$.
We have used the fact that $\Delta\neq 0$ and accounted for the fact that $c$ could be equal to one. Applying the Prokhorov theorem and passing to a subsequence if necessary, it is therefore possible to ensure that the sequences $\smash{(\nu_N)}$ and $\smash{(\nu_N^*)}$ converge weakly to probability measures $\smash{\nu\in \Pr[-1,1]}$ and $\smash{\nu^*\in \M_p}$ with
$$\int_{-1}^1 y^n\ud \nu(y)=\int_{-1}^1 y^n\ud \nu^*(y)$$
for all $n\neq 1$. Since the set of polynomials with degree one coefficient equal to zero form of a sub-algebra of the space of continuous functions on the compact set $[-1,1]$, the Stone-Weierstrass theorem implies that $\smash{\nu=\nu^*\in \M_p}$. Recalling that we denote $\bar \nu_N$ for the probability measure induced by $\nu_N$, and arguing as in the proof of \Cref{SBM assumptions for WP of HJ}, we have
\begin{align*}
\big\lvert \td \psi_{b,N}(\nu_N)-\td \psi_b(\nu^*)\big\rvert
&
\leq \big\lvert \td\psi_{b,N}(\nu_N)-\td \psi_{b,N}(\bar \nu_N)\big\rvert + \big\lvert \td\psi_{b,N}(\bar \nu_N)-\td \psi_{b,N}(\nu^*)\big\rvert+\big\lvert \td \psi_{b,N}(\nu^*)-\td \psi_b(\nu^*)\big\rvert\\
&
\leq C\TV(\nu_N,\bar \nu_N) +C W(\bar \nu_N, \nu^*)  + \big\lvert \td \psi_{b,N}(\nu^*)-\td \psi_b(\nu^*)\big\rvert,
\end{align*}
for some constant $C>0$ that depends only on $c$.
Recalling that the Wasserstein distance \eqref{eqn: SBM Wasserstein} metrizes the weak convergence of probability measures on $[-1,1]$, observing that 
$$\TV(\nu_N,\bar \nu_N) = \big|1- {\nu_N[-1,1]}  \big|, $$
and using \Cref{SBM convergence of initial condition} and \eqref{eqn: SBM disassortative Hopf-Lax upper bound P maximizer} to let $N$ tend to infinity in \eqref{eqn: SBM disassortative Hopf-Lax upper bound P pre} shows that
\begin{align*}
\sup_{\nu\in \Pr[-1,1]}\Par(1,0,\nu)&\leq \limsup_{N\to \infty}\widetilde{\Par}_{b,N}(1,0,\nu_N)-\frac{b}{2}=\td \psi_b(\nu^*)-\frac{1}{2}\int_{-1}^1 \td G_{b,\nu^*}(y)\ud \nu^*(y)-\frac{b}{2}\\
&\leq \sup_{\nu \in \M_p}\Par(1,0,\nu).
\end{align*}
Substituting this upper bound into \eqref{eqn: SBM disassortative Hopf-Lax upper bound key} yields
$$\limsup_{N\to \infty}\overline{F}_N\leq \sup_{\nu \in \M_p}\Par(1,0,\nu).$$
To express this upper bound in terms of the functional \eqref{eqn: SBM Parisi functional}, fix $\nu \in \M_p$ and denote by $x_1$ and $x_2$ two independent samples from the probability measure $\nu$. The definition of $g$ in \eqref{eqn: SBM g function} implies that
$$\Par(1,0,\nu)=\psi(\nu)+\frac{c}{2}-\frac{1}{2}\big(c+\Delta \m^2\big)\log(c)-\frac{c}{2}\sum_{n\geq 2}\frac{(-\Delta/c)^n}{n(n-1)}\big(\E x_1^n\big)^2.$$
A Taylor expansion of the logarithm shows that
\begin{equation}\label{eqn: SBM Taylor for mixed expectation}
c\sum_{n\geq 2}\frac{(-\Delta/c)^n}{n(n-1)}\big(\E x^n\big)^2=\E(c+\Delta x_1x_2)\log(c+\Delta x_1x_2)-\big(c+\Delta \m^2\big)\log(c)-\Delta \m^2
\end{equation}
from which it follows that
$$\Par(1,0,\nu)=\psi(\nu)+\frac{c}{2}-\frac{1}{2}\E(c+\Delta x_1x_2)\log(c+\Delta x_1x_2)+\frac{\Delta\m^2}{2}=\Par(\nu),$$
and therefore
$$\limsup_{N\to \infty}\overline{F}_N\leq \sup_{\nu\in \M_p}\Par(\nu).$$
This establishes the upper bound in \eqref{eqn: SBM disassortative Hopf-Lax}. To prove the corresponding lower bound, we follow~\cite{PanNote} and proceed by interpolation. Given a measure $\nu \in \M_p$, introduce the interpolating free energy
\begin{equation}
\p(t)=\widetilde{F}_N(t, 1-t, \nu)
\end{equation}
for the free energy $\widetilde{F}_N$ defined in \eqref{eqn: SBM enriched free energy s}. The derivative computations in \Cref{SBM enriched free energy time derivative Taylor} and \Cref{SBM enriched free energy s derivative} together with a computation identical to that in \Cref{SBM enriched free energy Gateaux derivative Taylor} imply that
\begin{align*}
\p'(t)&=\partial_t\widetilde{F}_N(t,1-t,\nu)- \partial_s\widetilde{F}_N(t,1-t,\nu)\\
&=\frac{c}{2}-\frac{1}{2}\big(c+\Delta \m^2\big)\log(c)+\frac{c}{2}\sum_{n\geq 2}\frac{(-\Delta/c)^n}{n(n-1)}\E\big\langle R_{[n]}^2\big\rangle- c\sum_{n\geq 2}\frac{(-\Delta/c)^n}{n(n-1)}\E\langle R_{[n]}\rangle \E x_1^n.\\
&=\frac{c}{2}-\frac{1}{2}\big(c+\Delta \m^2\big)\log(c)- \frac{c}{2}\sum_{n\geq 2}\frac{(-\Delta/c)^n}{n(n-1)}\big(\E x_1^n\big)^2+\frac{c}{2}\sum_{n\geq 2}\frac{(-\Delta/c)^n}{n(n-1)}\E\big\langle (R_{[n]}- \E x_1^n)^2\big\rangle.
\end{align*}
It follows by \eqref{eqn: SBM Taylor for mixed expectation} that
$$\p'(t)=\frac{c}{2}+\frac{\Delta\m^2}{2}-\frac{1}{2}\E(c+\Delta x_1 x_2)\log(c+\Delta x_1x_2)+\frac{c}{2}\sum_{n\geq 2}\frac{(-\Delta/c)^n}{n(n-1)}\E\big\langle (R_{[n]}- \E x_1^n)^2\big\rangle.$$
Since the final term in this equality is non-negative, the fundamental theorem of calculus reveals that
$$\overline{F}_N\geq \psi_N(\nu)+\frac{c}{2}+\frac{\Delta\m^2}{2}-\frac{1}{2}\E(c+\Delta x_1 x_2)\log(c+\Delta x_1x_2),$$
where we have used that $\p(1)=\overline{F}_N$ and $\p(0)=\psi_N(\nu)$. Using \Cref{SBM convergence of initial condition} to let $N$ tend to infinity gives the lower bound
$$\liminf_{N\to \infty}\overline{F}_N\geq \psi(\nu)+\frac{c}{2}+\frac{\Delta\m^2}{2}-\frac{1}{2}\E(c+\Delta x_1 x_2)\log(c+\Delta x_1x_2)=\Par(\nu).$$
Taking the supremum over all measures $\nu\in \M_p$ completes the proof.
\end{proof}

\begin{appendix}
\section{Asymptotic equivalence of free energy functionals}\label{SBM app equivalence}

In this appendix we show that the free energy functionals $\overline{F}_N^\circ$ and $\overline{F}_N$ defined in \eqref{eqn: SBM free energy 0} and \eqref{eqn: SBM free energy}, respectively, are asymptotically equivalent. This will be a consequence of the binomial-Poisson approximation. To state this result concisely, given a separable metric space $S$, recall the definition of the total variation distance
\begin{equation}
\TV(\P,\Q)=\sup\big\{\abs{\P(A)-\Q(A)}\mid A \text{ is a measurable subset of } S\big\}  
\end{equation}
between probability measures $\P,\Q\in \Pr(S)$. Approximating any measurable function with values in $S$ by a sequence of simple functions, one can check that the total variation distance admits the dual representation
\begin{equation}
\TV(\P,\Q)=\sup\bigg\{\Big\lvert \int_{-1}^1 f(x)\ud \P(x)-\int_{-1}^1 f(x)\ud \Q(x)\Big\rvert \mid f:S\to [0,1] \text{ measurable}\bigg\}.
\end{equation}
Using the Hahn-Jordan decomposition, it is also possible to show that
\begin{equation}\label{eqn: SBM dual total variation}
\TV(\P,\Q)=\inf\big\{\P\{X\neq Y\}\mid X\sim \P \text{ and } Y\sim \Q\big\}.
\end{equation}
We will use this result for discrete probability measures supported on the set of natural numbers in which case this representation follows from the Kantorovich-Rubinstein theorem (see Theorem 4.15 in \cite{PanL}). The binomial-Poisson approximation is an upper bound on the total variation distance between the convolution of Bernoulli distributions and an appropriate Poisson distribution.

\begin{lemma}\label{SBM Binomial-Poisson}
Consider independent Bernoulli random variables $X_i\sim \Ber(p_i)$ for $i\leq n$, and let $\smash{\lambda_n=\sum_{i\leq n}p_i}$. If $\smash{S_n=\sum_{i\leq n}X_i}$ and $\smash{\Pi_n\sim \Poi(\lambda_n)}$, then
\begin{equation}
\TV(S_n,\Pi_n)\leq \sum_{i\leq n}p_i^2.
\end{equation}
\end{lemma}

\begin{proof}
See Theorem 2.4 in \cite{PanL}.
\end{proof}

\begin{proposition}
The free energies \eqref{eqn: SBM free energy 0} and \eqref{eqn: SBM free energy} are asymptotically equivalent,
\begin{equation}
\lim_{N\to \infty}\big\lvert \overline{F}_N-\overline{F}_N^\circ\big\rvert=0.
\end{equation}
\end{proposition}

\begin{proof}
Introduce the Hamiltonians
\begin{align*}
\widetilde{H}_N^\circ(\sigma)&=\sum_{i<j}\bigg(G_{ij}\log(c+\Delta \sigma_i\sigma_j)-\frac{c+\Delta\sigma_i\sigma_j}{N}\bigg)\\
\widetilde{H}_N(\sigma)&=\sum_{k\leq \Pi_1}\bigg(G_{i_k,j_k}^k\log(c+\Delta \sigma_{i_k}\sigma_{j_k})-\frac{c+\Delta\sigma_{i_k}\sigma_{j_k}}{N}\bigg)
\end{align*}
on $\Sigma_N$, and denote by
$$\widetilde{F}_N^\circ=\frac{1}{N}\E\log\int_{\Sigma_N}\exp \widetilde{H}_N^\circ(\sigma)\ud P_N^*(\sigma) \quad \text{and} \quad \widetilde{F}_N=\frac{1}{N}\E\log\int_{\Sigma_N}\exp \widetilde{H}_N(\sigma)\ud P_N^*(\sigma)$$
their associated free energy functionals. A Taylor expansion of the logarithm shows that for any $\sigma\in \Sigma_N$,
\begin{align*}
 \big\lvert H_N^\circ(\sigma)-\widetilde{H}_N^\circ(\sigma)\big\rvert&\leq \sum_{i<j}\Big\lvert\frac{c+\Delta\sigma_i\sigma_j}{N} -(1-G_{ij})\Big(\frac{c+\Delta \sigma_i\sigma_j}{N}\Big)\Big\rvert+\BigO(1)\\
&\leq \frac{c+\abs{\Delta}}{N}\sum_{i<j}G_{ij}+\BigO(1) 
\end{align*}
and
\begin{align*}
\big\lvert H_N(\sigma)-\widetilde{H}_N(\sigma)\big\rvert&\leq \sum_{k\leq \Pi_1}\Big\lvert\frac{c+\Delta\sigma_{i_k}\sigma_{j_k}}{N} -(1-G_{i_k,j_k}^k)\Big(\frac{c+\Delta \sigma_{i_k}\sigma_{j_k}}{N}\Big)\Big\rvert+\BigO\big(\Pi_1/N^2\big)\\
&\leq \frac{c+\abs{\Delta}}{N}\sum_{k\leq \Pi_1}G_{i_k,j_k}^k+\BigO\big(\Pi_1/N^2\big).
\end{align*}
Since these bounds are uniform in $\sigma\in \Sigma_N$ and $\E\Pi_1=\binom{N}{2}$,
\begin{align*}
\big\lvert \overline{F}_N^\circ-\widetilde{F}_N^\circ\big\rvert&\leq \frac{c+\abs{\Delta}}{N^2}\sum_{i<j}\E G_{ij}+\BigO\big(N^{-1}\big) \leq \frac{(c+\abs{\Delta})^2}{N}+\BigO\big(N^{-1}\big)=\BigO\big(N^{-1}\big),\\
\big\lvert \overline{F}_N-\widetilde{F}_N\big\rvert&\leq \frac{c+\abs{\Delta}}{N^2}\E \sum_{k\leq \Pi_1} G_{i_k,j_k}^k+\BigO\big(\E\Pi_1/N^3\big) \leq \frac{(c+\abs{\Delta})^2}{N^3}\E\Pi_1+\BigO\big(N^{-1}\big)=\BigO\big(N^{-1}\big).
\end{align*}
By the triangle inequality, it therefore suffices to show that
\begin{equation}\label{eqn: SBM equivalence simplification 1}
\lim_{N\to \infty}\big\lvert \widetilde{F}_N-\widetilde{F}_N^\circ\big\rvert = 0.
\end{equation}
We now rewrite $\smash{\widetilde{F}_N}$ in a way that more closely resembles $\smash{\widetilde{F}_N^\circ}$. For each pair $i\leq j$ introduce the random index set
$$\I_{i,j}=\big\{k\leq \Pi_1\mid (i_k,j_k)=(i,j) \text{ or } (i_k,j_k)=(j,i)\big\},$$
and observe that
\begin{align*}
\widetilde{H}_N(\sigma)&=\sum_{i\leq j} \sum_{k\in \I_{i,j}}\bigg(G_{i_k,j_k}^k \log(c+\Delta \sigma_{i_k}\sigma_{j_k})-\frac{c+\Delta \sigma_{i_k}\sigma_{j_k}}{N}\bigg)\\
&=\sum_{i<j}\bigg(\widetilde{G}_{i,j}\log(c+\Delta \sigma_i\sigma_j)-\frac{c+\Delta \sigma_i\sigma_j}{N}\bigg)-\sum_{i\leq N}\bigg(\widetilde{G}_{i,i}\log(c+\Delta)-\frac{c+\Delta}{N}\bigg)
\end{align*}
for the random variables
$$\widetilde{G}_{i,j}=\sum_{k\in \I_{i,j}}G_{i,j}^k.$$
The Poisson coloring theorem (see Chapter 5 in \cite{Kingman}) implies that $\smash{\widetilde{G}_{i,j}}$ is a Poisson random variable with mean
$$\widetilde{\lambda}_{i,j}=\E \Pi_1 \cdot \P\{(i_1,j_1)=(i,j) \text{ or } (i_1,j_1)=(j,i)\}\cdot \P\{G^1_{i,j}=1\}=\begin{cases}
\frac{N-1}{N}\cdot \frac{c+\Delta \sigma_i^*\sigma_j^*}{N}& \text{if }i< j,\\
\frac{N-1}{2N}\cdot \frac{c+\Delta}{N}& \text{if } i=j.
\end{cases}$$
If we introduce the Hamiltonian
$$\widetilde{H}_N'(\sigma)=\sum_{i<j}\bigg(\widetilde{G}_{i,j}\log(c+\Delta \sigma_i\sigma_j)-\frac{c+\Delta \sigma_i\sigma_j}{N}\bigg)$$
and its associated free energy
$$\widetilde{F}_N'=\frac{1}{N}\E\log\int_{\Sigma_N}\exp \widetilde{H}_N'(\sigma)\ud P_N^*(\sigma),$$
then
$$\big\lvert \widetilde{F}_N'-\widetilde{F}_N\big\rvert\leq \frac{\abs{\log(c+\Delta)}}{N}\sum_{i\leq N}\E\widetilde{G}_{i,i}+\frac{c+\abs{\Delta}}{N}\leq \frac{\abs{\log(c+\Delta)}}{N}\sum_{i\leq N}\widetilde{\lambda}_{i,i}+\frac{2c}{N}=\BigO\big(N^{-1}\big).$$
Together with \eqref{eqn: SBM equivalence simplification 1} and the triangle inequality, this means that it suffices to show that
\begin{equation}\label{eqn: SBM equivalence simplification 2}
\lim_{N\to \infty}\big\lvert \widetilde{F}_N'-\widetilde{F}_N^\circ\big\rvert=0.
\end{equation}
At this point, for any random vector $Y=(Y_{i,j})_{i<j}$ introduce the Hamiltonian
$$\widetilde{H}_N(\sigma,Y)=\sum_{i<j}Y_{i,j}\log(c+\Delta \sigma_i\sigma_j)$$
and the measure 
$$\widetilde{P}_N^*(\sigma)=\exp\Big(-\sum_{i<j}\frac{c+\Delta \sigma_i\sigma_j}{N}\Big)P_N^*(\sigma).$$
Write
$$F_N(Y)=\frac{1}{N}\log \int_{\Sigma_N}\exp \widetilde{H}_N(\sigma,Y)\ud \widetilde{P}_N^*(\sigma)\quad \text{and} \quad \widetilde{F}_N(Y)=\E F_N(Y)$$
for the associated free energy functionals, and denote by $\Pi_{i,j}$ a Poisson random variable with distribution 
$$\Pi_{i,j}\sim \Poi\Big(\frac{c+\Delta\sigma_i^*\sigma_j^*}{N}\Big)$$
conditionally on $\sigma^*$. Observe that
$$\partial_{Y_{i,j}}F_N(Y)=\frac{1}{N}\big\langle \partial_{Y_{i,j}}\widetilde{H}_N(\sigma,Y)\big\rangle=\frac{1}{N}\langle \log(c+\Delta \sigma_i\sigma_j)\rangle\leq \frac{C}{N}$$
for some deterministic constant $C>0$ that depends only on $c$. Here, the bracket $\langle \cdot \rangle$ denotes the Gibbs measure associated with the Hamiltonian $\smash{\widetilde{H}_N(\sigma,Y)}$. It follows by the mean value theorem that
\begin{align*}
\big\lvert \widetilde{F}_N'-\widetilde{F}_N(\Pi)\big\rvert&=\big\lvert \widetilde{F}_N(\widetilde{G})-\widetilde{F}_N(\Pi)\big\rvert\leq \frac{C}{N}\sum_{i<j}\E\big\lvert  \widetilde{G}_{i,j}-\Pi_{i,j}\big\rvert\\
\big\lvert \widetilde{F}_N(\Pi)-\widetilde{F}_N^\circ\big\rvert&=\big\lvert \widetilde{F}_N(\Pi)-\widetilde{F}_N(G)\big\rvert\leq \frac{C}{N}\sum_{i<j}\E\big\lvert \Pi_{i,j}-G_{ij}\big\rvert.
\end{align*}
To bound the first of the sums observe that $\smash{\Pi_{i,j}\stackrel{d}{=} \widetilde{G}_{i,j}+\Pi'_{i,j}}$ for a Poisson random variable with distribution
$$\Pi'_{i,j}\sim\Poi\Big(\frac{c+\Delta \sigma_i^*\sigma_j^*}{N^2}\Big)$$
conditionally on $\sigma^*$. This means that
$$\big\lvert\widetilde{F}_N'-\widetilde{F}_N(\Pi)\big\rvert\leq \frac{C}{N}\sum_{i<j}\E\Pi'_{i,j}\leq \frac{Cc}{N}=\BigO\big(N^{-1}\big),$$
and therefore,
\begin{equation}\label{eqn: SBM equivalence simplification 3}
\big\lvert \widetilde{F}_N'-\widetilde{F}_N^\circ\big\rvert\leq \frac{C}{N}\sum_{i<j}\E\big\lvert \Pi_{i,j}-G_{ij}\big\rvert +\BigO\big(N^{-1}\big).
\end{equation}
If we write $\smash{\lambda_{i,j}=\frac{c+\Delta\sigma_i^*\sigma_j^*}{N}}$, then
\begin{align*}
\E\big\lvert \Pi_{i,j}-G_{ij}\big\rvert&\leq \E\lvert \Pi_{i,j}-G_{ij}\big\rvert\1\{\Pi_{i,j}\geq 2\}+\E\lvert \Pi_{i,j}-G_{ij}\big\rvert\1\{\Pi_{i,j}\leq 2\}\1\{\Pi_{i,j}\neq G_{ij}\}\\
&\leq \E\Pi_{i,j}\1\{\Pi_{i,j}\geq 2\}+\P\{\Pi_{i,j}\geq 2\}+3\P\{\Pi_{i,j}\neq G_{ij}\}\\
&\leq 3\P\{\Pi_{i,j}\neq G_{ij}\}+\BigO\big(\lambda_{i,j}^2\big),
\end{align*}
where we have used the fact that $\abs{G_{ij}}\leq 1$. Taking the infimum over all couplings of $\Pi_{i,j}$ and $G_{ij}$, recalling the definition of the total variation distance in \eqref{eqn: SBM dual total variation} and invoking the binomial-Poisson approximation in \Cref{SBM Binomial-Poisson} shows that
$$\E\big\lvert \Pi_{i,j}-G_{ij}\big\rvert\leq 3\TV\big(\Pi_{i,j},G_{ij}\big)+\BigO\big(\lambda_{i,j}^2\big)=\BigO\big(\lambda_{i,j}^2\big)=\BigO\big(N^{-2}\big).$$
Substituting this into \eqref{eqn: SBM equivalence simplification 3} and letting $N$ tend to infinity establishes \eqref{eqn: SBM equivalence simplification 2}. This completes the proof.
\end{proof}

\section{Concentration of the free energy}\label{SBM app FE concentration}

In this appendix we discuss the concentration of the free energy associated with the perturbed Hamiltonian \eqref{eqn: SBM perturbed Hamiltonian}. For simplicity of notation, we will ignore the Hamiltonian \eqref{eqn: SBM enriched Hamiltonian} and focus instead on the perturbed Hamiltonian
\begin{equation}\label{eqn: SBMA perturbed Hamiltonian}
H_N'(\sigma)=H_N^t(\sigma)+H_N^{\mathrm{gauss}}(\sigma)+H_N^{\mathrm{exp}}(\sigma),
\end{equation}
where the randomness of each of the Hamiltonians \eqref{eqn: SBM Hamiltonian t}, \eqref{eqn: SBM gaussian perturbation} and \eqref{eqn: SBM exponential perturbation} is independent of the randomness of the other Hamiltonians. The more general case is treated in an identical fashion but the notation becomes too cumbersome for comfort. We will often need to make the dependence of the perturbed Hamiltonian \eqref{eqn: SBMA perturbed Hamiltonian} on one of its sources of randomness $\smash{\sigma^*}$, $\smash{\Pi_t}$, $\smash{\I_1=(i_k,j_k)_{k\leq \Pi_t}}$, $\smash{\G=(G_{i_k,j_k}^k)_{k\leq \Pi_t}}$, $\smash{e=(e_{jk})}$, $\smash{\Pi'=(\pi_k)_{k\geq 0}}$, $\smash{\I_2=(i_{jk})_{j\leq \Pi'}}$ and $\smash{Z=(Z_{0,i})_{i\leq N}}$ explicit. To do so, we will abuse notation and write $\smash{H_N'(X)}$ when we want to study the dependence on the source of randomness $X$. The main objects of study will be the free energy,
\begin{equation}
F_N'=\frac{1}{N}\log \int \exp H_N'(\sigma)\ud P_N^*(\sigma),
\end{equation}
and its average 
\begin{equation}\label{eqn: SBMA perturbed free energy}
\overline{F}_N'=\frac{1}{N}\E \log \int \exp H_N'(\sigma)\ud P_N^*(\sigma).
\end{equation}
For each even $p\geq 2$, we will bound the concentration function,
\begin{equation}
v_{N,p}=\sup\Big\{\E\big\lvert F_N'-\overline{F}'_N\big\rvert^p\mid \lambda_k\in [2^{-k-1},2^{-k}] \text{ for all } k\geq 0\Big\},
\end{equation}
by means of the generalized Efron-Stein inequality (see Theorem 15.5 of \cite{Boucheron}).
\begin{lemma}[Generalized Efron-Stein inequality]\label{eqn: SBM generalized Efron-Stein}
Let $X=(X_1,\ldots,X_n)$ and $X'=(X_1',\ldots,X_n')$ be two independent copies of a vector of independent random variables, and let $f:\R^n\to \R$ be a measurable function. Introduce the random variable $Z=f(X)$, and for each $1\leq i\leq n$ let $Z_i'=f(X_1,\ldots,X_{i-1},X_i',X_{i+1},\ldots,X_n)$. If $q\geq 2$, then
\begin{equation}
\E\abs{Z-\E Z}^q\leq C \E \Big\lvert \sum_{i\leq n} \E_{X'}(Z- Z_i')^2\Big\rvert^{\frac{q}{2}},
\end{equation}
where $C>0$ is a constant that depends only on $q$.
\end{lemma}
A key observation that will be used repeatedly without further explanation is the following: given two sources of randomness $X$ and $X'$, a configuration-independent bound on the difference of the Hamiltonians $H_N'(X)$ and $H_N'(X')$,
\begin{equation}
\max_{\sigma\in \Sigma_N}\big\lvert H_N'(X)-H_N'(X')\big\rvert\leq Y,
\end{equation}
gives a control by the possibly random $Y$ on the difference of the free energy functionals $F_N'(X)$ and $F_N'(X')$,
\begin{equation}
\big\lvert F_N'(X)-F_N'(X')\big\rvert\leq \frac{Y}{N}.
\end{equation}
The following bounds on the moments of Poisson and binomial random variables will also play their part.

\begin{lemma}\label{SBMA Poisson moment bounds}
If $\Pi$ is a $\Poi(\lambda)$ random variable for some $\lambda\geq 1$ and $k\geq 2$ is an integer, then
\begin{equation}\label{eqn: SBMA Poisson moment bounds}
\E \Pi^k\leq C\lambda^k \quad \text{and} \quad \E \big(\Pi-\E \Pi)^k\leq C\lambda^{\lfloor k/2\rfloor}
\end{equation}
for some constant $C>0$ that depends only on $k$. 
\end{lemma}

\begin{proof}
Denote by $\smash{\left\{\genfrac{}{}{0pt}{}{k}{j}\right\}}$ the number of ways to partition a $k$ element set into $j$ non-empty subsets. In combinatorics, such numbers are known as Stirling numbers of the second kind, and they have the property that for any integer $m\geq 0$,
\begin{equation}\label{SBMA Stirling numbers}
m^k=\sum_{ j\leq k} \left\{\genfrac{}{}{0pt}{}{k}{j}\right\}(m)_j,
\end{equation}
where $(m)_j=m(m-1)\cdots (m-j+1)$ is the falling factorial. The basic properties of the Poisson distribution imply that
\begin{align*}
\E \Pi^k&=\sum_{m\geq 0}\sum_{j \leq k} \left\{\genfrac{}{}{0pt}{}{k}{j}\right\}(m)_j\frac{\lambda^m}{m!}\exp(-\lambda)=\sum_{j \leq k} \left\{\genfrac{}{}{0pt}{}{k}{j}\right\}\lambda^j\sum_{m\geq j}\frac{\lambda^{m-j}}{(m-j)!}\exp(-\lambda)=\sum_{j\leq k} \left\{\genfrac{}{}{0pt}{}{k}{j}\right\}\lambda^j\\
&\leq \max(1,\lambda^k)B_k,
\end{align*}
where $B_k$ denotes the $k$'th Bell number. This establishes the first bound in \eqref{eqn: SBMA Poisson moment bounds}. We now prove by induction that for each $k\geq 2$, the function $M_k(\lambda)=\E(\Pi-\E\Pi)^k$ is a polynomial of degree $\lfloor k/2\rfloor$. The base case holds since $M_2(\lambda)=\Var \Pi=\lambda$, so assume the result holds for all $2\leq i\leq k$. By the product rule
\begin{align*}
M_{k}'(\lambda)&=-\sum_{m\geq 0}k(m-\lambda)^{k-1}\frac{\lambda^m}{m!}\exp(-\lambda)+\sum_{m\geq 0}(m-\lambda)^km\frac{\lambda^{m-1}}{m!}\exp(-\lambda)-M_k(\lambda)\\
&=-kM_{k-1}(\lambda)+\sum_{m\geq 0}(m-\lambda)^k(m-\lambda+\lambda)\frac{\lambda^{m-1}}{m!}\exp(-\lambda)-M_k(\lambda)\\
&=-kM_{k-1}(\lambda)+\frac{1}{\lambda}\big(M_{k+1}(\lambda)+\lambda M_k(\lambda)\big)-M_k(\lambda)\\
&=-kM_{k-1}(\lambda)+\frac{1}{\lambda}M_{k+1}(\lambda).
\end{align*}
Invoking the induction hypothesis shows that $M_{k+1}(\lambda)$ has degree $\max(\lfloor k/2\rfloor, 1+\lfloor (k-1)/2\rfloor)$. This completes the proof.
\end{proof}

\begin{lemma}\label{SBMA binomial moment bounds}
If $X$ is a $\Bin(n,p)$ random variable with $np\geq 1$ and $k\geq 1$ is an integer, then
\begin{equation}
\E X^k\leq C(np)^k
\end{equation}
for some constant $C>0$ that depends only on $k$. 
\end{lemma}

\begin{proof}
Using \eqref{SBMA Stirling numbers} and identifying the probability density function of a $\Bin(n-j,p)$ shows that
\begin{align*}
\E X^k&=\sum_{m \leq n}\sum_{j \leq k} \left\{\genfrac{}{}{0pt}{}{k}{j}\right\}(m)_j\frac{n!}{(n-m)!m!}p^m(1-p)^{n-m}\\
&\leq \sum_{j \leq k}\left\{\genfrac{}{}{0pt}{}{k}{j}\right\}(np)^j \sum_{j\leq m \leq n} \frac{(n-j)!}{(n-m)!(m-j)!}p^{m-j}(1-p)^{n-m}\\
&\leq \max(1,(np)^k)B_k,
\end{align*}
where $B_k$ denotes the $k$'th Bell number. This completes the proof.
\end{proof}

\begin{proposition}\label{SBM free energy concentration}
For any two sequences $(\epsilon_N)$ and $(s_N)$ satisfying \eqref{eqn: SBM epsilon sequence} and \eqref{eqn: SBM s sequence} respectively and every even $p\geq 2$,
\begin{equation}
v_{N,p}\leq \frac{C(1+t^p)}{N^{p/2}}
\end{equation}
for some constant $C>0$ that depends only on $p$, $c$ and $\Delta$.
\end{proposition}

\begin{proof}
To alleviate notation, write $C>0$ for a constant that depends only on $p$, $c$ and $\Delta$ whose value might not be the same at each occurrence. Given a source of randomness $X$, write $\E_X$ for the average with respect to the randomness of $X$. The proof will rely upon the generalized Efron-Stein inequality in \Cref{eqn: SBM generalized Efron-Stein} and the fact that
$$\E=\E_{\sigma^*}\E_Z\E_{\Pi'}\E_{\I_2}\E_{e}\E_{\Pi_t}\E_{\I_1}\E_{\G|\sigma^*}.$$
Introduce the averaged free energy functionals
$$\widehat{F}_N'=\E_{\Pi_t}\E_{\I_1}\E_{\G|\sigma^*}F_N'\qquad \text{and}\qquad \widetilde{F}_N'=\E_Z\E_{\Pi'}\E_{\I_2}\E_{e}\widehat{F}_N'$$
in such a way that
\begin{equation}\label{eqn: SBMA free energy concentration key}
\E\big(F_N'-\overline{F}_N'\big)^p\leq C\Big(\E\big(F_N'-\widehat{F}_N'\big)^p+\E\big(\widehat{F}_N'-\widetilde{F}_N'\big)^p+\E\big(\widetilde{F}_N'-\overline{F}_N'\big)^p\Big).
\end{equation}
We will now bound each of these terms separately.\\
\step{1: proving $\E\big(F_N'-\widehat{F}_N'\big)^p=\BigO\big((t/N)^{p/2}\big)$.}\\
We decompose this further into
\begin{align}\label{eqn: SBMA step 1 key}
\E\big(F_N'-\widehat{F}_N'\big)^p\leq C\Big( \E\big(F_N'-\E_{\G|\sigma^*}F_N'\big)^p&+\E\big(\E_{\G|\sigma^*}F_N'-\E_{\I_1}\E_{\G|\sigma^*}F_N'\big)^p\notag\\
&+\E\big(\E_{\I_1}\E_{\G|\sigma^*}F_N'-\widehat{F}_N'\big)^p\Big)=C\big(I+II+III\big),
\end{align}
and proceed to bound $I$, $II$ and $III$ individually. By the generalized Efron-Stein inequality
$$I\leq C\E\bigg\lvert \sum_{\ell\leq \Pi_t}\E_{\G^{(\ell)}|\sigma^*}\Big(F_N'(\G)-F_N'\big(\G^{(\ell)}\big)\Big)^2\bigg\rvert^{p/2},$$
where $\smash{\big(\tG^{(\ell)}_{i,j}\big)_{i,j\in \N}}$ is an independent copy of $\smash{\big(G^{(\ell)}_{i,j}\big)_{i,j\in \N}}$. Since $\abs{\Delta}< c$ and all spin configuration coordinates are bounded by one,
\begin{align*}
\big\vert H_N'(\G)-H_N'\big(\G^{(\ell)}\big)\big\rvert\leq \abs{\log(c+&\Delta \sigma_{i_\ell}\sigma_{j_\ell})}\big\lvert G^\ell_{i_\ell,j_\ell}-\widetilde{G}_{i_\ell,j_\ell}^\ell\big\rvert\\
&+\Big\lvert \log\Big(1-\frac{c+\Delta \sigma_{i_\ell}\sigma_{j_\ell}}{N}\Big)\Big\rvert\big\lvert G_{i_\ell,j_\ell}^\ell-\widetilde{G}_{i_\ell,j_\ell}^\ell\big\rvert\leq C\big\lvert G^\ell_{i_\ell,j_\ell}-\widetilde{G}_{i_\ell,j_\ell}^\ell\big\rvert.
\end{align*}
It follows that
\begin{align*}
I&\leq \frac{C}{N^p}\E\Big\lvert \sum_{\ell\leq \Pi_t}\E_{\G^{(\ell)}|\sigma^*}\big\lvert G^\ell_{i_\ell,j_\ell}-\widetilde{G}_{i_\ell,j_\ell}^\ell\big\rvert^2\Big\rvert^{p/2}\leq \frac{C}{N^p}\E\Big\lvert \sum_{\ell\leq \Pi_t}\Big(\Big(1-\frac{2}{N}\Big)G_{i_\ell,j_\ell}^\ell+\frac{1}{N}\Big)\Big\rvert^{p/2}\\
&\leq \frac{C}{N^p}\Big(\E \Big\lvert \sum_{\ell\leq \Pi_t}G_{i_\ell,j_\ell}^\ell\Big\rvert^{p/2}+\frac{1}{N^{p/2}}\E \Pi_t^{p/2}\Big).
\end{align*}
Notice that $\smash{\sum_{\ell\leq \Pi_t}G_{i_\ell,j_\ell}^\ell}$ follows a $\smash{\Bin\Big(\Pi_t,\frac{c+\Delta \sigma_{i_\ell}^*\sigma_{j_\ell}^*}{N}\Big)}$ distribution conditionally on $\Pi_t$. Invoking \Cref{SBMA Poisson moment bounds} and \Cref{SBMA binomial moment bounds} yields
\begin{equation}\label{eqn: SBMA step 1 I}
I\leq \frac{C}{N^{p+\frac{p}{2}}}\E\Pi_t^{p/2}\leq \frac{Ct^{p/2}}{N^{p/2}}.
\end{equation}
Another application of the generalized Efron-Stein inequality gives
$$II\leq \E \bigg\lvert \sum_{\ell\leq \Pi_t}\E_{\I_1^{(\ell)}}\Big(\E_{\G|\sigma^*}F_N'(\I_1)-\E_{\G|\sigma^*}F_N'\big(\I_1^{(\ell)}\big)\Big)^2\bigg\rvert^{p/2},$$
where $\I_1^{(\ell)}$ has an independent copy $\smash{(i_\ell',j_\ell')}$ of $\smash{(i_\ell,j_\ell)}$ at the $\ell$'th coordinate but otherwise coincides with~$\I_1$. Taylor expanding the logarithm and remembering that $\smash{G_{i_\ell,j_\ell}^\ell\in\{0,1\}}$, it is readily verified that
$$\big\lvert H_N'(\I_1)\big\rvert\leq \big\lvert G_{i_\ell,j_\ell}^\ell\big\rvert\abs{\log(c+\Delta \sigma_{i_\ell}\sigma_{j_l})}+\big\lvert 1-G_{i_\ell,j_\ell}^\ell\big\rvert \Big\lvert \log\Big(1-\frac{c+\Delta \sigma_{i_\ell}\sigma_{j_\ell}}{N}\Big)\Big\rvert\leq C\Big(\big\lvert G_{i_\ell,j_\ell}^\ell\big\rvert+\frac{1}{N}\Big).$$
This means that $\smash{\big\lvert H_N'(\I_1)-H_N'\big(\I_1^{(\ell)})\big\rvert\leq C\big(\abs{G_{i_\ell,j_\ell}^\ell}+\abs{G_{i'_\ell,j'_\ell}^\ell}+\frac{1}{N}\big)}$, and therefore
$$\Big(\E_{\G|\sigma^*}F_N'(\I_1)-\E_{\G|\sigma^*}F_N'\big(\I_1^{(\ell)}\big)\Big)^2\leq \frac{C}{N^2}\Big(\E_{\G|\sigma^*}\big\lvert G_{i_\ell,j_\ell}^\ell\big\rvert+\E_{\G|\sigma^*}\big\lvert G_{i_\ell',j_\ell'}^\ell\big\rvert+\frac{1}{N}\Big)^2\leq \frac{C}{N^4}.$$
It follows that
\begin{equation}\label{eqn: SBMA step 1 II}
II\leq \frac{C}{N^{2p}}\E \Pi_t^{p/2}\leq \frac{Ct^{p/2}}{N^p}\leq \frac{Ct^{p/2}}{N^{p/2}}.
\end{equation}
A final application of the generalized Efron-Stein inequality reveals that
$$III\leq \E\Big\lvert \E_{\Pi_t'}\big(\E_{\I_1}\E_{\G|\sigma^*}F_N'(\Pi_t)-\E_{\I_1}\E_{\G|\sigma^*}F_N'(\Pi'_t)\big)^2\Big\rvert^{p/2},$$
where $\Pi_t'$ is an independent copy of $\Pi_t$. Slightly abusing notation and redefining $\Pi_t'$ to be the maximum between $\Pi_t$ and $\Pi_t'$, we see that
\begin{align*}
\big\lvert H_N'(\Pi_t')-H_N'(\Pi_t)\big\rvert&\leq \sum_{\Pi_t\leq k\leq \Pi_t'}\bigg(\big\lvert G_{i_k,j_k}^k\big\rvert\abs{\log(c+\Delta \sigma_{i_k}\sigma_{j_k})}+\big\lvert 1-G_{i_k,j_k}^k\big\rvert\Big\lvert \log\Big(1-\frac{c+\Delta \sigma_{i_k}\sigma_{j_k}}{N}\Big)\Big\rvert\bigg)\\
&\leq C\sum_{\Pi_t\leq k\leq \Pi_t'}\Big(\big\lvert G_{i_k,j_k}^k\big\rvert+\frac{1}{N}\Big).
\end{align*}
It follows that
$$\big\lvert \E_{\I_1}\E_{\G|\sigma^*}F_N'(\Pi_t)-\E_{\I_1}\E_{\G|\sigma^*}F_N'(\Pi'_t)\big\rvert\leq \frac{C}{N} \E_{\I_1}\E_{\G|\sigma^*} \sum_{\Pi_t\leq k\leq\Pi_t'}\Big(\big\lvert G_{i_k,j_k}^k\big\rvert+\frac{1}{N}\Big)\leq \frac{C}{N^2}\abs{\Pi_t'-\Pi_t},$$
and by Jensen's inequality and \Cref{SBMA Poisson moment bounds},
\begin{equation}\label{eqn: SBMA step 1 III}
III\leq \frac{C}{N^{2p}}\E\big\lvert \E_{\Pi_t'}\abs{\Pi_t'-\Pi_t}^2\big\rvert^{p/2}\leq \frac{C}{N^{2p}}\E \lvert \Pi_t-\E \Pi_t\rvert^p\leq \frac{Ct^{p/2}}{N^p}\leq \frac{Ct^{p/2}}{N^{p/2}}.
\end{equation}
Combining \eqref{eqn: SBMA step 1 key}, \eqref{eqn: SBMA step 1 I}, \eqref{eqn: SBMA step 1 II} and \eqref{eqn: SBMA step 1 III} reveals that $\E\big(F_N'-\widehat{F}_N'\big)^p=\BigO\big((t/N)^{p/2}\big)$.\\
\step{2: proving $\E\big(\widehat{F}_N'-\widetilde{F}_N'\big)^p=\BigO\big(N^{-p/2}\big).$}\\
We decompose this further into
\begin{align}\label{eqn: SBMA step 2 key}
\E\big(\widehat{F}_N'-\widetilde{F}_N'\big)^p\leq C\Big(\E\big(\widehat{F}_N'-&\E_e\widehat{F}_N'\big)^p+\E\big(\E_e\widehat{F}_N'-\E_{\I_2}\E_e\widehat{F}_N'\big)^p+\E\big(\E_{\I_2}\E_e\widehat{F}_N'-\E_{\Pi'}\E_{\I_2}\E_{e}\widehat{F}_N'\big)^p\notag\\
&+\E\big(\E_{\Pi'}\E_{\I_2}\E_e\widehat{F}_N'-\widetilde{F}_N'\big)^p\Big)=C\big(I+II+III+IV\big),
\end{align}
and proceed to bound $I$, $II$, $III$ and $IV$ individually.  By the generalized Efron-Stein inequality
$$I\leq\E \bigg\lvert \sum_{k\geq 0}\sum_{j\leq \pi_k}\E_{e^{(jk)}}\Big(\widehat{F}_N'(e)-\widehat{F}_N'\big(e^{(jk)}\big)\Big)^2\bigg\rvert^{p/2},$$
where $e^{(jk)}$ has an independent copy $e_{jk}'$ of $e_{jk}$ at the $jk$'th coordinate but otherwise coincides with~$e$. Since
$$\big\lvert H_N'(e)-H_N'\big(e^{(jk)}\big)\big\rvert\leq \frac{\abs{\lambda_k\sigma_{i_{jk}}}}{\abs{1+\lambda_k\sigma_{i_{jk}}^*}}\big \lvert e_{jk}-e'_{jk}\big\rvert\leq \frac{\lambda_k}{1-\lambda_k}\big\lvert e_{jk}-e'_{jk}\big\rvert,$$
and $\lambda_k\in [2^{-k-1},2^{-k}]$, we have
$$I\leq \frac{C}{N^{p}}\E \Big\lvert \sum_{k\geq 0}\frac{1}{2^{k}}\cdot \frac{1}{2^k}\sum_{j\leq \pi_k}\E_{e_{jk}'}\big\lvert e_{jk}-e'_{jk}\big\rvert^2 \Big\rvert^{p/2}.$$
It follows by two applications of Hölder's inequality and Jensen's inequality that
$$I\leq \frac{C}{N^p}\E \sum_{k\geq 0}\Big(\frac{1}{2^k}\sum_{j\leq \pi_k}\E_{e_{jk}'}\big\lvert e_{jk}-e_{jk}'\big\rvert^2\Big)^{p/2}\leq \frac{C}{N^p}\E\sum_{k\geq 0}\frac{1}{2^{\frac{kp}{2}}}\pi_k^{\frac{p}{2}-1}\sum_{j\leq \pi_k}\E_{e_{jk}'}\big\lvert e_{jk}-e_{jk}'\big\rvert^p.$$
Recalling that $e_{jk}\sim \Exp(1)$ while $\pi_{jk}\sim \Poi(s_N)$ and invoking \Cref{SBMA Poisson moment bounds} gives
\begin{equation}\label{eqn: SBMA step 2 I}
I\leq \frac{C}{N^p}\sum_{k\geq 0}\frac{1}{2^{\frac{kp}{2}}}\E\pi_k^{\frac{p}{2}}\leq C\Big(\frac{s_N}{N^2}\Big)^{p/2}\leq \frac{C}{N^{p/2}}.
\end{equation}
Similarly, by the generalized Efron-Stein inequality,
$$II\leq C\E\bigg\lvert \sum_{k\geq 0}\sum_{j\leq \pi_k}\E_{\I_2^{(jk)}}\Big(\E_e\widehat{F}_N'(\I_2)-\E_e\widehat{F}_N'\big(\I_2^{(jk)}\big)\Big)^2\bigg\rvert^{p/2},$$
where $\I_2^{(jk)}$ has an independent copy $i'_{jk}$ of $i_{jk}$ at the $jk$'th coordinate but otherwise coincides with~$\I_2$. By the mean value theorem,
\begin{align*}
\big\lvert H_N'(\I_2)-H_N'\big(\I^{(jk)}_2\big)\big\rvert &\leq \abs{\log(1+\lambda_k\sigma_{i_{jk}})-\log(1+\lambda_k\sigma_{i_{jk}'})}+\lambda_ke_{jk}\Big\lvert \frac{\sigma_{i_{jk}}}{1+\lambda_k\sigma^*_{i_{jk}}}-\frac{\sigma_{i_{jk}'}}{1+\lambda_k\sigma^*_{i_{jk}'}}\Big\rvert\\
&\leq C\lambda_k(1+e_{jk}).
\end{align*}
It follows once again by two applications of Hölder's inequality and \Cref{SBMA Poisson moment bounds} that
\begin{equation}\label{eqn: SBMA step 2 II}
II\leq \frac{C}{N^p}\E \sum_{k\geq 0}\frac{1}{2^{\frac{kp}{2}}}\pi_k^{\frac{p}{2}-1}\sum_{j\leq \pi_k}(1+e_{jk})^p\leq C\Big(\frac{s_N}{N^2}\Big)^{p/2}\leq \frac{C}{N^{p/2}}.
\end{equation}
Another application of the generalized Efron-Stein inequality yields
$$III\leq C\E\bigg\lvert \sum_{k\geq 0} \E_{\Pi'^{(k)}}\Big(\E_{\I_2}\E_e\widehat{F}_N'(\Pi')-\E_{\I_2}\E_e\widehat{F}_N'\big(\Pi'^{(k)}\big)\Big)^2\bigg\rvert^{p/2},$$
where $\smash{\Pi'^{(k)}}$ has an independent copy $\pi_k'$ of $\pi_k$ at the $k$'th coordinate but otherwise coincides with~$\Pi'$. Slightly abusing notation and redefining $\smash{\Pi'^{(k)}}$ to be the process with the larger $k$'th coordinate, we see that
\begin{align*}
\big\lvert H_N'(\Pi')-H_N'(\Pi'^{(k)})\big\rvert&\leq \sum_{\pi_k\leq j\leq \pi_k'}\Big\lvert \log(1+\lambda_k\sigma_{i_{jk}})-\frac{\lambda_ke_{jk}\sigma_{i_{jk}}}{1+\lambda_k\sigma^*_{i_{jk}}}\Big\rvert\\
&\leq \sum_{\pi_k\leq j\leq \pi_k'}\bigg(\Big\lvert\lambda_k\sigma_{i_{jk}}-\frac{\lambda_ke_{jk}\sigma_{i_{jk}}}{1+\lambda_k\sigma^*_{i_{jk}}}\Big\rvert+C\lambda_k^2\bigg)\\
&\leq \sum_{\pi_k\leq j\leq \pi_k'}\bigg(\lambda_k\Big\lvert 1-\frac{e_{jk}}{1+\lambda_k\sigma^*_{i_{jk}}}\Big\rvert +C\lambda_k^2\bigg)\\
&\leq \lambda_k\sum_{\pi_k\leq j\leq \pi_k'}\bigg(\frac{\abs{1-e_{jk}}+\lambda_k}{1-\lambda_k}+C\lambda_k\bigg).
\end{align*}
It follows by two applications of the Cauchy-Schwarz inequality that
\begin{align*}
III&\leq \frac{C}{N^p}\E \sum_{k\geq 0}\frac{1}{2^{\frac{kp}{2}}}\bigg(\sum_{\pi_k\leq j\leq \pi_k'}\bigg(\frac{\abs{1-e_{jk}}+\lambda_k}{1-\lambda_k}+C\lambda_k\bigg)\bigg)^p\\
&\leq \frac{C}{N^p}\E \sum_{k\geq 0}\frac{1}{2^{\frac{kp}{2}}}\abs{\pi_k-\pi_k'}^{p-1}\sum_{\pi_k\leq j\leq \pi_k'}\bigg(\frac{\abs{1-e_{jk}}+\lambda_k}{1-\lambda_k}+C\lambda_k\bigg)^p\\
&\leq \frac{C}{N^p} \sum_{k\geq 0}\frac{1}{2^{\frac{kp}{2}}}\E \abs{\pi_k-\pi_k'}^p.
\end{align*}
Since $\E\abs{\pi_k-\pi_k'}^p\leq C\E\abs{\pi_k-\E\pi_k}^p\leq s_N^{p/2}$ by \Cref{SBMA Poisson moment bounds}, this implies that
\begin{equation}\label{eqn: SBMA step 2 III}
III\leq C\Big(\frac{s_N}{N^2}\Big)^{p/2}\leq \frac{C}{N^{p/2}}.
\end{equation}
A final application of the generalized Efron-Stein inequality gives
$$IV\leq C\E \bigg\lvert \sum_{i\leq N}\E_{Z^{(i)}}\Big(\E_{\Pi'}\E_{\I_2}\E_e\widehat{F}_N'(Z)-\E_{\Pi'}\E_{\I_2}\E_e\widehat{F}_N'\big(Z^{(i)}\big)\Big)^2\bigg\rvert^{p/2},$$
where $Z^{(i)}$ has an independent copy $Z_{i,0}'$ of $Z_{i,0}$ at the $i$'th coordinate but otherwise coincides with~$Z$. Combining Hölder's inequality with the bound
$$\big\lvert H_N'(Z)-H_N'\big(Z^{(i)}\big)\big\rvert\leq \sqrt{\lambda_0\epsilon_N}\big\lvert Z_{i,0}-Z_{i,0}'\big\rvert$$
reveals that
\begin{equation}\label{eqn: SBMA step 2 IV}
IV\leq C\Big(\frac{\lambda_0 \epsilon_N}{N}\Big)^p\E \Big\lvert \sum_{i\leq N}\E_{Z^{(i)}}\big\lvert Z_{i,0}-Z_{i,0}'\big\rvert^2\Big\rvert^{p/2}\leq \frac{C}{N^p}N^{\frac{p}{2}-1}\sum_{i\leq N}\E \big\lvert Z_{i,0}-Z_{i,0}'\big\rvert^p\leq \frac{C}{N^{p/2}}.
\end{equation}
Together with \eqref{eqn: SBMA step 2 key}, \eqref{eqn: SBMA step 2 I}, \eqref{eqn: SBMA step 2 II} and \eqref{eqn: SBMA step 2 III}, this shows that $\E\big(\widehat{F}_N'-\widetilde{F}_N'\big)^p=\BigO\big(N^{-p/2}\big)$.\\
\step{3: proving $\E\big(\widetilde{F}_N'-\overline{F}_N'\big)^p=\BigO\big((t^2/N)^{p/2}\big)$}\\
Controlling the final term in \eqref{eqn: SBMA free energy concentration key} requires more care since $\widetilde{F}_N'$ depends on $\sigma^*$ both through $F_N'$ and through the conditional expectation $\E_{\G|\sigma^*}$. To simplify notation, write $\E'=\E_Z\E_{\Pi'}\E_{\I_2}\E_{e}\E_{\Pi_t}\E_{\I_1}$ in such a way that by the generalized Efron-Stein inequality
\begin{align}\label{eqn: SBMA step 3 key}
\E\big(\widetilde{F}_N'-\overline{F}_N'\big)^p&\leq C\E\bigg\lvert \sum_{\ell\leq N}\E_{\sigma^{*,(\ell)}}\Big(\E'\E_{\G|\sigma^*}F_N'(\sigma^*)-\E'\E_{\G|\sigma^{*,(\ell)}}F_N'\big(\sigma^{*,(\ell)}\big)\Big)^2\bigg\rvert^{p/2}\notag\\
&\leq C\E\bigg\lvert \sum_{\ell\leq N}\E_{\sigma^{*,(\ell)}}\Big(\E'\E_{\G|\sigma^*}F_N'(\sigma^*)-\E'\E_{\G|\sigma^{*}}F_N'\big(\sigma^{*,(\ell)}\big)\Big)^2\bigg\rvert^{p/2}\notag\\
&\quad +C\E\bigg\lvert \sum_{\ell\leq N}\E_{\sigma^{*,(\ell)}}\Big(\E'\E_{\G|\sigma^*}F_N'\big(\sigma^{*,(\ell)}\big)-\E'\E_{\G|\sigma^{*,(\ell)}}F_N'\big(\sigma^{*,(\ell)}\big)\Big)^2\bigg\rvert^{p/2}\notag\\
&=C(I+II),
\end{align}
where $\sigma^{*,(\ell)}$ has an independent copy $\widetilde{\sigma}^*_\ell$ of $\sigma^*_\ell$ at the $\ell$'th coordinate but otherwise coincides with~$\sigma^*$. Since
\begin{align*}
\big\lvert H_N'(\sigma^*)-H_N'(\sigma^{*,(l)})\big\rvert&\leq \lambda_0\epsilon_N\abs{\sigma_i}\big\lvert \sigma^*_\ell-\widetilde{\sigma}_\ell^*\big\rvert+\sum_{k\geq 1}\sum_{j:i_{jk}=\ell}\bigg\lvert \frac{\lambda_ke_{jk}\sigma_\ell}{1+\lambda_k\sigma^*_\ell}-\frac{\lambda_ke_{jk}\sigma_\ell}{1+\lambda_k\widetilde{\sigma}^*_\ell}\bigg\rvert\\
&\leq 2 \epsilon_N+\sum_{k\geq 1}\sum_{j:i_{jk}=\ell}\frac{2\lambda_k^2e_{jk}}{(1-\lambda_k)^2},
\end{align*}
and $\E e_{jk}=1$, the Fubini-Tonelli theorem and the basic properties of the multinomial distribution imply that
\begin{align*}
\big\lvert \E'\E_{\G|\sigma^*}F_N'(\sigma^*)-\E'\E_{\G|\sigma^*}F_N'\big(\sigma^{*,(\ell)}\big)\big\rvert&\leq \frac{2\epsilon_N}{N}+\frac{1}{N}\E_{\Pi'}\sum_{k\geq 1}\frac{2\lambda_k^2}{(1-\lambda_k)^2}\E_{\I_2}\abs{\{j:i_k=\ell\}}\\
&\leq \frac{2}{N}+\frac{1}{N^2}\E_{\Pi'}\sum_{k\geq 1}\frac{2\lambda_k^2}{(1-\lambda_k)^2}\pi_k\leq \frac{2}{N}+\frac{s_N}{N^2}\leq \frac{3}{N}.
\end{align*}
It follows that
\begin{equation}\label{eqn: SBMA step 3 I}
I\leq C\Big(\frac{N}{N^2}\Big)^{p/2}=\frac{C}{N^{p/2}}.
\end{equation}
To bound $II$ we will use an interpolation argument. Fix $1\leq \ell\leq N$ and condition on all sources of randomness other than $\G$. For each $u\in [0,1]$, $G\in \{0,1\}^{\Pi_t}$ and $k\leq \Pi_t$, let
\begin{align*}
P_u^{1,k}(G)&=G_k\bigg(\frac{c+\Delta\sigma_\ell^{*,u}\sigma^*_{j_k}}{N}\bigg)+(1-G_{k})\bigg(1-\frac{c+\Delta\sigma_{\ell}^{*,u}\sigma^*_{j_k}}{N}\bigg),\\
P_u^{2,k}(G)&=G_k\bigg(\frac{c+\Delta\sigma_{\ell}^{*,u}\sigma^*_{i_k}}{N}\bigg)+(1-G_{k})\bigg(1-\frac{c+\Delta\sigma_{\ell}^{*,u}\sigma^*_{i_k}}{N}\bigg),\\
P_u^{3}(G)&=G_{k}\bigg(\frac{c+\Delta(\sigma_{\ell}^{*,u})^2}{N}\bigg)+(1-G_{k})\bigg(1-\frac{c+\Delta(\sigma_{\ell}^{*,u})^2}{N}\bigg),
\end{align*}
where $\sigma_{\ell}^{*,u}=(1-u)\sigma^*_{\ell}+u\widetilde{\sigma}^*_{\ell}$.
Write $\sigma^{*,u}$ for the configuration with $\ell$'th coordinate $\sigma_\ell^{*,u}$ which otherwise coincides with~$\sigma^*$, and introduce the sets
\begin{alignat*}{2}
\I_1^1&=\{k\mid i_k=l \text{ and } j_k\neq l\}& \qquad\qquad\I_1^3&=\{k\mid i_k=j_k=l\}\\
\I_1^2&=\{k\mid i_k\neq l \text{ and } j_k=l\}& \qquad\qquad\I_1^4&=\{k\mid  i_k\neq l\neq j_k\}.
\end{alignat*}
Denote $\widetilde{\G}=(\G_k)_{k\in \I_1^4}$, $\G^{(k)}=\G\setminus \G_k$, $\widetilde{G}=(G_k)_{k\in \I_1^4}$ and $G^{(k)}=G\setminus G_k$. Define the interpolating free energy
$$\p(u)=\sum_{G\in \{0,1\}^{\Pi_t}}F_N'\big(\sigma^{*,(\ell)},G\big)\P\big\{\widetilde{\G}=\widetilde{G}|\sigma^*\big\}\cdot \prod_{k\in \I_1^1}P_u^{1,k}(G)\prod_{k\in \I_1^2}P_u^{2,k}(G)\prod_{k\in \I_1^3}P_u^{3,k}(G)$$
in such a way that $\p(1)=\E_{\G|\sigma^*}F_N'\big(\sigma^{*,(\ell)}\big)$ and $\p(0)=\E_{\G|\sigma^{*,(\ell)}}F_N'\big(\sigma^{*,(\ell)}\big)$. By the product rule,
\begin{align*}
\p'(u)&=\sum_{k\in \cup_{i\leq 3}\I_1^i}\sum_{G\in \{0,1\}^{\Pi_t}}F_N'\big(\sigma^{*,(\ell)},G\big)\P\big\{\G^{(k)}=G^{(k)}|\sigma^{*,u}\big\}(2G_k-1)\\
&\qquad\qquad\qquad\qquad\qquad\Bigg(\frac{\Delta(\widetilde{\sigma}_\ell^*-\sigma^*_\ell)(\sigma_{j_k}^*\1_{\{k\in \I_1^1\}}+\sigma_{i_k}^*\1_{\{k\in \I_1^2\}}+2\sigma_\ell^{*,u}\1_{\{k\in \I_1^3\}})}{N}\Bigg)\\
&=\sum_{k\in \I_1^1}D_k\frac{\Delta(\widetilde{\sigma}_\ell^*-\sigma^*_\ell)\sigma_{j_k}^*}{N}+\sum_{k\in \I_1^2}D_k\frac{\Delta(\widetilde{\sigma}_\ell^*-\sigma^*_\ell)\sigma_{i_k}^*}{N}\\
&\qquad\qquad\qquad\qquad\qquad\qquad\qquad\qquad+\sum_{k\in \I_1^3}D_k\frac{2\Delta\big((1-u)\sigma_\ell^*+u\widetilde{\sigma}_\ell^*\big)(\widetilde{\sigma}_\ell^*-\sigma_\ell^*)}{N}
\end{align*}
for $D_k=\E_{\G^{(k)}|\sigma^{*,u}}F_N'\big(\sigma^{*,(\ell)},\G^{(k)},\G_k=1\big)-\E_{\G^{(k)}|\sigma^{*,u}}F_N'\big(\sigma^{*,(\ell)},\G^{(k)},\G_k=0\big)$. Since
$$\big\lvert H_N'\big(\G^{(k)}, \G_k=1\big)-H_N'\big(\G^{(k)},\G_k=0\big)\big\rvert\leq \Big\rvert\log(c+\Delta \sigma_{i_k}\sigma_{j_k})-\log\Big(1-\frac{c+\Delta \sigma_{i_k}\sigma_{j_k}}{N}\Big)\Big\lvert\leq C,$$
we have $\abs{D_k}\leq \frac{C}{N}$, so the fundamental theorem of calculus yields
$$\big\lvert \E_{\G|\sigma^*}F_N'\big(\sigma^{*,(\ell)}\big)-\E_{\G|\sigma^{*,(\ell)}}F_N'\big(\sigma^{*,(l)}\big)\big\rvert\leq \sup_{u\in (0,1)}\abs{\p'(u)}\leq \frac{C}{N^2}\big(\abs{\I_1^1}+\abs{\I_1^2}+\abs{\I_1^3}\big).$$
It follows by the basic properties of the binomial distribution that
\begin{align*}
\Big(\E'\E_{\G|\sigma^*}F_N'\big(\sigma^{*,(\ell)}\big)-\E'\E_{\G|\sigma^{*,(\ell)}}F_N'\big(\sigma^{*,(\ell)}\big)\Big)^2&\leq \frac{C}{N^4}\sum_{1 \leq i \leq 3} \big(\E_{\Pi_t}\abs{\I_1^i}\big)^2=\frac{C}{N^4}\sum_{1\leq i \leq 3} \Big(\E_{\Pi_t}\frac{\Pi_t}{N}\Big)^2\\
&\leq \frac{Ct^2}{N^2},
\end{align*}
and thus
\begin{equation}\label{eqn: SBMA step 3 II}
II\leq C\bigg(\frac{N t^2}{N^2}\bigg)^{p/2}=\frac{Ct^p}{N^{p/2}}.
\end{equation}
Combining \eqref{eqn: SBMA step 3 key}, \eqref{eqn: SBMA step 3 I} and \eqref{eqn: SBMA step 3 II} reveals that $\E\big(\widetilde{F}_N'-\overline{F}_N'\big)^p=\BigO\big((t^2/N)^{p/2}\big)$. Together with \eqref{eqn: SBMA free energy concentration key}, step $1$ and step $2$ this completes the proof.
\end{proof}

\section{Multi-overlap concentration}\label{SBM app multioverlap concentration}
In this appendix we prove a finitary version of the main result in \cite{BarMOC} regarding the concentration of the multi-overlaps \eqref{eqn: SBM enriched multi-overlap}. Instead of focusing on the stochastic block model \eqref{eqn: SBM Hamiltonian} or \eqref{eqn: SBM enriched Hamiltonian}, we will work in the general setting of optimal Bayesian inference; this presents no additional difficulty, and we suspect that our restatement of the multi-overlap concentration result in \cite{BarMOC} will be useful for the analysis of other statistical inference models.

Let us describe a general optimal Bayesian inference model following \cite{BarMOC}. We consider a ground-truth signal $\sigma^*\in \Sigma_N=\{-1,+1\}^N$ with independent coordinates generated from a prior distribution $P^*$,
\begin{equation}
\sigma^*\sim P^*=\prod_{i\leq N}P_i^*.    
\end{equation}
We suppose that the data $\D=\D(\sigma^*)$ is sampled conditionally on the unknown signal $\sigma^*$ from a probability distribution $P_{\mathrm{out}}$,
\begin{equation}
\D\sim P_{\text{out}}\{\cdot|\sigma^*\}.
\end{equation}
The inference task is to recover the signal $\sigma^*$ as accurately as possible given the data $\D$ under the assumption that the likelihood $P_{\text{out}}$ and the prior $P^*$ are known to the statistician. In this setting the posterior of the model can be written explicitly. Indeed, if
\begin{equation}\label{eqn: SBMA log-likelihood}
H_N(\sigma)=\log P_{\text{out}}\big\{\D|\sigma^*=\sigma\big\}
\end{equation}
denotes the Hamiltonian or log-likelihood of the model, then the Gibbs measure or posterior distribution of the model is given by Bayes' formula,
\begin{equation}\label{eqn: SBMA generic Gibbs measure}
G_N(\ud \sigma)=\P\{\sigma^*\in \ud \sigma|\D\}=\frac{\exp H_N(\sigma) P^*(\ud \sigma)}{\int \exp H_N(\tau)P^*(\ud \tau)}.
\end{equation}
In addition to working in the context of optimal Bayesian inference, we will assume that the Hamiltonian \eqref{eqn: SBMA log-likelihood} satisfies symmetry between sites. This means that for any permutation $\rho$ of the spin indices,
\begin{equation}\label{eqn: SBMA symmetry between sites}
\P\{\sigma^*\in \ud \sigma|\D\}\stackrel{d}{=}\P\{\rho(\sigma^*)\in \ud \sigma|\D\}.
\end{equation}
Notice that both the stochastic block model \eqref{eqn: SBM Hamiltonian} and its enriched version \eqref{eqn: SBM enriched Hamiltonian} fall into the setting we have just described. For instance, in the enriched stochastic block model we have $P_i^*=\Ber(p)$ and $\smash{\D=\widetilde{\D}^{t,\mu[-1,1],\bar \mu}}$, where the data $\smash{\widetilde{\D}^{t,s,\mu}}$ was defined in \eqref{eqn: SBM enriched data}.

The concentration of the multi-overlaps \eqref{eqn: SBM enriched multi-overlap} associated with the Hamiltonian \eqref{eqn: SBMA log-likelihood} will be enforced through a small perturbation which will not affect the limit of the associated free energy \eqref{eqn: SBM free energy}. Recall the choice of the sequences $(\epsilon_N)$ and $(s_N)$ satisfying \eqref{eqn: SBM epsilon sequence} and \eqref{eqn: SBM s sequence}, respectively. Fix an integer $K_+\geq 1$, and for each perturbation parameter $\lambda \in \R^{1+K_+}$ with $\lambda_k\in [2^{-k-1},2^{-k}]$ for $0\leq k\leq K_+$ recall the definition of the perturbation Hamiltonians \eqref{eqn: SBM gaussian perturbation} and \eqref{eqn: SBM exponential perturbation}. Introduce the perturbed Hamiltonian 
\begin{equation}\label{eqn: SBMA perturbed generic Hamiltonian}
H_N(\sigma,\lambda)=H_N(\sigma)+H_N^{\mathrm{gauss}}(\sigma,\lambda_0)+H_N^{\mathrm{exp}}(\sigma,\lambda),
\end{equation}
where the randomness of each Hamiltonian is independent of the randomness of the other Hamiltonians. The multi-overlaps associated with this perturbed Hamiltonian are defined as in \eqref{eqn: SBM enriched multi-overlap},
\begin{equation}\label{eqn: SBMA generic multi-overlaps}
R_{\ell_1,\ldots,\ell_n}=\frac{1}{N}\sum_{i\leq N}\sigma_i^{\ell_1}\cdots \sigma_i^{\ell_n},
\end{equation}
where $(\sigma^\ell)$ denotes a sequence of i.i.d.\@ replicas sampled from the Gibbs measure $\langle \cdot \rangle$ associated with the perturbed Hamiltonian \eqref{eqn: SBMA perturbed generic Hamiltonian}. It is actually these multi-overlaps that will be shown to concentrate. The proof of \Cref{SBM perturbation effect on free energy} shows that the free energy functionals associated with the Hamiltonians \eqref{eqn: SBMA log-likelihood} and \eqref{eqn: SBMA perturbed generic Hamiltonian} are asymptotically equivalent; for our purposes these two Hamiltonians can therefore be thought to describe the same model. 

To establish the concentration of the multi-overlaps \eqref{eqn: SBMA generic multi-overlaps} we will closely follow the arguments in \cite{BarMOC}. The authors in \cite{BarMOC} obtain the concentration of the multi-overlaps \eqref{eqn: SBMA generic multi-overlaps} for some perturbation parameter $\lambda$ by showing that it holds on average over the set of admissible perturbation parameters. In the proof of \Cref{SBM main result} we will need to be able to obtain multi-overlap concentration for a specific perturbation parameter. Following the strategy in \cite{BarMOC}, we will propose a verifiable condition on a perturbation parameter $\lambda$ which ensures the concentration of its associated multi-overlaps. To be more precise, we will obtain the concentration of the multi-overlaps \eqref{eqn: SBMA generic multi-overlaps} up to a small error for any sequence of perturbation parameters $(\lambda^N)$ with
\begin{equation}\label{eqn: SBMA limits to get FDS}
\lim_{N\to \infty}\E\big\langle (\L_k-\E\langle \L_k\rangle)^2\big\rangle=0.
\end{equation}
The quantities $\L_k$ are defined in \eqref{eqn: SBM lambda 0 derivative of Hamiltonian} and \eqref{eqn: SBM lambda k derivative of Hamiltonian} for $0\leq k\leq K_+$, and, through a slight abuse of notation, we have written $\smash{\langle \cdot \rangle}$ for the Gibbs average with respect to the perturbed Hamiltonian \eqref{eqn: SBMA perturbed generic Hamiltonian} associated with the perturbation parameters $\smash{(\lambda^N)}$. If necessary, we will write $\langle \cdot \rangle_N$ to emphasize the dependence of this Gibbs measure on $N$.

We begin by showing that \eqref{eqn: SBMA limits to get FDS} implies the concentration of the magnetization $R_1$ and of the overlap $R_{1,2}$; the former will be immediate from the Nishimori identity \eqref{eqn: SBM Nishimori identity} while the latter will follow from a standard application of the Gaussian integration by parts formula (see Lemma 1.4 in~\cite{PanSKB}).

\begin{lemma}\label{SBMA concentration of R1}
For any integer $N\geq 1$,
\begin{equation}
\E\big\langle (R_1-\E\langle R_1\rangle)^2\big\rangle\leq \frac{1}{N}.
\end{equation}
\end{lemma}

\begin{proof}
Applying the Nishimori identity reveals that
$$\E\big\langle (R_1-\E\langle R_1\rangle)^2\big\rangle=\frac{1}{N^2}\sum_{i\leq N}\E\big(\sigma_i^*-\E \sigma_i^*\big)^2\leq \frac{1}{N}.$$
This completes the proof.
\end{proof}

\begin{lemma}\label{SBMA concentration of R2}
For any integer $N\geq 1$,
\begin{equation}
\E\big\langle (R_{1,2}-\E\langle R_{1,2}\rangle)^2\big \rangle\leq 4\E\big\langle (\L_0-\E\langle \L_0\rangle)^2\big\rangle.
\end{equation}
\end{lemma}

\begin{proof}
The proof is taken from the Appendix of \cite{BarMOC}, and it consists in testing the concentration of the overlap $\smash{R_{1,*}=\frac{\sigma\cdot \sigma^*}{N}}$ against the Hamiltonian $\L_0$ by means of the Gaussian integration by parts formula. Using \eqref{eqn: SBM lambda 0 derivative of Hamiltonian} shows that
\begin{align}\label{SBMA concentration of R2 covariance}
\E\big\langle (R_{1,*}-\E\langle R_{1,*}\rangle)(\L_0-\E\langle \L_0\rangle)\big\rangle
=\E\big\langle R_{1,*}(&R_{1,*}-\E\langle R_{1,*}\rangle)\big\rangle\notag\\
&+\frac{1}{2N\sqrt{\lambda_{0,N}}}\E\big\langle R_{1,*}(\sigma\cdot Z_0-\E\langle \sigma\cdot Z_0\rangle)\big\rangle.
\end{align}
The Gaussian integration by parts formula and the Nishimori identity imply that
$$\E\langle \sigma\cdot Z\rangle=N\sqrt{\lambda_{0,N}}\big(1-\E\langle R_{1,*}\rangle\big)\quad \text{and} \quad \E\big\langle R_{1,*}\sigma\cdot Z\big\rangle=N\sqrt{\lambda_{0,N}}\big(\E\langle R_{1,*}\rangle -\E\langle R_{1,*}\rangle^2\big).$$
Substituting these two equalities into \eqref{SBMA concentration of R2 covariance} reveals that
$$\E\big\langle (R_{1,*}-\E\langle R_{1,*}\rangle)(\L_0-\E\langle \L_0\rangle)\big\rangle=\frac{1}{2}\E\big\langle (R_{1,*}-\langle R_{1,*}\rangle)^2 \big\rangle+\frac{1}{2}\E\big\langle(R_{1,*}-\E\langle R_{1,*}\rangle)^2\big\rangle.$$
(There seems to be a sign error in equation (5.2) of \cite{BarMOC}.)
It follows by the Nishimori identity that
$$\E\big\langle (R_{1,*}-\E\langle R_{1,*}\rangle)(\L_0-\E\langle \L_0\rangle)\big\rangle\geq \frac{1}{2}\E\big\langle(R_{1,2}-\E\langle R_{1,2}\rangle)^2\big\rangle.$$
Invoking the Cauchy-Schwarz inequality and the Nishimori identity completes the proof.
\end{proof}

The concentration of the multi-overlaps \eqref{eqn: SBMA generic multi-overlaps} is considerably more complicated to obtain, and follows from the Franz-de Sanctis identities described in \cite{BarMOC}. The first section of this appendix will be devoted to establishing this implication. In the second section we will prove a finitary version of the multi-overlap concentration result in \cite{BarMOC} which will be uniform over an appropriate class of random probability measures. This uniformity plays its part in the proof \Cref{SBM free energy approximate solution at contact point}.

\subsection{Franz-de Sanctis identities}

The Franz-de Sanctis identities may be thought of as the Ghirlanda-Guerra identities of optimal Bayesian inference. A random probability measure which satisfies the Ghirlanda-Guerra identities must have an ultrametric support; a deep insight which leads to the appearance of the intricate Poisson-Dirichlet probability cascades in many spin glass models \cite{PanSKB}. The Franz-de Sanctis identities enforce a much simpler and more rigid structure on a random probability measure which we will describe in due course. To state these identities, it will be convenient to fix a uniform index $i\in \{1,\ldots,N\}$ and an exponential random variable $e\sim \Exp(1)$ independent of all other sources of randomness, and introduce the random variables
\begin{equation}\label{eqn: SBMA Franz de Sanctis variables}
y_{ik}=\frac{e}{1+\lambda_k\sigma_i^*}, \quad \theta_{ik}^\ell=\log(1+\lambda_k\sigma_i^\ell)-\lambda_ky_{ik}\sigma_i^\ell,\quad \text{and} \quad d_{ik}^\ell=\frac{y_{ik}\sigma_i^\ell}{1+\lambda_k\sigma_i^*}
\end{equation}
for $1\leq k\leq K_+$.

\begin{proposition}[Franz-de Sanctis identities in inference]\label{SBMA Franz de Sanctis}
For any $1\leq k\leq K_+$ and any function $f_n$ of finitely many spins on $n$ replicas and of the signal $\sigma^*$ with $\norm{f_n}_{L^\infty}\leq 1$,
\begin{equation}
\bigg\lvert \E \frac{\langle f_nd_{ik}^1\exp\big(\sum_{\ell\leq n}\theta_{ik}^\ell\big)\rangle}{\langle \exp(\theta_{ik})}\rangle^n-\E\langle f_n\rangle \E\frac{\langle d_{ik}\exp(\theta_{ik})\rangle}{\langle \exp(\theta_{ik})\rangle}\bigg\rvert\leq \bigg(2\E\big\langle( \L_k-\E\langle \L_k\rangle)^2\big\rangle+\frac{16}{s_N}\bigg)^{1/2}.
\end{equation}
\end{proposition}

\noindent The Franz-de Sanctis identities are proved in a similar fashion to the Ghirlanda-Guerra identities by testing the concentration of the quantities
\begin{equation}
\widetilde{\L}_k=\frac{1}{s_N}\sum_{j\leq \pi_k}\frac{\sigma_{i_{jk}}e_{jk}}{(1+\lambda_k\sigma^*_{i_{jk}})^2}
\end{equation}
defined for $1\leq k\leq K_+$ against an arbitrary function of finitely many spins and of the signal $\sigma^*$. Notice that  $\smash{\widetilde{\L}_k}$ is none other than the second term in the sum defining each $\L_k$ in \eqref{eqn: SBM lambda k derivative of Hamiltonian}. The reason for focusing only on the second term is that the first term concentrates automatically by the Nishimori identity. This is the content of Proposition 3.4 in \cite{BarMOC} which we reproduce here for completeness. We present a slightly simpler proof than that in \cite{BarMOC} which was kindly shared with us by Dmitry Panchenko.

\begin{lemma}\label{SBMA L tilde concentration}
For any $1\leq k\leq K_+$ and every large enough $N\geq 1$,
\begin{equation}
\E\Big\langle \big(\widetilde{\L}_k-\E\langle \widetilde{\L}_k\rangle\big)^2\Big\rangle\leq 2\E\big(\langle \L_k-\E\langle \L_k\rangle)^2\big\rangle+\frac{16}{s_N}.
\end{equation}
\end{lemma}

\begin{proof}
Introduce the quantity $$g(\sigma,\pi_k)=\sum_{j\leq \pi_k}\frac{ \sigma_{i_{jk}}}{1+\lambda_k\sigma_{i_{jk}}}$$
in such a way that $\smash{\widetilde{\L}_k=s_N^{-1}g(\sigma,\pi_k)}-\L_k$. Write $\Var$ for the variance with respect to the measure $\E\langle\cdot \rangle$. Since the variance of a sum of two random variables is bounded by twice the sum of the variance of each of the random variables,
\begin{equation}\label{eqn: SBMA L tilde concentration key}
\Var \big(\widetilde{\L}_k\big)\leq 2 \Big(\Var(\L_k)+\frac{1}{s_N^2}\Var(g)\Big).
\end{equation}
By the Nishimori identity and a direct computation,
\begin{equation}\label{eqn: SBMA L tilde concentration Var g decomposition}
\Var\big(g(\sigma,\pi_k)\big)=\E\bigg(\sum_{j\leq \pi_k}\frac{\sigma^*_{i_{jk}}}{1+\lambda_k\sigma^*_{i_{jk}}}\bigg)^2-\bigg(\E\sum_{j\leq \pi_k}\frac{\sigma^*_{i_{jk}}}{1+\lambda_k\sigma^*_{i_{jk}}}\bigg)^2.
\end{equation}
Recalling that the coordinates of the signal $\sigma^*$ are i.i.d.\@ and averaging with respect to the randomness of the indices $(i_{jk})$ reveals that
\begin{align*}
\E\bigg(\sum_{j\leq \pi_k}\frac{\sigma^*_{i_{jk}}}{1+\lambda_k\sigma^*_{i_{jk}}}\bigg)^2&=\frac{1}{N}\E\sum_{j,j'\leq \pi_k}\frac{1}{(1+\lambda_k\sigma_1^*)^2}+\frac{N^2-N}{N}\E\sum_{j,j'\leq \pi_k}\frac{\sigma_1^*\sigma_2^*}{(1+\lambda_k \sigma_1^*)(1+\lambda_k\sigma_2^*)}\\
&\leq \frac{4\E\pi_k^2}{N}+\E\pi_k^2\Big(\E\frac{\sigma_1^*}{1+\lambda_k\sigma_1^*}\Big)^2,
\end{align*}
where we have used that $\sigma_1^2=1$ and $\lambda_k\leq 1/2$. Similarly,
$$\bigg(\E\sum_{j\leq \pi_k}\frac{\sigma^*_{i_{jk}}}{1+\lambda_k\sigma^*_{i_{jk}}}\bigg)^2=\big(\E\pi_k\big)^2\Big(\E\frac{\sigma_1^*}{1+\lambda_k\sigma_1^*}\Big)^2.$$
Substituting these two bounds into \eqref{eqn: SBMA L tilde concentration Var g decomposition}, recalling \eqref{eqn: SBM s sequence} and choosing $N$ large enough yields
$$\Var\big(g(\sigma,\pi_k)\big)\leq \frac{4 \E\pi_k^2}{N}+\Var (\pi_k)\Big(\E\frac{\sigma_1^*}{1+\lambda_k\sigma_1^*}\Big)^2\leq 8s_N.$$
Plugging this into \eqref{eqn: SBMA L tilde concentration key} completes the proof.
\end{proof}

\begin{proof}[Proof of \Cref{SBMA Franz de Sanctis}.]
We follow the proof of Theorem 3.3 in \cite{BarMOC}; we will not give full details, and instead encourage the interested reader to consult \cite{BarMOC}. The Cauchy-Schwarz inequality and the fact that $\norm{f_n}_{L^\infty}\leq 1$ imply that
\begin{align*}
\big\lvert\E \big\langle f_n\widetilde{\L}_k(\sigma^1)\big\rangle-\E\langle f_n\rangle \E\big\langle \widetilde{\L}_k(\sigma)\big\rangle\big\rvert&\leq \E\big\langle (f_n-\langle f_n\rangle)^2\big\rangle^{1/2}\E\Big\langle\big( \widetilde{\L}_k-\E\langle \widetilde{\L}_k\rangle\big)^2\Big\rangle^{1/2}\\
&\leq \E\Big\langle\big( \widetilde{\L}_k-\E\langle \widetilde{\L}_k\rangle\big)^2\Big\rangle^{1/2}.
\end{align*}
By \Cref{SBMA L tilde concentration} it is therefore sufficient to prove that
\begin{equation}
\E\big\langle f_n\widetilde{\L}_k(\sigma^1)\big\rangle=\E \frac{\langle f_nd_{ik}^1\exp\big(\sum_{\ell\leq n}\theta_{ik}^\ell\big)\rangle}{\langle \exp(\theta_{ik})\rangle^n} \quad \text{and} \quad \E\big\langle\widetilde{\L}_k(\sigma)\big\rangle=\E\frac{\langle d_{ik}\exp(\theta_{ik})\rangle}{\langle \exp(\theta_{ik})\rangle}.\label{eqn: SBMA Franz de Sanctis goal}
\end{equation}
Since $\pi_k$ is independent of all other sources of randomness, taking the expectation with respect to this random variable first shows that
\begin{equation}\label{eqn: SBMA Franz de Sanctis key}
\E\big\langle f_n\widetilde{\L}_k(\sigma^1)\big\rangle=\sum_{r\geq 1}\frac{s_N^{r-1}}{(r-1)!}\exp(-s_N)\E\big\langle f_nD_{1k}^1\big\rangle_{\pi_k=r},
\end{equation}
where $D_{1k}^1=\sigma_{i_{1k}}^1e_{1k}/(1+\lambda_k\sigma^*_{i_{1k}})^2$.
To simplify this expression, we will isolate the first replica $\sigma^1$ appearing in each of the averages. It will be convenient to introduce the quantities
$$\Theta_{jk}^\ell=\log(1+\lambda_k\sigma_{i_{jk}}^\ell)-\frac{\lambda_ke_{jk}\sigma_{i_{jk}}^\ell}{1+\lambda_k\sigma^*_{i_{jk}}} \quad \text{and} \quad \HH_k^{r-1}(\sigma^\ell)=\sum_{2\leq j\leq r}\Theta_{jk}^\ell$$
as well as the partially perturbed Hamiltonian
\begin{equation}\label{eqn: SBM Franz de Sanctis partially perturbed Hamiltonian}
H_N'(\sigma)=H_N(\sigma)+H_N^{\mathrm{gauss}}(\sigma)+\sum_{\substack{1\leq k'\leq K_+\\k'\neq k}}\HH_{k'},
\end{equation}
where $\HH_k$ is defined in \eqref{eqn: SBM exponential perturbation}. Denoting by $\langle \cdot\rangle_{\pi_k=r}'$ the Gibbs measure corresponding to the Hamiltonian $H_N'(\sigma)+\HH_k^{r-1}(\sigma^\ell)$, we have
\begin{equation}\label{eqn: SBMA Franz de Sanctis distributional identity}
\E\big\langle f_nD_{1k}^1\big\rangle_{\pi_k=r}=\E\E_{i_{1k}}\E_{e_{1k}}\frac{ \langle f_nD_{1k}^1\exp\big(\sum_{\ell\leq n}\Theta^\ell_{1k}\big)\rangle'_{\pi_k=r}}{(\langle \exp(\Theta_{1k})\rangle'_{\pi_k=r})^n}
\end{equation}
for each $r\geq 1$.
Since the uniform random variable $i_{1k}$ and the exponential random variable $e_{1k}$ no longer appear in the Gibbs average $\langle \cdot \rangle_{\pi_k=r}'$, we may replace them by a uniform random variable $i\in \{1,\ldots,N\}$ and an exponential random variable $e\sim \Exp(1)$ independent of all other sources of randomness as in the statement of the theorem. To emphasize this change, we also replace $D_{1k}^1$ and $\Theta_{1k}^\ell$ by $d_{ik}^1=\sigma_i^1e/(1+\lambda_k\sigma_i^*)^2$ and $\theta_{ik}^\ell=\log(1+\lambda_k\sigma_i^\ell)-\lambda_k\sigma_i^\ell e/(1+\lambda_k\sigma_i^*)$, respectively. Notice that this matches the definitions in \eqref{eqn: SBMA Franz de Sanctis variables}. In this new notation \eqref{eqn: SBMA Franz de Sanctis distributional identity} reads
$$\E\big\langle f_nD_{1k}^1\big\rangle_{\pi_k=r}=\E\frac{\langle f_nd_{ik}^1\exp\big(\sum_{\ell\leq n}\theta_{ik}^\ell\big)\rangle'_{\pi_k=r}}{(\langle \exp(\theta_{ik})\rangle'_{\pi_k=r})^n}.$$
Substituting this into \eqref{eqn: SBMA Franz de Sanctis key} and making the change of variables $m=r-1$ reveals that
$$\E\big\langle f_n\widetilde{\L}_k(\sigma^1)\big\rangle=\sum_{m\geq 0}\frac{s_N^m}{m!}\exp(-s_N)\E\frac{\langle f_nd_{ik}^1\exp\big(\sum_{\ell\leq n}\theta_{ik}^\ell\big)\rangle'_{\pi_k=m+1}}{(\langle \exp(\theta_{ik})\rangle'_{\pi_k=m+1})^n}.$$
Notice that whenever $\pi_k=m+1$, the Hamiltonian defining the Gibbs average $\smash{\langle \cdot\rangle'_{\pi_k=m+1}}$ is given by $\smash{H_N'(\sigma)+\HH_k^{m}(\sigma)}$. This Hamiltonian has the same distribution as the Hamiltonian in \eqref{eqn: SBMA perturbed Hamiltonian} defining the original Gibbs average $\langle \cdot\rangle_{\pi_k=m}$. It follows that
$$\E\big\langle f_n\widetilde{\L}_k(\sigma^1)\big\rangle=\sum_{m\geq 0}\frac{s_N^m}{m!}\exp(-s_N)\E\frac{\langle f_nd_{ik}^1\exp\big(\sum_{\ell\leq n}\theta_{ik}^\ell\big)\rangle_{\pi_k=m}}{(\langle \exp(\theta_{ik})\rangle_{\pi_k=m})^n}=\E\frac{\langle f_nd_{ik}^1\exp\big(\sum_{\ell\leq n}\theta_{ik}^\ell\big)\rangle}{\langle \exp(\theta_{ik})\rangle^n}.$$
This is the first equality in \eqref{eqn: SBMA Franz de Sanctis goal}. The second equality in \eqref{eqn: SBMA Franz de Sanctis goal} is obtained by taking $n=1$ and $f_1=1$ in the first equality. This completes the proof. 
\end{proof}

Applying this result along a sequence of perturbation parameters $\smash{(\lambda^N)}$ satisfying \eqref{eqn: SBMA limits to get FDS} reveals that
\begin{equation}\label{eqn: SBMA overlap and Franz de Sanctis in limit 0}
\lim_{N\to \infty}\Big\lvert \E \frac{\langle f_nd_{ik}^1\exp\big(\sum_{\ell\leq n}\theta_{ik}^\ell\big)\rangle_N}{\langle \exp(\theta_{ik})\rangle_N^n}-\E\langle f_n\rangle_N \E\frac{\langle d_{ik}\exp(\theta_{ik})\rangle_N}{\langle \exp(\theta_{ik})\rangle_N}\Big\rvert=0
\end{equation}
for any $1\leq k\leq K_+$ and any function $f_n$ of finitely many spins on $n$ replicas and of the signal $\sigma^*$ with $\norm{f_n}_{L^\infty}\leq 1$. Observing that the denominators $\langle \exp(\theta_{ik}) \rangle_N$ do not depend on the signal $\sigma^*$, it is possible to use the Nishimori identity \eqref{eqn: SBM Nishimori identity} to replace all occurrences of the signal $\sigma^*$ in \eqref{eqn: SBMA overlap and Franz de Sanctis in limit 0} by another replica. For convenience of notation we will denote this new replica by $\sigma^\diamond$ to distinguish it from the signal $\sigma^*$ and at the same time not occupy any specific index. The equations in \eqref{eqn: SBMA overlap and Franz de Sanctis in limit 0} now read
\begin{equation}\label{eqn: SBMA overlap and Franz de Sanctis in limit 1}
\lim_{N\to \infty}\Big\lvert \E\E_\diamond \frac{\langle f_nd_{ik}^1  \exp\big(\sum_{\ell\leq n}\theta_{ik}^\ell\big)\rangle_N}{\langle \exp(\theta_{ik}) \rangle_N^n}-\E\E_\diamond\langle f_n\rangle_N \E\E_\diamond\frac{\langle d_{ik} \exp(\theta_{ik})\rangle_N}{\langle \exp(\theta_{ik})\rangle_N}\Big\rvert=0,
\end{equation}
where $\E_\diamond$ denotes the Gibbs average with respect to the replica $\sigma^\diamond$ only, the bracket $\langle \cdot\rangle_N$ denotes the Gibbs average with respect to all other standard replicas, the function $f_n$ depends on finitely many spins on the $n$ standard replicas and on $\sigma^\diamond$, and, with some abuse of notation,
\begin{equation}
y_{ik}=\frac{e}{1+\lambda_k\sigma_i^\diamond}, \quad \theta_{ik}^\ell=\log(1+\lambda_k\sigma_i^\ell)-\lambda_ky_{ik}\sigma_i^\ell,\quad \text{and} \quad d_{ik}^\ell=\frac{y_{ik}\sigma_i^\ell}{1+\lambda_k\sigma_i^\diamond}
\end{equation}
with $\lambda_k=\lambda_k^N$ for $1\leq k\leq K_+$. We now simplify \eqref{eqn: SBMA overlap and Franz de Sanctis in limit 1} for functions $f_n$ that do not depend on the spin coordinate indexed by $1$.  Introduce the collection of functions
\begin{align}
\F_n=\big\{\text{f}&\text{unctions } f_n \text{ of }\text{ finitely many spins } \sigma_i^\ell, \sigma_i^\diamond  \text{ with } 2\leq i\leq N \notag\\
& \text{ of the } n \text{ standard replicas }  (\sigma^\ell)_{\ell\leq n} \text{ and the special replica }\sigma^\diamond \text{ with } \norm{f_n}_{L^\infty}\leq 1\big\}
\end{align}
and the quantities
\begin{equation}\label{eqn: SBMA Franz de Sanctis asymptotic variables}
y_{k}=\frac{e}{1+\lambda_k\sigma_1^\diamond}, \quad \theta_{k}^\ell=\log(1+\lambda_k\sigma_1^\ell)-\lambda_ky_{k}\sigma_1^\ell,\quad \text{and} \quad d_{k}^\ell=\frac{y_{k}\sigma_1^\ell}{1+\lambda_k\sigma_1^\diamond}.
\end{equation}
For functions $f_n\in \F_n$, the symmetry between sites \eqref{eqn: SBMA symmetry between sites} and the fact that $i\in \{2,\ldots,N\}$ with overwhelming probability in the limit, allow us to replace the uniform random index $i\in \{1,\ldots,N\}$ by the index $1$. The Franz de-Sanctis identities together with assumption \eqref{eqn: SBMA limits to get FDS} therefore have the following important implication.

\begin{corollary}[Asymptotic Franz-de Sanctis identities in inference]
If \eqref{eqn: SBMA limits to get FDS} holds, then for every $1\leq k\leq K_+$ and all functions $f_n\in \F_n$,
\begin{equation}\label{eqn: SBMA overlap and Franz de Sanctis in limit 2}
\lim_{N\to \infty}\Big\lvert \E\E_\diamond \frac{\langle f_nd_{k}^1 \exp\big(\sum_{\ell\leq n}\theta_{k}^\ell\big)\rangle_N}{\langle  \exp(\theta_{k}) \rangle_N^n}-\E\E_\diamond\langle f_n\rangle_N \E\E_\diamond\frac{\langle d_{k} \exp(\theta_{k}) \rangle_N}{\langle  \exp(\theta_{k}) \rangle_N}\Big\rvert=0.
\end{equation}
\end{corollary}

\subsection{Finitary multi-overlap concentration}

The finitary version of the multi-overlap concentration result in \cite{BarMOC} will be uniform over an appropriate class of random probability measures which we now describe. For each integer $N\geq 1$ consider the set of random probability measures on $\Sigma_N$ thought of as a subset of $\{-1,0,1\}^\N$,
\begin{equation}
\G_N=\Big\{G\mid G \text{ is a random probability measure on } \Sigma_N\times \{0\}^\N\Big\},
\end{equation}
and say that a measure $G\in \G_N$ satisfies symmetry between sites if, for any sequence of i.i.d.\@ replicas $(\sigma^\ell)_{\ell\geq 1}$ sampled from $G$,
\begin{equation}\label{eqn: SBMA generic symmetry between sites} 
\big(\sigma_i^\ell)_{i,\ell\geq 1}\stackrel{d}{=}\big(\sigma_{\rho_1(i)}^{\rho_2(\ell)}\big)_{i,\ell\geq 1}
\end{equation}
for any permutation $\rho_1$ on the finite set $\{1,\ldots,N\}$ and any permutation $\rho_2$ of finitely many indices. Denote by $\G_N^s$ the subset of $\G_N$ which satisfies symmetry between sites. Notice that each Gibbs measure $G_N$ defined by \eqref{eqn: SBMA generic Gibbs measure} can be thought of as an element of $\G_N$ by setting $\sigma_i=0$ when $i>N$ for any replica $\sigma\in \Sigma_N$ sampled from $G_N$. In this way, the symmetry between sites in \eqref{eqn: SBMA generic symmetry between sites} and \eqref{eqn: SBMA symmetry between sites} coincide, so in fact $G_N\in \G_N^s$. This identification also suggests that the appropriate notion of the multi-overlap \eqref{eqn: SBMA generic multi-overlaps} for a random probability measure $G\in \G_N$ should be
\begin{equation}\label{eqn: SBMA generic multi-overlaps 2}
R_{\ell_1,\ldots,\ell_n}=\frac{1}{N}\sum_{i\leq N}\sigma_i^{\ell_1}\cdots \sigma_i^{\ell_n},
\end{equation}
where $(\sigma^\ell)$ denotes a sequence of i.i.d.\@ replicas sampled from the Gibbs measure $G$. Denoting by $\langle \cdot \rangle_G$ the average with respect to the random probability measure $G$ we may now sate the main result of this appendix.

\begin{proposition}\label{SBMA finitary MOC}
For every $\epsilon>0$ there exists $\delta>0$ such that the following holds. Let $N\geq \lfloor\delta^{-1}\rfloor$ and $G\in \G_N$ be a random probability measure such that for all $1\leq k\leq K_+=\lfloor \delta^{-1}\rfloor$ and any function $f_n\in \F_n$,
\begin{align}\label{eqn: SBMA finitray MOC assumptions}
\E\big\langle (R_1-&\E\langle R_1\rangle_G)^2\big\rangle_G \leq \delta, \qquad\qquad  \E\big\langle (R_{1,2}-\E\langle R_{1,2}\rangle_G )^2\big\rangle_G\leq \delta,\\
&\Big\lvert \E\E_\diamond \frac{\langle f_nd_{k}^1 \exp\big(\sum_{\ell\leq n}\theta_{k}^\ell\big)\rangle_G}{\langle \exp(\theta_{k})\rangle^n_G}-\E\E_\diamond\langle f_n\rangle_G \E\E_\diamond\frac{\langle d_{k} \exp(\theta_{k})\rangle_G}{\langle \exp(\theta_{k})\rangle_G}\Big\rvert\leq \delta.
\end{align}
Then for any $1\leq m\leq \lfloor \epsilon^{-1}\rfloor$, we have
\begin{equation}
\E\big\langle (R_{1,\ldots,m}-\E\langle R_{1,\ldots,m}\rangle_G)^2\big\rangle_G\leq \epsilon.
\end{equation}

\end{proposition}

The proof proceeds by contradiction and closely follows Section 3.5 and Section 3.7 of \cite{BarMOC}. Suppose there exists some $\epsilon>0$ such that no matter the choice of $\delta>0$, it is always possible to find some integer $N=N(\delta)\geq \lfloor \delta^{-1}\rfloor$ and some random probability measure $G=G(\delta)\in \G_N$ with
\begin{align}
\E\big\langle (R_1-&\E\langle R_1\rangle_G)^2\big\rangle_G \leq \delta, \qquad\qquad  \E\big\langle (R_{1,2}-\E\langle R_{1,2}\rangle_G )^2\big\rangle_G\leq \delta,\label{eqn: SBMA multi-overlap concentration assumption 1}\\
&\Big\lvert \E\E_\diamond \frac{\langle f_nd_{k}^1 \exp\big(\sum_{\ell\leq n}\theta_{k}^\ell\big)\rangle_G}{\langle \exp(\theta_{k}) \rangle^n_G}-\E\E_\diamond\langle f_n\rangle_G \E\E_\diamond\frac{\langle d_{k} \exp(\theta_{k}) \rangle_G}{\langle \exp(\theta_{k})\rangle_G}\Big\rvert\leq \delta\label{eqn: SBMA multi-overlap concentration assumption 2}
\end{align}
for any $1\leq k\leq K_+=\lfloor \delta^{-1}\rfloor$ and any $f_n\in \F_n$ for which there exists some $1\leq m=m(\delta) \leq \lfloor \epsilon ^{-1}\rfloor$ with
\begin{equation}\label{eqn: SBMA multi-overlap concentration absurd}
\E\big\langle (R_{1,\ldots,m}-\E\langle R_{1,\ldots,m}\rangle_G)^2\big\rangle_G > \epsilon.
\end{equation}
Applying the Prokhorov theorem on the compact metric space $\{-1,0,+1\}^{\N^2}$ and noticing that there are only finitely many choices for $m=m(\delta)$, it is possible to find a subsequence with $\delta\to 0$ along which the distribution of the array $\smash{\big(\sigma_i^\ell\big)_{i,\ell\geq 1}}$ under the averaged Gibbs measure $\smash{\E\langle \cdot \rangle_{G(\delta)}}$ converges in the sense of finite dimensional distributions and along which \eqref{eqn: SBMA multi-overlap concentration assumption 1}, \eqref{eqn: SBMA multi-overlap concentration assumption 2} and \eqref{eqn: SBMA multi-overlap concentration absurd} hold for every $k\geq 1$ and a fixed $1\leq m\leq \lfloor \epsilon^{-1}\rfloor$. Since $\smash{N(\delta)\to \infty}$ and $\smash{G(\delta)\in \G_N}$, in the limit, the distribution of spins will be a measure on $\{-1,+1\}^{\N^2}$ which will inherit the symmetry between sites \eqref{eqn: SBMA generic symmetry between sites}. By the Aldous-Hoover representation (see Theorem 1.4 in \cite{PanSKB}), this symmetry implies the existence of some function $\sigma:[0,1]^4\to \{-1,+1\}$ with
\begin{equation}\label{eqn: SBMA Aldous-Hoover}
\big(\sigma_i^\ell)_{i,\ell\geq 1}\stackrel{d}{=}\big(\sigma(w,u_\ell,v_i,x_{i,\ell})\big)_{i,\ell\geq 1},
\end{equation}
where $w$, $(u_\ell)$, $(v_i)$ and $(x_{i,\ell})$ are i.i.d.\@ uniform random variables on $[0,1]$. Since $\sigma$ takes values in $\{-1,+1\}$, the distribution of the array $(\sigma_i^\ell)$ is encoded by the function
\begin{equation}
\bar{\sigma}(w,u,v)=\E_{x_{i,\ell}}\sigma(w,u,v,x_{i,\ell})=\int_0^1 \sigma(w,u,v,x)\ud x.
\end{equation}
Indeed, the last coordinate $x_{i,\ell}$ is a dummy variable that
corresponds to flipping a biased coin to generate a Bernoulli random variable with expected value $\bar{\sigma}(w,u,v)$. To clarify this further, let $\ud u$ and $\ud v$ denote Lebesgue measure on $[0,1]$, and define the random probability measure
\begin{equation}
G=G_w=\ud u\circ \big(u\mapsto \bar{\sigma}(w,u,\cdot)\big)^{-1}
\end{equation}
on the space of functions of $v\in [0,1]$,
\begin{equation}\label{eqn: SBMA Hilbert space}
H=L^2\big([0,1],\ud v\big)\cap \big\{\norm{\bar{\sigma}}_{L^\infty}\leq 1\big\},
\end{equation}
equipped with the topology of $L^2([0,1],\ud v)$. As described in Section 2 of \cite{PanRSB}, the whole process of generating spins can be broken into the following steps:
\begin{enumerate}[label=(\roman*)]
    \item generate the asymptotic Gibbs measure $G=G_w$ using the uniform random variable $w$;
    \item consider an i.i.d.\@ sequence $\bar{\sigma}^\ell=\bar{\sigma}(w,u_\ell,\cdot)$ of replicas from $G$, which are functions in $H$;
    \item plug in i.i.d.\@ uniform random variables $(v_i)_{i\geq 1}$ to obtain the array $\bar{\sigma}^\ell(v_i)=\bar{\sigma}(w,u_\ell,v_i)$;
    \item finally, use the random variables $(x_{i,\ell})$ to generate $(\sigma_i^\ell)$ by flipping a coin with expected value $\bar{\sigma}^\ell(v_i)$,
    \begin{equation}
    \sigma_i^\ell=2\1\Big\{x_{i,\ell}\leq \frac{1+\bar{\sigma}^\ell(v_i)}{2}\Big\}-1.
    \end{equation}
\end{enumerate}
This suggests that the asymptotic Gibbs average should be the average with respect to the random variables $(u_\ell)$ and $(x_{i,\ell})$ that depend on the replica indices, which we will denote by
\begin{equation}\label{eqn: SBMA asymptotic Gibbs measure}
\langle\cdot \rangle=\E_{(u_\ell),(x_{i,\ell})}.
\end{equation}
We can also expect the asymptotic multi-overlap to be given by 
\begin{equation}\label{eqn: SBMA asymptotic multi-overlaps}
R^\infty_{\ell_1,\ldots,\ell_n}(w,(u_{\ell_j})_{j\leq n})=\E_v\prod_{j\leq n}\bar{\sigma}(w,u_{\ell_j},v)=\int_0^1 \prod_{j\leq n}\bar{\sigma}(w,u_{\ell_j},v)\ud v.
\end{equation}
This intuition is confirmed by the two following results adapted from Section 3.5 and the Appendix in \cite{BarMOC}.

\begin{lemma}
For any finite set of $n$ replicas and every collection $\{\CC_\ell\}_{\ell\leq n}$ of finite indices,
\begin{equation}
\lim_{\delta \to 0}\E\prod_{\ell\leq n}\Big\langle \prod_{i\in \CC_\ell}\sigma_i^\ell\Big\rangle_{G(\delta)}=\E_{w,(v_i)}\prod_{\ell\leq n}\Big\langle \prod_{i\in \CC_\ell} \sigma_i^\ell\Big\rangle.
\end{equation}
\end{lemma}

\begin{proof}
Let $\CC=\{(i,\ell)\mid \ell \leq n \text{ and } i\in \CC_\ell\}$. By definition of weak convergence in the sense of finite-dimensional marginal distributions,
$$\lim_{\delta \to 0}\E\prod_{\ell\leq n}\Big\langle \prod_{i\in \CC_\ell}\sigma_i^\ell\Big\rangle_{G(\delta)}=\lim_{\delta \to 0}\E\Big\langle\prod_{(i,\ell)\in \CC}\sigma_i^\ell\Big\rangle_{G(\delta)}=\E \prod_{(i,\ell)\in \CC}\sigma(w,u_\ell,v_i,x_{i,\ell}).$$
Recalling the notation \eqref{eqn: SBMA asymptotic Gibbs measure},
$$\E \prod_{(i,\ell)\in \CC}\sigma(w,u_\ell,v_i,x_{i,\ell})=\E_{w,(v_i)}\prod_{\ell\leq n}\E_{(u_\ell)}\prod_{i\in \CC_\ell}\E_{(x_{i,\ell})}\sigma(w,u_\ell,v_i,x_{i,\ell})=\E_{w,(v_i)}\prod_{\ell\leq n}\Big\langle \prod_{i\in \CC_\ell} \sigma_i^\ell\Big\rangle.$$
as required.
\end{proof}

\begin{lemma}
For any collection of sets $\{\L_i\}_{i\geq 1}$ only finitely many of which are non-empty,
\begin{equation}
\lim_{\delta\to 0}\E\Big\langle\prod_{i\geq 1} R_{\L_i}\Big\rangle_{G(\delta)}=\E_w\Big\langle \prod_{i\geq 1}R_{\L_i}^\infty\Big\rangle.
\end{equation}
\end{lemma}

\begin{proof}
Write $N\geq \lfloor \delta^{-1}\rfloor$ for the unique integer with $G(\delta)\in \G_{N}$ and suppose without loss of generality that the sets $\L_i$ for $i\leq j$ are non-empty while the sets $\L_i$ for $i>j$ are empty. From \eqref{eqn: SBMA generic multi-overlaps 2},
$$\E\Big\langle\prod_{i\geq 1} R_{\L_i}\Big\rangle_{G(\delta)}=\frac{1}{N^j}\sum_{i_1,\ldots,i_j}\E\Big\langle \prod_{\ell_1\in \L_1}\cdots \prod_{\ell_j\in \L_j}\sigma_{i_1}^{\ell_1}\cdots \sigma_{i_j}^{\ell_j}\Big\rangle_{G(\delta)}.$$
The number of terms in this sum for which at least two of the indices $i_1,\ldots,i_j$ are equal is of order $N^{j-1}$ and is therefore negligible in the limit. Moreover, the symmetry between sites \eqref{eqn: SBMA generic symmetry between sites} implies that whenever $i_1,\ldots,i_j$ are all distinct
$$\E\Big\langle \prod_{\ell_1\in \L_1}\cdots \prod_{\ell_j\in \L_j}\sigma_{i_1}^{\ell_1}\cdots \sigma_{i_j}^{\ell_j}\Big\rangle_{G(\delta)}=\E\Big\langle \prod_{\ell_1\in \L_1}\cdots \prod_{\ell_j\in \L_j}\sigma_{1}^{\ell_1}\cdots \sigma_{j}^{\ell_j}\Big\rangle_{G(\delta)}=\E\Big\langle \prod_{i\geq 1}\prod_{\ell\in \L_i}\sigma^\ell_i\Big\rangle_{G(\delta)}.$$
(This seems to fix a small typo in the second-to-last display of the Appendix in \cite{BarMOC}). Combining these two observations shows that
$$\lim_{\delta \to 0}\E\Big\langle\prod_{i\geq 1} R_{\L_i}\Big\rangle_{G(\delta)}=\E_{w,(u_\ell)}\prod_{i\geq 1}\E_{v_i}\prod_{\ell\in \L_i}\E_{x_{i,\ell}}\sigma(w,u_\ell,v_i,x_{i,\ell})=\E_{w,(u_\ell)}\prod_{i\geq 1}R_{\L_i}^\infty.$$
This completes the proof.
\end{proof}

In the notation of \eqref{eqn: SBMA asymptotic Gibbs measure} and \eqref{eqn: SBMA asymptotic multi-overlaps}, the asymptotic version of \eqref{eqn: SBMA multi-overlap concentration assumption 1} and \eqref{eqn: SBMA multi-overlap concentration assumption 2} therefore reads
\begin{align}
\E \langle (R_1&^\infty)^2\rangle=\big(\E \langle R_1^\infty\rangle\big)^2, \qquad \E \langle (R_{1,2}^\infty)^2\rangle=\big(\E \langle R_{1,2}^\infty\rangle\big)^2 \label{eqn: SBMA asymtotic overlap and Franz de Sanctis 1}\\
&\E\E_\diamond \frac{\langle f_nd_{k}^1 \exp\big(\sum_{\ell\leq n}\theta_{k}^\ell\big)\rangle}{\langle \exp(\theta_{k}) \rangle^n}=\E\E_\diamond\langle f_n\rangle \E\E_\diamond\frac{\langle d_{k} \exp(\theta_{k})\rangle}{\langle \exp(\theta_{k}) \rangle}\label{eqn: SBMA asymtotic overlap and Franz de Sanctis 2}
\end{align}
for any $k\geq 1$ and any function $f_n\in \F_n$, while the asymptotic version of \eqref{eqn: SBMA multi-overlap concentration absurd} becomes
\begin{equation}\label{eqn: SBMA multi-overlap concentration asymptotic absurd}
\E\big\langle (R_{1,\ldots,m}^\infty-\E\langle R_{1,\ldots,m}^\infty \rangle)^2\big\rangle > \epsilon
\end{equation}
for some $1\leq m\leq \lfloor \epsilon^{-1}\rfloor$. We now derive the two most important consequences of the identities in \eqref{eqn: SBMA asymtotic overlap and Franz de Sanctis 1} and \eqref{eqn: SBMA asymtotic overlap and Franz de Sanctis 2} that will allow us to establish multi-overlap concentration. On the one hand, the concentration of the overlap $\E\langle (R_{1,2}^\infty)^2\rangle=(\E\langle R_{1,2}^\infty\rangle)^2$ implies that the system lies in a ``thermal pure state'' and that the function $\bar{\sigma}(w,u,v)$ is therefore almost surely independent of $u$. The proof of this fact is taken from Theorem 3.1 in \cite{PanKSAT}.

\begin{lemma}
If $\E\langle (R_{1,2}^\infty)^2\rangle=(\E\langle R_{1,2}^\infty\rangle)^2$, then for almost all $u,v,w\in [0,1]$,
\begin{equation}
\bar{\sigma}(w,u,v)=\E_u\bar{\sigma}(w,u,v).
\end{equation}
\end{lemma}

\begin{proof}
Denote by $\cdot$ the inner product on the Hilbert space \eqref{eqn: SBMA Hilbert space},
$$\bar{\sigma}^1\cdot \bar{\sigma}^2=\E_{v}\bar{\sigma}^1(w,u_1,v)\bar{\sigma}^2(w,u_2,v)=R_{1,2}^\infty,$$
and observe that
\begin{align*}
0&=\E_{w,(u_\ell)} \big(R_{1,2}^\infty\big)^2-\E_{w,(u_\ell)} R_{1,2}^\infty R_{3,4}^\infty=\E_w\Var_{(u_\ell)}\bar{\sigma}^1\cdot \bar{\sigma}^2.
\end{align*}
It follows that for almost all $w\in [0,1]$, the inner product $\bar{\sigma}^1\cdot \bar{\sigma}^2$ of any two replicas sampled from the Gibbs measure $G_w$ is constant almost surely. In other words, the measure $G_w$ is concentrated on a single function which may depend on $w$. This completes the proof.
\end{proof}

The second identity in \eqref{eqn: SBMA asymtotic overlap and Franz de Sanctis 1} therefore implies that, instead of the equality in distribution \eqref{eqn: SBMA Aldous-Hoover}, we actually have
\begin{equation}\label{eqn: SBMA Aldous-Hoover pure state}
\big(\sigma_i^\ell\big)_{i,\ell\geq 1}\stackrel{d}{=}\big(\sigma(w,v_i,x_{i,\ell})\big)_{i,\ell\geq 1}
\end{equation}
for any function $\sigma:[0,1]^3\to \{-1,+1\}$ such that $\int_0^1 \sigma(w,v,x)\ud x=\bar{\sigma}(w,v)$. In particular, the Gibbs average \eqref{eqn: SBMA asymptotic Gibbs measure} simplifies to
\begin{equation}\label{eqn: SBMA asymptotic Gibbs measure pure state}
\langle\cdot \rangle=\E_{(x_{i,\ell})}
\end{equation}
while the multi-overlap \eqref{eqn: SBMA asymptotic multi-overlaps} becomes
\begin{equation}
R_{\ell_1,\ldots,\ell_n}^\infty(w)=\E_v\prod_{j\leq n}\bar{\sigma}(w,v)=\E_v\big(\bar{\sigma}(w,v)^n\big)=\int_0^1 \bar{\sigma}(w,v)^n\ud v.
\end{equation}
On the other hand, the asymptotic Franz-de Sanctis identities in \eqref{eqn: SBMA asymtotic overlap and Franz de Sanctis 2} imply the following decoupling property of the asymptotic Gibbs measure. This is lemma 3.5 in \cite{BarMOC}.

\begin{lemma}[A decoupling lemma]
Fix $\lambda\in \{\lambda_k\mid k\geq 1\}$. If $e_1,e_2$ are independent $\Exp(1)$ random variables and, for $j=1,2$,
\begin{equation}
y_j=\frac{e_j}{1+\lambda \sigma_j^\diamond},\quad \theta_j=\log(1+\lambda \sigma_j)-\lambda y_j\sigma_j,\quad\text{and}\quad d_j=\frac{y_j\sigma_j}{1+\lambda \sigma_j^\diamond},
\end{equation}
then
\begin{equation}\label{eqn: SBMA decoupling lemma}
\E\E_\diamond \frac{\langle d_1 \exp(\theta_1)d_2 \exp(\theta_2)\rangle}{\langle \exp(\theta_1) \exp(\theta_2)\rangle}=\E\E_\diamond \frac{\langle d_1\exp(\theta_1)\rangle}{\langle \exp(\theta_1)\rangle}\E\E_\diamond \frac{\langle d_2 \exp(\theta_2)\rangle}{\langle \exp(\theta_2)\rangle}.
\end{equation}
\end{lemma}

\begin{proof}
The idea is to combine a technical truncation argument with the Weierstrass approximation theorem and an application of \eqref{eqn: SBMA asymtotic overlap and Franz de Sanctis 2} to the function $f_{n+1}=d_2^1\exp \sum_{\ell\leq n+1}\theta_2^\ell$; we encourage the interested reader to consult \cite{BarMOC}.
\end{proof}

We are finally in a position to contradict \eqref{eqn: SBMA multi-overlap concentration asymptotic absurd} and prove \Cref{SBMA finitary MOC}. The calculations are very similar in spirit to those in \cite{PanRSBp, PanRSB}, and are taken from Theorem 2.2 in \cite{BarMOC}.

\begin{proof}[Proof of \Cref{SBMA finitary MOC}.]
We follow the proof of Theorem 2.2 in \cite{BarMOC}; we will not give full details, and instead encourage the interested reader to consult \cite{BarMOC}. Recall from \eqref{eqn: SBMA Aldous-Hoover pure state} that $\smash{\sigma_j=\sigma(w,v_j,x_j)}$ and $\smash{\sigma_j^\diamond=\sigma(w,v_j,x_j^\diamond)}$. Since all random variables indexed by $j=1,2$ are independent, if we denote by $\smash{\E_{|w}=\E_{(e_j),(v_j),x_j,x_j^\diamond}}$ the conditional expectation given $w$ and we introduce the random variable
$$Y(w)=\E_{|w}\frac{\langle d_1 \exp(\theta_1)\rangle}{\langle \exp(\theta_1)\rangle}=\E_{|w}\frac{y_1}{1+\lambda\sigma_1^\diamond}\frac{\langle \sigma_1 \exp(\theta_1)\rangle}{\langle \exp(\theta_1)\rangle},$$
which depends implicitly on $\lambda$ through $y_1$ and $\theta_1$, then \eqref{eqn: SBMA decoupling lemma} reads
\begin{align*}
\E \Big(\E_{|w}\frac{\langle d_1 \exp(\theta_1)\rangle}{\langle  \exp(\theta_1)\rangle}\Big)\Big(\E_{|w}\frac{\langle d_2 \exp(\theta_2)\rangle}{\langle \exp(\theta_2)\rangle}\Big)-\Big(\E \frac{\langle d_1 \exp(\theta_1)\rangle}{\langle  \exp(\theta_1)\rangle}\Big)\Big(\E\frac{\langle d_2 \exp(\theta_2)\rangle}{\langle \exp(\theta_2)\rangle}\Big)&=\E\Var_{|w}Y(w)\\
&=0.
\end{align*}
This means that $Y=\E Y$ almost surely. To exploit this fact, through a slight abuse of notation, write $\sigma$ for $\sigma_1$ and observe that conditionally on $\sigma_1^\diamond$,
$$Y(w)=\E_{|w}\int_0^\infty \langle \exp(-\lambda \sigma y)\rangle \frac{\langle \sigma (1+\lambda \sigma) \exp(-\lambda y\sigma)\rangle}{\langle (1+\lambda \sigma) \exp(-\lambda y \sigma)\rangle}y \exp(-y)\ud y.$$
Using the analyticity of both
$$g_w:\gamma\mapsto g_w(\gamma)=\E_{|w}\int_0^\infty \langle \exp(-\gamma \sigma y)\rangle\frac{\langle \sigma(1+\gamma \sigma) \exp(-\gamma y \sigma)\rangle}{\langle (1+\gamma \sigma) \exp(-\gamma y \sigma)\rangle}y \exp(-y)\ud y$$
for a fixed $w$ as well as its $w$-expectation $\E g_w(\gamma)$, it is possible to deduce that $Y(w)=\E Y$ for all $\lambda$ in a small neighbourhood of the origin. With this in mind, introduce the random variable
$$Z(w)=\E_{|w}\int_0^\infty \langle \sigma(1+\lambda \sigma) \exp(-\lambda \sigma)\rangle y \exp(-y)\ud y$$
which is deterministic by the first identity in \eqref{eqn: SBMA asymtotic overlap and Franz de Sanctis 1}. This implies that the random variable
$$X(w)=\frac{Z(w)-Y(w)}{\lambda}=\E_{|w}\int_0^\infty \langle \sigma \exp(-\lambda y\sigma)\rangle\frac{\langle \sigma (1+\lambda \sigma) \exp(-\lambda y\sigma) \rangle}{\langle (1+\lambda \sigma) \exp(-\lambda y\sigma)\rangle}y \exp(-y)\ud y$$
is deterministic for all $\lambda$ in a small neighbourhood of the origin. In particular, all its $\lambda$-derivatives are also independent of $w$. We will now deduce from this observation that all multi-overlaps concentrate. Given $n\geq 1$, applying $\frac{\partial^n}{\partial \lambda^n}$ to the denominator in the expression inside the integral defining $X$ and evaluating at $\lambda=0$ yields the term
$$n!R_{1,\ldots,n+2}\E e(e-1)^n,$$
where $e$ is an $\Exp(1)$ random variable. Since $\E e(e-1)^n=\E(e-1)^{n+1}+\E(e-1)^n>0$ for all $n\geq 1$, the term obtained by applying all derivatives to the denominator in the expression inside the integral defining $X$ produces the multi-overlap $\smash{R_{1,\ldots,n+2}^\infty}$. If along the way we apply a derivative of $\lambda$ to any factor other than the denominator, this will not create a new replica, so all those terms will produce a linear combination of multi-overlaps on strictly less than $n+2$ replicas which by induction we assume to be independent of $w$. This establishes the concentration of all multi-overlaps and contradicts \eqref{eqn: SBMA multi-overlap concentration asymptotic absurd} thus completing the proof.
\end{proof}
\end{appendix}

\bibliographystyle{abbrv}
\bibliography{sparse_prob}

\end{document}